\documentclass[a4paper,12pt,twoside]{article}
\usepackage[english]{babel}
\usepackage[dvips]{graphics}
\usepackage{epsf}
\usepackage{graphicx}
\usepackage{epsfig}
\usepackage{amsmath}
\usepackage[psamsfonts]{amssymb}
\usepackage[psamsfonts]{eucal}
\usepackage{amsthm}
\usepackage[utf8]{inputenc}
\usepackage{textcomp}
\usepackage{mathptmx}
\usepackage[scaled]{helvet}
\usepackage[all]{xy}
\usepackage[colorlinks=true]{hyperref}
\hypersetup{urlcolor=blue, citecolor=red}
\usepackage{url}
\swapnumbers
\theoremstyle{plain}
\newtheorem{theo}[subsubsection]{Theorem}
\newtheorem{lemma}[subsubsection]{Lemma}
\newtheorem{coro}[subsubsection]{Corollary}
\newtheorem{prop}[subsubsection]{Proposition}
\theoremstyle{definition}
\newtheorem{defi}[subsubsection]{Definition}
\newtheorem{defis}[subsubsection]{Definitions}
\newtheorem{rmk}[subsubsection]{Remark}
\newtheorem{rmks}[subsubsection]{Remarks}
\newtheorem{ex}[subsubsection]{Example}
\newtheorem{exs}[subsubsection]{Examples}  
\newcommand{\field}[1]{\mathbb{#1}}

\newcommand{\NN}{\field{N}}

\newcommand{\RR}{\field{R}}
\newcommand{\ZZ}{\field{Z}}
\def\id{\mathop{\rm id}\nolimits}

\def\Ad{\mathop{\rm Ad}\nolimits}
\def\ad{\mathop{\rm ad}\nolimits}

\def\ch{\mathop{\rm ch}\nolimits}
\def\sh{\mathop{\rm sh}\nolimits}

\def\vect#1{\overrightarrow{#1}}

\def\SO{\mathop{\rm SO}\nolimits}
\def\so{\mathop{\mathfrak so}\nolimits}

\def\d{\mathrm{d}}
\def\mathi{\mathrm{i}}

\title{From Tools in Symplectic and Poisson Geometry\\ to Souriau's theories\\
 of Statistical Mechanics and Thermodynamics}
\author{Charles-Michel Marle\\
  Universit{\'e} Pierre et Marie Curie\\
  Paris,  France}

\begin{document}
\maketitle
\setcounter{tocdepth}{1} 
\tableofcontents

\noindent
{\bf Abstract\hfill}

\noindent
I present in this paper some tools in Symplectic and Poisson
Geometry in view of their applications in Geometric Mechanics and
Mathematical Physics. After a short discussion of the Lagrangian an Hamiltonian
formalisms, including the use of symmetry groups, 
and a presentation of the Tulczyjew's isomorphisms (which
explain some aspects of the relations between these formalisms),
I explain the concept of manifold of motions of a mechanical system and its use,
due to J.-M. Souriau, in Statistical Mechanics and Thermodynamics. The generalization of
the notion of thermodynamic equilibrium in which the one-dimensional group 
of time translations is replaced by a multi-dimensional, maybe non-commutative Lie group,
is discussed and examples of applications in Physics are given.
\newpage
\hbox to \textwidth{\hfill\emph{In memory of Jean-Marie Souriau (1922--2012)}}

\section{Introduction}\label{Introduction}

\subsection{Contents of the paper, sources  and further reading}
This paper presents tools in Symplectic and Poisson Geometry in view of their application
in Geometric Mechanics and Mathematical 
Physics. The Lagrangian formalism and symmetries
of Lagrangian systems are discussed in
Sections 2 and 3, the Hamiltonian formalism and symmetries of Hamiltonian systems in Sections
4 and 5. Section 6 introduces the concepts of Gibbs state and of thermodynamic equilibrium of a
mechanical system, and presents several examples. For a monoatomic classical ideal gas, eventually 
in a gravity field, or a monoatomic relativistic gas the Maxwell-Boltzmann and Maxwell-Jüttner 
probability distributions are derived. The Dulong and Petit law which governs the specific heat of
solids is obtained. Finally Section 7 presents the generalization of the concept of Gibbs state, due to
Jean-Marie Souriau, in which the group of time translations is replaced by a (multi-dimensional 
and eventually non-Abelian)
Lie group.
\par\smallskip

Several books \cite{AbrahamMarsden, Arnold, cannasdasilva, GuilleminSternberg, holm3, 
Iglesias2000, LaurentPichereauVanhaecke, LibermannMarle1987, OrtegaRatiu2004, Vaisman94} 
discuss, much more fully than in the present paper, the contents
of Sections 2 to 5. The interested reader is referred to these books for detailed proofs of 
results whose proofs
are only briefly sketched here. The recent paper \cite{Marle2014} contains detailed proofs of most
results presented here in Sections 4 and 5.	
\par\smallskip

The main sources used for Sections 6 and 7 are the book and papers by Jean-Marie Souriau
\cite{Souriau1969, Souriau1966, Souriau1974, Souriau1975, Souriau1984} and the beautiful small book 
by G.~W.~Mackey \cite{Mackey1963}.
\par\smallskip

The Euler-Poincaré equation, which is presented with Lagrangian symmetries at the end of Section 3, 
is not really related to symmetries of a Lagrangian system, 
since the Lie algebra which acts on the configuration space of the system is not a Lie algebra 
of symmetries of the Lagrangian. Moreover in its intrinsic form that equation uses the concept 
of Hamiltonian momentum map presented later, in Section 5. Since the Euler-Poincaré equation is not used
in the following sections, the reader can skip the corresponding subsection at his or her first reading.
\par\smallskip

\subsection{Notations}
The notations used are more or less those generally used now in Differential Geometry. The tangent and cotangent
bundles to a smooth manifold $M$ are denoted by $TM$ and $T^*M$, respectively, and their canonical projections 
by $\tau_M:TM\to M$ and $\pi_M:T^*M\to M$. The vector spaces of $k$-multivectors and $k$-forms on $M$
are denoted by $A^k(M)$ and $\Omega^k(M)$, respectively, with $k\in\ZZ$ and, of course,
$A^k(M)=\{0\}$ and $\Omega^k(M)=\{0\}$ if $k<0$ and if $k>\dim M$, $k$-multivectors and $k$-forms being
skew-symmetric. The exterior algebras of multivectors and forms of all degrees are denoted by $A(M)=\oplus_k A^k(M)$
and $\Omega(M)=\oplus_k\Omega^k(M)$, respectively. The exterior differentiation operator
of differential forms on a smooth manifold $M$ is denoted by $\d:\Omega(M)\to\Omega(M)$.
The interior product of a differential form $\eta\in\Omega(M)$ by a vector field
$X\in A^1(M)$ is denoted by $\mathi(X)\eta$. 
\par\smallskip

Let $f:M\to N$ be a smooth map defined on a smooth manifold $M$, with values in another smooth manifold $N$. The pull-back of a form $\eta\in\Omega(N)$
by a smooth map $f:M\to N$ is denoted by $f^*\eta\in\Omega(M)$.
\par\smallskip

A smooth, time-dependent vector field on the smooth manifold $M$ is a smooth map
$X:\RR\times M\to TM$ such that, for each $t\in \RR$ and $x\in M$, $X(t,x)\in T_xM$, the vector space tangent to $M$ at $X$. When, for any $x\in M$, $X(t,x)$ does not depend on
$t\in\RR$, $X$ is a smooth vector field in the usual sense, \emph{i.e.}, an element
in $A^1(M)$. Of course a time-dependent vector field can be defined on an open subset 
of $\RR\times M$ instead than on 
the whole $\RR\times TM$. It defines a differential equation
 $$\frac{\d\varphi(t)}{\d t}=X\bigl(t,\varphi(t)\bigr)\,,\eqno{(*)}
 $$
said to be \emph{associated} to $X$. The (full) \emph{flow} of $X$ is the map
$\Psi^X$, defined on an open subset of $\RR\times\RR\times M$, taking its values in $M$, 
such that for each $t_0\in\RR$ and $x_0\in M$
the parametrized curve $t\mapsto\Psi^X(t,t_0,x_0)$ is the maximal integral curve of
$(*)$ satisfying $\Psi(t_0,t_0,x_0)=x_0$. When $t_0$ and $t\in \RR$ are fixed, the map 
$x_0\mapsto\Psi^X(t,t_0,x_0)$ is a
diffeomorphism, defined on an open subset of $M$ (which may be empty) and taking its values in another open subset of $M$, denoted by $\Psi^X_{(t,\,t_0)}$. When $X$ is in fact a vector field in the usual sense (not dependent on time), $\Psi^X_{(t,\,t_0)}$ only depends on
$t-t_0$. Instead of the full flow of $X$ we can use its \emph{reduced flow} 
$\Phi^X$, defined on an open subset of $\RR\times M$ and taking its values in $M$, related to the full flow $\Psi^X$ by
 $$\Phi^X(t,x_0)=\Psi^X(t,0,x_0)\,,\quad \Psi^X(t,t_0,x_0)=\Phi^X(t-t_0,x_0)\,.
 $$
For each $t\in\RR$, the map $x_0\mapsto\Phi^X(t,x_0)=\Psi^X(t,0,x_0)$ is a diffeomorphism,
denoted by $\Phi^X_t$, defined on an open subset of $M$ (which may be empty) onto another open subset of $M$.
\par\smallskip
  
When $f:M\to N$ is a smooth map defined on a smooth manifold $M$, with values in another smooth manifold $N$, there exists a smooth map $Tf:TM\to TN$ called the 
\emph{prolongation of $f$ to vectors}, which for each fixed $x\in M$ linearly maps
$T_xM$ into $T_{f(x)}N$. When $f$ is a diffeomorphism of $M$ onto $N$,
$Tf$ is an isomorphism of $TM$ onto $TN$. That property allows us to define the
\emph{canonical lifts} of a vector field $X$ in $A^1(M)$ to the tangent bundle $TM$
and to the cotangent bundle $T^*M$. Indeed, for each $t\in\RR$, $\Phi^X_t$ is a diffeomorphism of an open subset of $M$ onto another open subset of $M$. Therefore
$T\Phi^X_t$ is a diffeomorphism of an open subset of $TM$ onto another open 
subset of $TM$. It turns out that when $t$ takes all possible values in $\RR$ the set of all diffeomorphisms $T\Phi^X_t$ is the reduced flow of a vector field $\overline X$ on 
$TM$, which is the \emph{canonical lift} of $X$ to the tangent bundle $TM$. 
\par\smallskip

Similarly, the transpose $(T\Phi^X_{-t})^T$ of $T\Phi^X_{-t}$ is a diffeomorphism of an open subset of the cotangent bundle $T^M$ onto another open subset of $T^*N$, and when
$t$ takes all possible values in $\RR$ the set of all diffeomorphisms $(T\Phi^X_{-t})^T$ 
is the reduced flow of a vector field $\widehat X$ on 
$T^*M$, which is the \emph{canonical lift} of $X$ to the cotangent bundle $T^*M$.
\par\smallskip

The canonical lifts of a vector field to the tangent and cotangent bundles are used
in Sections~3 and~5. They can be defined too for time-dependent vector fields.        

\section{The Lagrangian formalism}\label{LagrangianFormalism}

\subsection{The configuration space and the space of kinematic states} 
\label{ConfigurationSpace}
The principles of Mechanics were stated by the great English mathematician 
\emph{Isaac \hbox{Newton}} (1642--1727)
in his book \emph{Philosophia Naturalis Principia Mathematica} published in 1687 
\cite{Newton1687}.
On this basis, a little more than a century later, \emph{Joseph Louis Lagrange} (1736--1813)
in his book \emph{M\'ecanique analytique} \cite{Lagrange5} derived the equations (today known as
the \emph{Euler-Lagrange equations}) which govern the motion
of a mechanical system made of any number of material points or rigid 
material bodies interacting between them by  very general forces, 
and eventually submitted to external forces. 
\par\smallskip

In modern mathematical language, these equations are written on the \emph{configuration space} and 
on the \emph{space of kinematic states} of the considered mechanical system. The \emph{configuration space} is
a smooth $n$-dimensional manifold $N$ whose elements are all the possible configurations of the system
(a configuration being the position in space of all parts of the system). The \emph{space of kinematic states} is the tangent bundle $TN$ to the configuration space, which is $2n$-dimensional. Each element of the space of kinematic states is a vector tangent to the configuration space at one of its elements, \emph{i.e.} at a configuration of the mechanical system, which describes the velocity at which this configuration changes with time. In local coordinates a configuration of the system
is determined by the $n$ coordinates $x^1,\ldots,x^n$ of a point in $N$, and a kinematic state by the $2n$
coordinates $x^1,\ldots,x^n,v^1,\ldots v^n$ of a vector tangent to $N$ at some element in $N$.

\subsection{The Euler-Lagrange equations}

When the mechanical system is \emph{conservative}, the Euler-Lagrange equations involve a single
real valued function $L$ called the \emph{Lagrangian} of the system,  defined 
on the product of the real line $\RR$ (spanned by the variable $t$ representing the time) with
the manifold $TN$ of \emph{kinematic states} of the system. In local coordinates, the Lagrangian $L$ is expressed as a function of the
$2n+1$ variables, $t,x^1,\ldots,x^n,v^1,\ldots,v^n$ and the Euler-Lagrange equations have the remarkably simple form

 $$\frac{\d}{\d t}\left(\frac{\partial L}{\partial v^i}\bigl(t,x(t),v(t)\bigr)\right)
                       -\frac{\partial L}{\partial x^i}\bigl(t,x(t),v(t)\bigr)=0\,,\quad 1\leq i\leq n\,,
 $$  
where $x(t)$ stands for $x^1(t),\ldots,x^n(t)$ and $v(t)$ for $v^1(t),\ldots,v^n(t)$ with, of course,
 $$v^i(t)=\frac{\d x^i(t)}{\d t}\,,\quad 1\leq i\leq n\,.
 $$

\subsection{Hamilton's principle of stationary action}

The great Irish mathematician \emph{William Rowan Hamilton} (1805--1865) observed 
\cite{hamilton8, hamilton9}
that the Euler-Lagrange equations can be obtained by applying the standard techniques of
\emph{Calculus of Variations}, due to \emph{Leonhard Euler} (1707--1783) and 
\emph{Joseph Louis Lagrange},
to the \emph{action integral}\footnote{Lagrange observed that fact before Hamilton, 
but in the last edition of his book he chose to derive the Euler-Lagrange equations 
by application of the \emph{principle of virtual works}, using a very clever evaluation 
of the virtual work of inertial forces for a smooth infinitesimal variation of the motion.}
 $$I_L(\gamma)=\int_{t_0}^{t_1}L\bigl(t,x(t),v(t)\bigr)\,\d t\,,\quad\hbox{with}\ v(t)=\frac{\d x(t)}{\d t}\,, 
 $$
where $\gamma:[t_0,t_1]\to N$ is a smooth curve in $N$ parametrized by the time $t$. These equations express the fact that the action integral $I_L(\gamma)$ is \emph{stationary} with respect to any smooth 
infinitesimal variation of $\gamma$ with fixed end-points $\bigl(t_0,\gamma(t_0)\bigr)$ and
$\bigl(t_1,\gamma(t_1)\bigr)$. This fact is today called \emph{Hamilton's principle of stationary action}. The reader interested in Calculus of Variations and its applications in Mechanics and Physics is referred to the books \cite{berest, bourguignon, lanczos}.

\subsection{The Euler-Cartan theorem}\label{EulerCartan}

The \emph{Lagrangian formalism} is the use of Hamilton's principle of stationary action for 
the derivation of the equations of motion of a system. It is widely used in Mathematical Physics, 
often with more general Lagrangians involving more than one independent variable and higher order 
partial derivatives of dependent variables. For simplicity I will consider here only  the Lagrangians
of (maybe time-dependent) conservative mechanical systems.
\par\smallskip

An intrinsic geometric expression of the Euler-Lagrange equations, 
wich does not use local coordinates, was obtained by the great French mathematician 
\emph{\'Elie Cartan}
(1869--1951). Let us introduce the concepts used by the statement of this theorem.

\begin{defis}\label{DefisLegendreMapEtAl} 
Let $N$ be the configuration space of a mechanical system and let its tangent bundle
$TN$ be the space of kinematic states of that system. We assume that the evolution with time of 
the state of the system is governed by the Euler-Lagrange equations for a smooth, 
maybe time-dependent Lagrangian $L:\RR\times TN\to \RR$.
\par\smallskip\noindent

{\rm 1.\quad} The cotangent bundle $T^*N$ is called the
\emph{phase space} of the system.  
\par\smallskip\noindent

{\rm 2.\quad} The map ${\mathcal L}_L:\RR\times TN\to T^*N$
 $${\mathcal L}_L(t,v)=\d_{\rm vert} L(t,v)\,,\quad t\in\RR\,,\ v\in TN\,,
 $$
where $\d_{\rm vert} L(t,v)$ is the \emph{vertical differential} of $L$ at $(t,v)$, \emph{i.e} the differential at $v$
of the the map $w\mapsto L(t,w)$, with $w\in\tau_N^{-1}\bigl(\tau_N(v)\bigr)$, is called the
\emph{Legendre map} associated to $L$.  
\par\smallskip\noindent

{\rm 3.\quad} The map $E_L:\RR\times TN\to \RR$ given by
 $$E_L(t,v)=\langle{\mathcal L}_L(t,v),v\bigr\rangle-L(t,v)\,,\quad t\in\RR\,,\ v\in TN\,,
 $$
is called the the \emph{energy function} associated to $L$.
\par\smallskip\noindent

{\rm 4.\quad} The $1$-form on $\RR\times TN$
 $$\widehat\varpi_L={\mathcal L}_L^*\theta_N-E_L(t,v)\d t\,,
 $$
where $\theta_N$ is the Liouville $1$-form on $T^*N$, is called the 
\emph{Euler-Poincaré} $1$-form. 
\end{defis}

\begin{theo}[Euler-Cartan theorem]\label{EulerCartan} 
A smooth curve $\gamma:[t_0,t_1]\to N$ parametrized by the time
$t\in[t_0,t_1]$ is a solution of the Euler-Lagrange equations if and only if, for each $t\in[t_0,t_1]$
the derivative with respect to $t$ of the map $\displaystyle t\mapsto\left(t,\frac{\d\gamma(t)}{\d t}\right)$
belongs to the kernel of the $2$-form $\d\widehat\varpi_L$, in other words if and only if
 $$\mathi\left(\frac{\d}{\d t}\left(t,\frac{\d\gamma(t)}{\d t}\right)\right)\d\widehat\varpi_L
        \left(t,\frac{\d\gamma(t)}{\d t}\right)=0\,.
 $$
\end{theo}

The interested reader will find the proof of that theorem in \cite{Malliavin}, (theorem 2.2, chapter IV, page 262) or,
for hyper-regular Lagrangians (an additional assumption which in fact, is not necessary) in \cite{Sternberg1964}, chapter IV, theorem 2.1 page 167.

\begin{rmk}
In his book \cite{Souriau1969}, \emph{Jean-Marie Souriau} 
uses a slightly different terminology: for him
the odd-dimensional space $\RR\times TN$ is the \emph{evolution space} of the system, and the exact $2$-form $\d\widehat\varpi_L$ on that space is the \emph{Lagrange form}. He defines that $2$-form in a
setting more general than that of the Lagrangian formalism.
\end{rmk}

\section{Lagrangian symmetries}\label{LagrangianSymmetries}
\label{LagrangianSymmetries}

\subsection{Assuumptions and notations}\label{AssumptionsLagrangianSymmetries}
In this section $N$ is the configuration space of a conservative Lagrangian mechanical system 
with a smooth, maybe time dependent Lagrangian $L:\RR\times TN\to \RR$. 
Let $\widehat\varpi_L$ be the Poincaré-Cartan $1$-form on the evolution space $\RR\times TN$. 
\par\smallskip

Several kinds of symmetries can be defined for such a system. Very often, they are  
special cases of \emph{infinitesimal symmetries of the Poincaré-Cartan form}, 
which play an important part in the famous \emph{Noether theorem}.  

\begin{defi}\label{DefiSymmetryPoincareCartan} 
An \emph{infinitesimal symmetry} of the Poincaré-Cartan form $\widehat\varpi_L$ is a 
vector field $Z$ on $\RR\times TN$ such that
 $${\mathcal L}(Z)\widehat\varpi_L=0\,,
 $$
$\mathcal L(Z)$ denoting the Lie derivative of differential forms with respect to $Z$. 
\end{defi}

\begin{exs}\label{ExamplesInfinitesimalSymmetriesPoincareCartan}\hfill
\par\noindent
{\rm 1.\quad}
Let us assume that the Lagrangian $L$ does not depend on the time $t\in\RR$, \emph{i.e.} 
is a smooth function on $TN$. The vector field on $\RR\times TN$ denoted by 
$\displaystyle \frac{\partial}{\partial t}$,
whose projection on $\RR$ is equal to $1$ and whose projection on $TN$ is $0$, is an infinitesimal symmetry of
$\widehat\varpi_L$.
\par\smallskip\noindent

{\rm 2.\quad} 
Let $X$ be a smooth vector field on $N$ 
and $\overline X$ be its canonical lift to the tangent bundle $TN$. 
We still assume that $L$ does not depend on the time $t$. Moreover we assume that $\overline X$ is an infinitesimal symmetry of the Lagrangian $L$, \emph{i.e.} that
${\mathcal L}(\overline X)L=0$. Considered as a vector field on $\RR\times TN$ whose projection 
on the factor $\RR$ is $0$, $\overline X$ is an infinitesimal symmetry of $\widehat\varpi_L$. 
\end{exs}

\subsection{The Noether theorem in Lagrangian formalism}\label{NoetherLagrange}

\begin{theo}[E.~Noether's theorem in Lagrangian formalism]\label{TheoremNoetherLagrange}
Let $Z$ be an infinitesimal symmetry of the Poincaré-Cartan form $\widehat\varpi_L$. 
For each possible motion $\gamma:[t_0,t_1]\to N$ of the Lagrangian system, the function
$\mathi(Z)\widehat\varpi_L$, defined on $\RR\times TN$,
keeps a constant value along the parametrized curve 
$\displaystyle t\mapsto\left(t,\frac{\d\gamma(t)}{\d t}\right)$.
\end{theo}

\begin{proof}
Let $\gamma:[t_0,t_1]\to N$ be a motion of the Lagrangian system, \emph{i.e.} a solution of the
Euler-Lagrange equations. The Euler-Cartan theorem \ref{EulerCartan} proves that, for any
$t\in[t_0,t_1]$,
 $$\mathi\left(\frac{\d}{\d t}\left(t,\frac{\d\gamma(t)}{\d t}\right)\right)\d\widehat\varpi_L
   \left(t,\frac{\d\gamma(t)}{\d t}\right)=0\,.
 $$
Since $Z$ is an infinitesimal symmetry of $\widehat\varpi_L$,
 $${\mathcal L}(Z)\widehat\varpi_L=0\,.
 $$
Using the well known formula relating the Lie derivative, the interior product and the exterior derivative 
 $${\mathcal L}(Z)=\mathi(Z)\circ \d+\d\circ\mathi(Z)
 $$
we can write
 \begin{align*}
  \frac{\d}{\d t}\left(\mathi(Z)\widetilde\varpi_L
   \left(t,\frac{\d\gamma(t)}{\d t}\right)\right)
  &=\left\langle\d\mathi(Z)\widehat\varpi_L,\frac{\d}{\d t}
      \left(t,\frac{\d\gamma(t)}{\d t}\right)\right\rangle\\
  &=-\left\langle\mathi(Z)\d\widehat\varpi_L,\frac{\d}{\d t}
      \left(t,\frac{\d\gamma(t)}{\d t}\right)\right\rangle\\
  &=0\,.\hfill\qedhere
 \end{align*}
\end{proof}

\begin{ex}\label{ExampleNotherLagrange} 
When the Lagrangian $L$ does not depend on time, application of Emmy Noether's theorem to the vector field $\displaystyle\frac{\partial}{\partial t}$ shows that the energy $E_L$ remains constant during any possible motion of the system, since 
$\displaystyle \mathi\left(\frac{\partial}{\partial t}\right)\widehat\varpi_L=-E_L$.  
\end{ex}

\begin{rmks}\label{GeneralizationNoetherLagrange}\hfill

\par\noindent
{\rm 1.\quad} Theorem \ref{TheoremNoetherLagrange} is due to the German mathematician
Emmy Noether (1882--1935), who proved it under much more general assumptions than 
those used here. For a very nice presentation of Emmy Noether's theorems in a much more
general setting and their applications in Mathematical Physics, interested readers are 
referred to the very nice book by Yvette Kosmann-Schwarzbach \cite{Kosmann2011}.

\par\smallskip\noindent
{\rm 2.\quad}
Several generalizations of the Noether theorem exist. For example, if instead of
being an infinitesimal symmetry of $\widehat\varpi_L$, \emph{i.e.} instead of satisfying
 ${\mathcal L}(Z)\widehat\varpi_L=0
 $
the vector field $Z$ satisfies
 $${\mathcal L}(Z)\widehat\varpi_L=\d f\,,
 $$
where $f:\RR\times TM\to \RR$ is a smooth function, which implies of course
 ${\mathcal L}(Z)(\d\widehat\varpi_L)=0
 $,
the function 
 $$\mathi(Z)\widehat\varpi_L-f
 $$
keeps a constant value along
$\displaystyle t\mapsto\left(t,\frac{\d\gamma(t)}{\d t}\right)$.
\par\smallskip
\end{rmks}

\subsection{The Lagrangian momentum map}\label{LagrangianMomentum}
The Lie bracket of two infinitesimal symmetries of $\widehat\varpi_L$ is too 
an infinitesimal symmetry
of $\widehat\varpi_L$. Let us therefore assume that there exists a finite-dimensional 
Lie algebra of vector fields on $\RR\times TN$ whose elements are infinitesimal symmetries 
of $\widehat\varpi_L$.

\begin{defi}\label{DefiLagrangianMomentum}
Let $\psi:{\mathcal G}\to A^1(\RR\times TN)$ be a Lie algebras homomorphism of a 
finite-dimensional real Lie algebra $\mathcal G$ into the Lie algebra of smooth vector fields 
on $\RR\times TN$ such that, for each $X\in {\mathcal G}$, $\psi(X)$ is an infinitesimal 
symmetry of $\widehat\varpi_L$. The Lie algebras homomorphism $\psi$ is said to be a 
\emph{Lie algebra action on $\RR\times TN$ by infinitesimal symmetries of $\widehat\varpi_L$}.
The map $K_L:\RR\times TN\to{\mathcal G}^*$, which takes its values in the dual 
${\mathcal G}^*$ of the Lie algebra $\mathcal G$, defined by
 $$\bigl\langle K_L(t,v),X\bigr\rangle=\mathi\bigl(\psi(X)\bigr)\widehat\varpi_L(t,v)\,,
    \quad X\in{\mathcal G}\,,\quad(t,v)\in\RR\times TN\,,
 $$
is called the \emph{Lagrangian momentum} of the Lie algebra action $\psi$. 
\end{defi} 

\begin{coro}[of E.~Noether's theorem]\label{CoroNoetherLagrange}
Let $\psi:{\mathcal G}\to A^1(\RR\times TM)$ be an action 
of a finite-dimensional real Lie algebra $\mathcal G$ 
on the evolution space $\RR\times TN$
of a conservative Lagrangian system, by infinitesimal symmetries of the Poincaré-Cartan form
$\widehat\varpi_L$. For each possible motion
$\gamma:[t_0,t_1]\to N$ of that system, the Lagrangian momentum map $K_L$
keeps a constant value along the parametrized curve 
$\displaystyle t\mapsto\left(t,\frac{\d\gamma(t)}{\d t}\right)$. 
\end{coro}

\begin{proof} Since for each $X\in{\mathcal G}$ the function 
$(t,v)\mapsto \bigl\langle K_L(t,v),X\bigr\rangle$ keeps a constant value
along the parametrized curve 
$\displaystyle t\mapsto\left(t,\frac{\d\gamma(t)}{\d t}\right)$, the map $K_L$ itself
keeps a constant value along that parametrized curve.
\end{proof}

\begin{ex}\label{ExampleCoroNoetherLagrange}
Let us assume that the Lagrangian $L$ does not depend explicitly on the time $t$ and 
is invariant by the canonical lift to the tangent bundle of the action on $N$ of the 
six-dimensional group of Euclidean diplacements (rotations and translations) of the 
physical space. The corresponding infinitesimal action of the Lie algebra of infinitesimal
Euclidean displacements (considered as an action on $\RR\times TN$, the action 
on the factor $\RR$ being trivial) is an action by infinitesimal symmetries of 
$\widehat\varpi_L$. The \emph{six components of the Lagrangian momentum map} are the 
\emph{three components of the total linear momentum} and the \emph{three components 
of the total angular momentum}.
\end{ex}

\begin{rmk} 
These results are valid \emph{without any assumption of hyper-regularity} of the Lagrangian.
\end{rmk}

\subsection{The Euler-Poincar\'e equation}\label{EulerPoincareEquation}

In a short Note \cite{Poincare01} published in 1901, the great french mathematician 
\emph{Henri Poincaré} (1854--1912) proposed a new formulation of the equations of Mechanics. 
\par\smallskip

Let $N$ be the configuration manifold of a conservative
Lagrangian system, with a smooth Lagrangian $L:TN\to \RR$ which does not depend
explicitly on time. Poincaré assumes that there exists an homomorphism $\psi$ of a finite-dimensional 
real Lie algebra $\mathcal G$ into the Lie algebra $A^1(N)$ of smooth vector fields on $N$, such that
for each $x\in N$, the values at $x$ of the vetor fields $\psi(X)$, when $X$ varies in $\mathcal G$,
completely fill the tangent space $T_xN$. 
The action $\psi$ is then said to be 
\emph{locally transitive}.
\par\smallskip

Of course these assumptions imply $\dim{\mathcal G}\geq\dim N$.
\par\smallskip

Under these assumptions, Henri Poincaré
proved that the equations of motion of the Lagrangian system could be written
on $N\times{\mathcal G}$ or on $N\times{\mathcal G}^*$, where ${\mathcal G}^*$ is the dual of the
Lie algebra $\mathcal G$, instead of on the tangent bundle $TN$. When $\dim{\mathcal G}=\dim N$
(which can occur only when the tangent bundle $TN$ is trivial) the obtained equation, called the
\emph{Euler-Poincaré equation}, is \emph{perfectly equivalent to the Euler-Lagrange equations} and may, 
in certain cases, be easier to use. But when $\dim{\mathcal G}>\dim N$, the system made by the Euler-Poincaré equation is 
\emph{underdetermined}.  
\par\smallskip

Let $\gamma:[t_0,t_1]\to N$ be a smooth parametrized curve in $N$. Poincaré proves that there exists a
smooth curve $V:[t_0,t_1]\to{\mathcal G}$ in the Lie algebra ${\mathcal G}$ such that, for each
$t\in[t_0,t_1]$,
 $$\psi\bigl(V(t)\bigr)\bigl(\gamma(t)\bigr)=\frac{\d\gamma(t)}{\d t}\,.\eqno(*)
 $$
When $\dim{\mathcal G}>\dim N$ the smooth curve $V$ in $\mathcal G$ is not uniquely 
determined by the smooth curve $\gamma$ in $N$. However, instead of writing the 
second-order Euler-Lagrange differential equations on $TN$ satisfied by $\gamma$ 
when this curve is a possible motion of the Lagrangian system,
Poincaré derives a \emph{first order differential equation for the curve $V$} and proves that it is 
satisfied, together with Equation $(*)$,
\emph{if and only if $\gamma$ is a possible motion of the Lagrangian system}.
\par\smallskip

Let $\varphi:N\times{\mathcal G}\to TN$ and $\overline L:N\times{\mathcal G}\to\RR$ be the maps
 $$\varphi(x,X)=\psi(X)(x)\,,\quad \overline L(x,X)=L\circ\varphi(x,X)\,.
 $$
We denote by $\d_1\overline L:N\times{\mathcal G}\to T^*N$ and by 
$d_2\overline L:N\times{\mathcal G}\to{\mathcal G}^*$ the partial differentials of 
$\overline L:N\times{\mathcal G}\to\RR$ with respect to its first variable $x\in N$ and with respect to its second variable $X\in{\mathcal G}$.
\par\smallskip

The map $\varphi:N\times{\mathcal G}\to TN$ is a \emph{surjective vector bundles morphism} 
of the trivial vector bundle $N\times{\mathcal G}$ into the tangent bundle $TN$. Its 
\emph{transpose}
$\varphi^T:T^*N\to N\times{\mathcal G}^*$ is therefore an \emph{injective vector bundles morphism},
which can be written
 $$\varphi^T(\xi)=\bigl(\pi_N(\xi),J(\xi)\bigr)\,,
 $$
where $\pi_N:T^*N\to N$ is the canonical projection of the cotangent bundle and $J:T^*N\to{\mathcal G}^*$
is a smooth map whose restriction to each fibre $T_x^*N$ of the cotangent bundle is linear, and is the 
transpose of the map $X\mapsto\varphi(x,X)=\psi(X)(x)$.

\begin{rmk}\label{RemarkMomentumMapEulerPoincare}
The homomorphism $\psi$ of the Lie algebra $\mathcal G$ into the Lie algebra
$A^1(N)$ of smooth vector fields on $N$ is an action of that Lie algebra, in the sense defined below 
(\ref{DefiActionLieGroupLieAlgebra}). That action can be canonically lifted into a Hamiltonian action of $\mathcal G$ on $T^*N$, endowed with
its canonical symplectic form $\d\theta_N$ (\ref{DefisHamiltonianAction}). The map $J$ is in fact a 
\emph{Hamiltonian momentum map} for that Hamiltonian action (\ref{ExistenceMomentum}).
\end{rmk}
 
Let ${\mathcal L}_L=\d_{\rm vert}L:TN\to T^*N$ be the \emph{Legendre map} defined in \ref{DefisLegendreMapEtAl}.

\begin{theo}[Euler-Poincaré equation]\label{EulerPoincareEquation} 
With the above defined notations, let $\gamma:[t_0,t_1]\to N$
be a smooth parametrized curve in $N$ and $V:[t_0,t_1]\to{\mathcal G}$ be a smooth parametrized curve
such that, for each $t\in[t_0,t_1]$,
$$\psi\bigl(V(t)\bigr)\bigl(\gamma(t)\bigr)=\frac{\d\gamma(t)}{\d t}\,.\eqno(*)
 $$
The curve $\gamma$ is a possible motion of the Lagrangian system if and only if 
$V$ satisfies the equation
 $$\left(\frac{\d}{\d t}-\ad^*_{V(t)}\right)\Bigl(J\circ{\mathcal L}_L\circ\varphi\bigl(\gamma(t),V(t)\bigr)\Bigr)
     -J\circ\d_1 \overline L\bigl(\gamma(t),V(t)\bigr)=0\,.\eqno({*}{*})
 $$
\end{theo}

The interested reader will find a proof of that theorem in local coordinates in the original Note by Poincaré \cite{Poincare01}. More intrinsic proofs can be found in \cite{marlejgsp, Marle2014}. Another proof is possible, in which
that theorem is deduced from the Euler-Cartan theorem \ref{EulerCartan}.

\begin{rmk}
Equation $(*)$ is called the \emph{compatibility condition} and Equation $(**)$ is the
\emph{Euler-Poincaré equation}. It can be written under the equivalent form
 $$\left(\frac{\d}{\d t}-\ad^*_{V(t)}\right)\Bigl(\d_2\overline L\bigl(\gamma(t),V(t)\bigr)\Bigr)
     -J\circ\d_1 \overline L\bigl(\gamma(t),V(t)\bigr)=0\,.\eqno({*}{*}{*})
 $$
\end{rmk}

Examples of applications of the Euler-Poincaré equation can be found in 
\cite{holm3, marlejgsp, Marle2014} and, for an application in Thermodynamics,
\cite{Barbaresco2015}.

\section{The Hamiltonian formalism}\label{HamiltonianFormalism}

The Lagrangian formalism can be applied to any smooth Lagrangian. Its application yields 
\emph{second order differential equations} on $\RR\times TN$ (in local coordinates, the 
\emph{Euler-Lagrange equations})
which in general \emph{are not solved with respect to the second order derivatives 
of the unknown functions with respect to time}. The classical existence and unicity theorems 
for the solutions of differential equations (such as the \emph{Cauchy-Lipschitz theorem})
therefore \emph{cannot be applied to these equations}. 

Under the additional assumption that the Lagrangian is \emph{hyper-regular}, 
a very clever change of variables discovered by 
\emph{William Rowan Hamilton}~\footnote{\emph{Lagrange} obtained however Hamilton's equations before 
\emph{Hamilton}, but
only in a special case, for the slow \lq\lq variations of constants\rq\rq\ such 
as the orbital parameters of planets in the solar system \cite{Lagrange3, Lagrange4}.} 
\cite{hamilton8, hamilton9}
allows a new formulation of these equations in the framework of \emph{symplectic geometry}. 
The \emph{Hamiltonian formalism} discussed below is the use of these new equations. 
It was later generalized independently of the Lagrangian formalism.

\subsection{Hyper-regular Lagrangians}\label{HyperRegularLagrangians}

\subsubsection{Assumptions made in this section}\label{AssumptionsHyperRegularLagrangians}
We consider in this section a smooth, maybe time-dependent Lagrangian $L:\RR\times TN\to\RR$,
which is such that the Legendre map 
(\ref{DefisLegendreMapEtAl}) ${\mathcal L}_L:\RR\times TN\to T^*N$ satisfies the following property:
for each fixed value of the time $t\in\RR$, the map
$v\mapsto{\mathcal L}_L(t,v)$ is a smooth diffeomorphism of the tangent bundle $TN$ 
onto the cotangent bundle $T^*N$. An equivalent assumption is the following: the map
$(\id_\RR,{\mathcal L}_L):(t,v)\mapsto\bigl(t,{\mathcal L}_L(t,v)\bigr)$ is a smooth 
diffeomorphism of $\RR\times TN$
onto $\RR\times T^*N$. The Lagrangian $L$ is then said to be 
\emph{hyper-regular}. The equations of motion can be
written on $\RR\times T^*N$ instead of $\RR\times TN$.

\begin{defis}\label{DefiHamiltonian}
Under the assumption \ref{AssumptionsHyperRegularLagrangians}, the function 
$H_L:\RR\times T^*N\to \RR$ given by
 $$H_L(t,p)=E_L\circ(\id_\RR,{\mathcal L}_L)^{-1}(t,p)\,,\quad t\in\RR\,,\ p\in T^*N\,,
 $$    
($E_L:\RR\times TN\to \RR$ being the \emph{energy function} defined in \ref{DefisLegendreMapEtAl}) 
is called the \emph{Hamiltonian} associated to the hyper-regular Lagrangian $L$.
\par\smallskip

The $1$ form defined on $\RR\times T^*N$
 $$\widehat\varpi_{H_L}=\theta_N - H_L\d t\,,
 $$
where $\theta_N$ is the \emph{Liouville $1$-form} on $T^*N$, is 
called the \emph{Poincaré-Cartan $1$-form in the Hamiltonian formalism}.
\end{defis}

\begin{rmk}\label{PoincareCartanLagrangianAndHamiltonian}
The \emph{Poincaré-Cartan $1$-form} $\widehat\varpi_L$ on $\RR\times TN$, defined in
\ref{DefisLegendreMapEtAl}, is the pull-back, by the diffeomorphism
$(\id_\RR,{\mathcal L}_L):\RR\times TN\to\RR\times T^*N$, of the 
Poincaré-Cartan $1$-form $\widehat\varpi_{H_L}$ in the Hamiltonian
formalism on $\RR\times T^*N$ defined above. 
\end{rmk}

\subsection{Presymplectic manifolds}

\begin{defis}\label{DefiPresymplectic}
 A \emph{presymplectic form} on a smooth manifold $M$ is a $2$-form $\omega$ on $M$ which is closed, \emph{i.e.} such that $\d\omega=0$. A manifold $M$ equipped with a presymplectic form
$\omega$ is called a \emph{presymplectic manifold} and denoted by $(M,\omega)$.
\par\smallskip
The \emph{kernel} $\ker\omega$ of a presymplectic form $\omega$ defined on a smooth manifold $M$ is the set of vectors $v\in TM$ such that $\mathi(v)\omega=0$. 
\end{defis}
 
\begin{rmks}\label{KernelOfPresymplecticForm} 
A symplectic form $\omega$ on a manifold $M$ is a presymplectic form which, 
moreover, is non-degenerate, \emph{i.e.} such that for each $x\in M$ and each non-zero
vector $v\in T_xM$, there exists another vector $w\in T_xM$ such that $\omega(x)(v,w)\neq 0$.
Or in other words, a presymplectic form $\omega$ whose kernel is the set of null vectors.
\par\smallskip

The kernel of a presymplectic form $\omega$ on a smooth manifold $M$ is a vector
sub-bundle of $TM$ if and only if for each $x\in M$, the vector subspace 
$T_xM$ of vectors $v\in T_xM$ which satisfy
$\mathi(v)\omega=0$ is of a fixed dimension, the same for all points $x\in M$. A presymplectic 
form which satisfies that condition is said to be of \emph{constant rank}. 
\end{rmks}

\begin{prop}\label{QuotientOfPresymplecticManifold}
Let $\omega$ be a presymplectic form of constant rank (\ref{KernelOfPresymplecticForm}) on a smooth
manifold $M$. The kernel $\ker\omega$ of $\omega$ is a completely integrable vector sub-bundle of $TM$, 
which defines a foliation ${\mathcal F}_\omega$ of $M$ into connected immersed submanifolds which, 
at each point of $M$, have the fibre of $\ker\omega$ at that point as tangent vector space. 
\par\smallskip

We now assume in addition that this foliation is simple, \emph{i.e.} such that the set of leaves of 
${\mathcal F}_\omega$, denoted by $M/\ker\omega$, has a smooth manifold structure for which the canonical projection 
$p:M\to M/\ker\omega$ (which associates to each point $x\in M$ the leaf which contains $x$) is a smooth submersion.
There exists on $M/\ker\omega$ a unique symplectic form $\omega_r$ such that
  $$\omega= p^*\omega_r\,.
  $$
\end{prop}

\begin{proof} Since $\d\omega=0$, the fact that $\ker\omega$ is completely integrable is an
immediate consequence of the Frobenius' theorem (\cite{Sternberg1964}, chapter III,
theorem 5.1 page 132). The existence and unicity 
of a symplectic form $\omega_r$ on $M/\ker\omega$ such that $\omega=p^*\omega_r$ results from the
fact that $M/\ker\omega$ is built by quotienting $M$ by the kernel of $\omega$.
\end{proof}

\subsubsection{Presymplectic manifolds in Mechanics}
\label{PresymplecticInMechanics}
Let us go back to the assumptions and notations of \ref{AssumptionsHyperRegularLagrangians}.
We have seen in \ref{PoincareCartanLagrangianAndHamiltonian} that 
the Poincaré-Cartan $1$-form in Hamiltonian formalism $\widehat\varpi_{H_L}$ on $\RR\times T^*N$
and the Poincaré-Cartan $1$-form in Lagrangian formalism $\widehat\varpi_L$ on $\RR\times TN$
are related by
 $$\widehat\varpi_L=(\id_\RR,{\mathcal L}_L)^*\widehat\varpi_{H_L}\,.
 $$
Their exterior differentials $\d\widehat\varpi_L$ and $\d\widehat\varpi_{H_L}$ 
both are \emph{presymplectic $2$-forms} on the odd-dimensional manifolds $\RR\times TN$ and $\RR\times T^*N$,
respectively. At any point of these manifolds, the kernels 
of these closed $2$ forms are one-dimensional. They therefore (\ref{QuotientOfPresymplecticManifold}) 
determine \emph{foliations into smooth curves} of these manifolds. The Euler-Cartan theorem (\ref{EulerCartan})
shows that each of these curves is a possible \emph{motion} of the system, described either in the Lagrangian formalism, 
or in the Hamiltonian formalism, respectively. 
\par\smallskip

The set of all possible motions of the system, called by Jean-Marie Souriau the \emph{manifold of motions}
of the system, is described by the quotient $(\RR\times TN)/\ker\d\widehat\varpi_L$ in the Lagrangian formalism,  
and by the quotient $(\RR\times T^*N)/\ker\d\widehat\varpi_{H_L}$ in the Hamiltonian formalism. 
Both are (maybe non-Hausdorff) \emph{symplectic manifolds}, the projections on these quotient manifolds 
of the presymplectic forms $\d\widehat\varpi_L$ and $\d\widehat\varpi_{H_L}$ both being symplectic forms. Of course
the diffeomorphism $(\id_\RR,{\mathcal L}_L):\RR\times TN\to\RR\times T^*N$ projects onto a symplectomorphism
between the Lagrangian and Hamiltonian descriptions of the manifold of motions of the system.

\subsection{The Hamilton equation}

\begin{prop}\label{HamiltonAndEnergyEquations}
Let $N$ be the configuration manifold of a Lagrangian system whose Lagrangian 
$L:\RR\times TN\to \RR$, maybe time-dependent, is smooth and hyper-regular, and $H_L:\RR\times T^*N\to\RR$
be the associated Hamiltonian (\ref{DefiHamiltonian}). 
Let $\varphi:[t_0,t_1]\to N$ be a smooth curve parametrized by the time $t\in[t_0,t_1]$, and let
$\psi:[t_0,t_1]\to T^*N$ be the parametrized curve in $T^*N$
 $$\psi(t)={\mathcal L}_L\left(t,\frac{\d\gamma(t)}{\d t}\right)\,,\quad t\in[t_0,t_1]\,,
 $$
where ${\mathcal L}_L:\RR\times TN\to T^*N$ is the Legendre map (\ref{DefisLegendreMapEtAl}).
\par\smallskip

The parametrized curve $t\mapsto\gamma(t)$ is a motion of the system if and only if the parametrized 
curve $t\mapsto\psi(t)$ satisfies the equatin, called the Hamilton equation,
 $$\mathi\left(\frac{\d\psi(t)}{\d t}\right)\d\theta_N
    =-\d H_{L\,t}\,,
 $$
where $\displaystyle \d H_{L\,t}=\d H_L-\frac{\partial H_L}{\partial t}\,\d t$ is the differential of the function
$H_{L\,t}:T^*N\to \RR$ in which the time $t$ is considered as a parameter with respect to which there is no
differentiation.
\par\smallskip

When the parametrized curve $\psi$ satisfies the Hamilton equation stated above, it satisfies too the equation,
called the energy eqution
 $$\frac{\d}{\d t}\Bigl(H_L\bigl(t,\psi(t)\bigr)\Bigr)
    =\frac{\partial H_L}{\partial t}\bigl(t,\psi(t)\bigr)\,.
 $$
\end{prop}

\begin{proof}
These results directly follow from the Euler-Cartan theorem (\ref{EulerCartan}).
\end{proof}

\begin{rmks}
The $2$-form $\d\theta_N$ is a symplectic form on the cotangent bundle $T^*N$, called its
\emph{canonical symplectic form}. We have shown that when the Lagrangian $L$ is hyper-regular,
the equations of motion can be written in \emph{three equivalent manners}:

\begin{enumerate} 

\item{} as the \emph{Euler-Lagrange equations} on $\RR\times TM$,

\item{} as the equations given by the \emph{kernels of the presymplectic forms} $\d\widehat\varpi_L$ or 
$\d\widehat\varpi_{H_L}$ which determine the foliations into curves of the evolution spaces $\RR\times TM$
in the Lagrangian formalism, or $\RR\times T^*M$ in the Hamiltonian formalism,

\item{} as the \emph{Hamilton equation} associated to the Hamiltonian $H_L$ on the symplectic manifold
$(T^*N,\d\theta_N)$, often called the \emph{phase space} of the system.
\end{enumerate}
\end{rmks}

\subsubsection{The Tulczyjew isomorphisms}\label{TulczyjewIsomorphisms}
Around 1974, \emph{W.M.~Tulczyjew} \cite{Tulczyjew1974, Tulczyjew1989} discovered~\footnote{$\beta_N$ 
was probably known long before 1974, but I believe that $\alpha_N$, much more hidden, was noticed 
by Tulczyjew for the first time.} 
two remarkable vector bundles isomorphisms
$\alpha_N:TT^*N\to T^*TN$ and $\beta_N:TT^*N\to T^*T^*N$.
\par\smallskip

The first one $\alpha_N$ is an isomorphism
of the bundle $(TT^*N,T\pi_N,TN)$ onto the bundle $(T^*TN,\pi_{TN},TN)$, while the second
$\beta_N$ is an isomorphism of the bundle $(TT^*N,\tau_{T^*N}, T^*N)$ onto the bundle
$(T^*T^*N,\pi_{T^*N},T^*N)$. The diagram below is commutative.
 $$
  \xymatrix{
   {T^*T^*N}\ar[d]_{\pi_{T^*N}}
  &{TT^*N}\ar[l]_{\beta_N} 
           \ar[dl]^{\tau_{T^*N}} 
            \ar[dr]_{T\pi_N} 
             \ar[r]^{\alpha_N}
  &{T^*TN}\ar[d]^{\pi_{TN}}
           \\
   {T^*N}\ar[dr]_{\pi_N}
  &&{TN}\ar[dl]^{\tau_N}
  \\
  &{N}&&
  }
 $$
Since they are the total spaces of cotangent bundles, the manifolds $T^*TN$ and $T^*T^*N$
are endowed with the Liouville $1$-forms $\theta_{TN}$ and $\theta_{T^*N}$,
and with the canonical symplectic forms $\d\theta_{TN}$ and $\d\theta_{T^*N}$, respectively.
Using the isomorphisms $\alpha_N$ and $\beta_N$, we can therefore define on $TT^*N$ 
two $1$-forms $\alpha_N^*\theta_{TN}$ and $\beta_N^*\theta_{T^*N}$, and two symplectic $2$-forms
$\alpha_N^*(\d\theta_{TN})$ and $\beta_N^*(\d\theta_{T^*N})$.
The very remarkable property of the isomorphisms $\alpha_N$ and $\beta_N$ is that 
\emph{the two symplectic forms so obtained on $TT^*N$ are equal}:
 $$\alpha_N^*(\d\theta_{TN})=\beta_N^*(\d\theta_{T^*N})\,.
 $$ 
The $1$-forms $\alpha_N^*\theta_{TN}$ and $\beta_N^*\theta_{T^*N}$ are not equal, their difference is
the differential of a smooth function.
\par\smallskip

\subsubsection{Lagrangian submanifolds}\label{LagrangianSubmanifolds}
In view of applications to implicit Hamiltonian systems, let us recall here that a Lagrangian submanifold
of a symplectic manifold $(M,\omega)$ is a submanifold $N$ whose dimension is half the dimension of $M$,
on which the form induced by the symplectic form $\omega$ is $0$.
\par\smallskip

Let $L:TN\to\RR$ and $H:T^*N\to\RR$ be two smooth real valued functions, defined on $TN$ and on $T^*N$,
respectively.
The graphs $\d L(TN)$ and $\d H(T^*N)$ of their differentials are Lagrangian submanifolds
of the symplectic manifolds $(T^*TN,\d\theta_{TN})$ and $(T^*T^*N,\d\theta_{T^*N})$.
Their pull-backs
$\alpha_N^{-1}\bigl(\d L(TN)\bigr)$ and $\beta_N^{-1}\bigl(\d H(T^*N)\bigr)$ by the symplectomorphisms
$\alpha_N$ and $\beta_N$ are therefore two Lagrangian submanifolds of the manifold $TT^*N$ 
endowed with the  symplectic form $\alpha_N^*(\d\theta_{TN})$, which is equal to the symplectic form
$\beta_N^*(\d\theta_{T^*N})$.
\par\smallskip

The following theorem enlightens some aspects of the relationships between the Hamiltonian and
the Lagrangian formalisms. 

\begin{theo}[W.M.~Tulczyjew]
With the notations specified above (\ref{LagrangianSubmanifolds}),
let $X_H:T^*N\to TT^*N$ be the Hamiltonian vector field on the symplectic manifold $(T^*N,\d\theta_N)$
associated to the Hamiltonian $H:T^*N\to\RR$, defined by
$\mathi(X_H)\d\theta_N=-\d H$. Then
 $$X_H(T^*N)=\beta_N^{-1}\bigl(\d H(T^*N)\bigr)\,.
 $$
Moreover, the equality
 $$\alpha_N^{-1}\bigl(\d L(TN)\bigr)=\beta_N^{-1}\bigl(\d H(T^*N)\bigr)
 $$
holds if and only if the Lagrangian $L$ is hyper-regular and such that
 $$\d H=\d\bigl(E_L\circ{\mathcal L}_L^{-1}\bigr)\,,
 $$
where ${\mathcal L}_L:TN\to T^*N$ is the Legendre map and $E_L:TN\to\RR$ the energy
associated to the Lagrangian $L$.
\end{theo}

The interested reader will find the proof of that theorem in the works of W. Tulczyjew 
(\cite{Tulczyjew1974, Tulczyjew1989}).
\par\smallskip

When $L$ is not hyper-regular, $\alpha_N^{-1}\bigl(\d L(TN)\bigr)$ still is a Lagrangian submanifold
of the symplectic manifold $\bigl(TT^*N,\alpha_N^*(\d\theta_{TN})\bigr)$, but it is no more the graph
of a smooth vector field $X_H$ defined on $T^*N$. Tulczyjew proposes to consider this Lagrangian
submanifold as an \emph{implicit Hamilton equation} on $T^*N$.
\par\smallskip

These results can be extended to Lagrangians and Hamiltonians which may depend on time.

\subsection{The Hamiltonian formalism on symplectic and Poisson manifolds}\label{HamiltonOnSimplecticAndPoisson}

\subsubsection{The Hamilton formalism on symplectic manifolds}
\label{HamiltonianFormalismOnSymplecticManifolds}
In pure mathematics as well as in applications of mathematics to Mechanics and Physics, 
symplectic manifolds other than cotangent bundles are encountered. A theorem due to the
french mathematician \emph{Gaston Darboux} (1842--1917) asserts that any symplectic manifold 
$(M,\omega)$ is of even dimension $2n$ and is locally isomorphic to the cotangent bundle to a 
$n$-dimensional manifold: 
in a neighbourhood of each of its point there exist local coordinates
$(x^1,\ldots,x^n,p_1,\ldots,p_n)$, called \emph{Darboux coordinates} with which the symplectic form $\omega$ is expressed exactly as the canonical symplectic form of a cotangent bundle:
 $$\omega=\sum_{i=1}^n\d p_i\wedge\d x^i\,.
 $$ 
Let $(M,\omega)$ be a symplectic manifold and $H:\RR\times M\to\RR$ a smooth function, said to be a
\emph{time-dependent Hamiltonian}. It determines a \emph{time-dependent Hamiltonian vector field}
$X_H$ on $M$, such that
 $$\mathi(X_H)\omega=-\d H_t\,,
 $$
$H_t:M\to\RR$ being the function $H$ in which the variable $t$ is considered as a parameter 
with respect to which no differentiation is made. 
\par\smallskip

The \emph{Hamilton equation} determined by $H$ is the differential equation
 $$\frac{\d \psi(t)}{\d t}= X_H\bigl(t,\psi(t)\bigr)\,.
 $$
The Hamiltonian formalism can therefore be applied to any smooth, maybe time dependent Hamiltonian on $M$,
even when there is no associated Lagrangian. 
\par\smallskip

The Hamiltonian formalism is not limited to symplectic manifolds: it can be applied, for example, to 
\emph{Poisson manifolds} \cite{Lichnerowicz77},
\emph{contact manifolds} and \emph{Jacobi manifolds}
\cite{Lichnerowicz79}. For simplicity
I will consider only Poisson manifolds. Readers interested in Jacobi manifolds and their generalizations are referred to the papers by A.~Lichnerowicz quoted above and to the very
important paper by A.~Kirillov \cite{Kirillov76}.

\begin{defi}\label{DefiPoissonManifold} 
A \emph{Poisson manifold} is a smooth manifold $P$ whose algebra of smooth functions
$C^\infty(P,\RR)$ is endowed with a bilinear composition law, called the \emph{Poisson bracket}, 
which associates to any pair $(f,g)$ of smooth functions on $P$ another smooth function 
denoted by $\{f,g\}$, that composition satisfying the three properties
\begin{enumerate}

\item{} it is skew-symmetric,\\
\centerline{$\displaystyle\{g,f\}=-\{f,g\}$,}

\item{} it satisfies the \emph{Jacobi identity}\\
\centerline{ $\displaystyle\bigl\{f,\{g,h\}\bigr\}+\bigl\{g,\{h,f\}\bigr\}+\bigl\{h,\{f,g\}\bigr\}=0$,}   

\item{} it satisfies the Leibniz identity\\
\centerline{ $\displaystyle \{f,gh\}=\{f,g\}h+g\{f,h\}$.}
\end{enumerate}
\end{defi}

\begin{exs}\label{ExamplesPoissonManifolds}\hfill
\par\noindent
{\rm 1.\quad} On the vector space of smooth functions defined on a symplectic manifold
$(M,\omega)$, there exists a composition law, called the \emph{Poisson bracket}, which satisfies 
the properties stated in \ref{DefiPoissonManifold}. Let us recall briefly its definition. The symplectic form
$\omega$ allows us to associate, to any smooth function $f\in C^\infty(M,\RR)$, a smooth vector field 
$X_f\in A^1(M,\RR)$, called the \emph{Hamiltonian vector field associated to $f$}, defined by
 $$\mathi(X_f)\omega=-\d f\,.
 $$ 
The Poisson bracket $\{f,g\}$ of two smooth functions $f$ and $g\in C^\infty(M,\RR)$
is defined by the three equivalent equalities
 $$\{f,g\}=\mathi(X_f)\d g=-\mathi(X_g)\d f=\omega(X_f,X_g)\,.
 $$
Any symplectic manifold is therefore a Poisson manifold.
\par\smallskip

The Poisson bracket of smooth functions defined on a symplectic manifold (when that symplectic manifold is a cotangent bundle) was discovered by Siméon Denis Poisson
(1781--1840) \cite{Poisson2}.

\par\smallskip\noindent
{\rm 2.\quad} Let $\mathcal G$ be a finite-dimensional real Lie algebra, and let ${\mathcal G}^*$
be its dual vector space. For each smooth function
$f\in C^\infty({\mathcal G}^*,\RR)$ and each $\zeta\in{\mathcal G}^*$, the differential $\d f(\zeta)$ 
is a linear form on ${\mathcal G}^*$, in other words an element of the dual vector space of ${\mathcal G}^*$.
Identifying with $\mathcal G$ the dual vector space of ${\mathcal G}^*$, we can therefore consider
$\d f(\zeta)$ as an element in $\mathcal G$. With this identification, we can define the Poisson bracket
of two smooth functions $f$ and $g\in C^\infty({\mathcal G}^*,\RR$ by
 $$\{f,g\}(\zeta)=\bigl[\d f(\zeta),\d g(\zeta)\bigr]\,,\quad \zeta\in{\mathcal G}^*\,,
 $$ 
the bracket in the right hand side being the bracket in the Lie algebra $\mathcal G$.
The Poisson bracket of functions in $C^\infty({\mathcal G}^*,\RR)$ so defined satifies the properties
stated in \ref{DefiPoissonManifold}. The dual vector space of any finite-dimensional real Lie algebra is therefore
endowed with a Poisson structure, called its \emph{canonical Lie-Poisson structure} or its
\emph{Kirillov-Kostant-Souriau Poisson structure}. Discovered by Sophus Lie, this structure was indeed rediscovered 
independently by Alexander Kirillov, Bertram Kostant and Jean-Marie Souriau. 

\par\smallskip\noindent
{\rm 3.\quad} A \emph{symplectic cocycle} of a finite-dimensional, real Lie algebra $\mathcal G$ 
is a skew-symmetric bilinear map $\Theta:{\mathcal G}\times{\mathcal G}\to{\mathcal G}^*$ which 
satisfies, for all $X$, $Y$ and $Z\in{\mathcal G}$, 
 $$\Theta\bigl([X,Y],Z\bigr) + \Theta\bigl([Y,Z],X\bigr) + \Theta\bigl([Z,X],Y\bigr)=0\,.
 $$
The canonical Lie-Poisson bracket of two smooth functions $f$ and $g\in C^\infty({\mathcal G}^*,\RR)$
can be modified by means of the symplectic cocycle $\Theta$, by setting
 $$\{f,g\}_\Theta(\zeta)=\bigl[\d f(\zeta),\d g(\zeta)\bigr] 
  - \Theta\bigl(\d f(\zeta),\d g(\zeta)\bigr)\,,\quad \zeta\in{\mathcal G}^*\,.
 $$
This bracket still satifies the properties
stated in \ref{DefiPoissonManifold}, therefore defines on ${\mathcal G}^*$ a Poisson structure called its
\emph{canonical Lie-Poisson structure modified by $\Theta$}.
\end{exs}

\subsubsection{Properties of Poisson manifolds}\label{PropertiesPoissonManifolds}
The interested reader will find the proofs of the properties recalled here
in \cite{Vaisman94}, \cite{LibermannMarle1987}, \cite{LaurentPichereauVanhaecke} or \cite{OrtegaRatiu2004}.

\par\smallskip\noindent
{\rm 1.\quad}
On a Poisson manifold $P$, the Poisson bracket $\{f,g\}$ of two smooth functions $f$and $g$
can be expressed by means of a smooth field of bivectors $\Lambda$:
 $$\{f,g\}=\Lambda(\d f,\d g)\,,\quad f\ \hbox{and}\ g\in C^\infty(P,\RR)\,,
 $$
called the \emph{Poisson bivector field} of $P$. The considered Poisson manifold is often
denoted by $(P,\Lambda)$. The Poisson bivector field $\Lambda$ identically satisfies
 $$[\Lambda,\Lambda]=0\,,
 $$
the bracket $[\ ,\ ]$ in the left hand side being the \emph{Schouten-Nijenhuis bracket}.
That bivector field
determines a vector bundle morphism $\Lambda^\sharp:T^*P\to TP$, defined by
 $$\Lambda(\eta,\zeta)=\bigl\langle\zeta,\Lambda^\sharp(\eta)\bigr\rangle\,,
 $$
where $\eta$ and $\zeta\in T^*P$ are two covectors attached to the same point in $P$.
\par\smallskip

Readers interested in the Schouten-Nijenhuis bracket will find thorough presentations of its properties in \cite{Koszul85} or \cite{Marle2008}.

\par\smallskip\noindent
{\rm 2.\quad}
Let $(P,\Lambda)$ be a Poisson manifold. A (maybe time-dependent) vector field on $P$ can be
associated to each (maybe time-dependent) smooth function $H:\RR\times P\to \RR$. It is called 
the \emph{Hamiltonian vector field} associated to the \emph{Hamiltonian} $H$, and denoted by 
$X_H$. 
Its expression is
 $$X_H(t,x)=\Lambda^\sharp(x)\bigl(\d H_t(x)\bigr)\,,
 $$
where $\displaystyle \d H_t(x)=\d H(t,x)-\frac{\partial H(t,x)}{\partial t}\d t$ is the differential
of the function deduced from $H$ by considering $t$ as a parameter with respect to which no differentiation
is made.  
\par\smallskip

The \emph{Hamilton equation} determined by the (maybe time-dependent) Hamiltonian $H$ is
 $$\frac{\d\varphi(t)}{\d t}=X_H(\bigl(t,\varphi(t)\bigr)
                            =\Lambda^\sharp(\d H_t)\bigl(\varphi(t)\bigr)\,.
 $$
\par\smallskip\noindent
{\rm 3.\quad} Any Poisson manifold is foliated, by a generalized foliation whose leaves
may not be all of the same dimension, into immersed connected symplectic manifolds called
the \emph{symplectic leaves} of the Poisson manifold. The value, at any point of a Poisson
manifold, of the Poisson bracket of two smooth functions only depends on the restrictions of these functions to the symplectic leaf through the considered  point, and can be calculated
as the Poisson bracket of functions defined on that leaf,  with the Poisson structure associated to the symplectic structure of that leaf. This property was discovered by 
Alan Weinstein, in his very thorough study of the local structure of Poisson manifolds
\cite{Weinstein83}.
   
\section{Hamiltonian symmetries}\label{HamiltonianSymmetries}

\subsection{Presymplectic, symplectic and Poisson maps and vector fields}

Let $M$ be a manifold endowed with some structure, which can be either 

\begin{itemize}
 
\item{} a \emph{presymplectic structure}, determined by a presymplectic form, \emph{i.e.}, a $2$-form $\omega$
which is closed ($\d\omega=0$),

\item{} a \emph{symplectic structure}, determined by a symplectic form $\omega$, \emph{i.e.}, a $2$-form 
$\omega$ which is both  closed ($\d\omega=0$) and nondegenerate ($\ker\omega=\{0\}$),

\item{} a \emph{Poisson structure}, determined by a smooth Poisson bivector field $\Lambda$ satisfying $[\Lambda,\Lambda]=0$.

\end{itemize}

\begin{defi} A \emph{presymplectic} (resp. \emph{symplectic}, resp. \emph {Poisson}) diffeomorphism of a
presymplectic (resp., symplectic, resp. Poisson) manifold $(M,\omega)$ (resp. $(M,\Lambda)$) is a 
smooth diffeomorphism $f:M\to M$ such that $f^*\omega=\omega$ (resp. $f^*\Lambda=\Lambda$).
\end{defi}

\begin{defi} A smooth vector field $X$ on a presymplectic (resp. symplectic, resp. Poisson) manifold
$(M,\omega)$ (resp. $(M,\Lambda)$) is said to be a \emph{presysmplectic} (resp. \emph{symplectic},
resp. \emph{Poisson}) vector field if ${\mathcal L}(X)\omega=0$ (resp. if ${\mathcal L}(X)\Lambda=0$),
where ${\mathcal L}(X)$ denotes the Lie derivative of forms or mutivector fields with respect to $X$.
\end{defi}

\begin{defi} Let $(M,\omega)$ be a presymplectic or symplectic manifold. A smooth vector field $X$ on 
$M$ is said to be \emph{Hamiltonian} if there exists a smooth function $H:M\to\RR$, called a
\emph{Hamiltonian} for $X$, such that
 $$\mathi(X)\omega=-\d H\,.
 $$
\end{defi}
Not any smooth function on a presymplectic manifold can be a Hamiltonian.

\begin{defi} Let $(M,\Lambda)$ be a Poisson manifold. A smooth vector field $X$ on $M$ is said to be
\emph{Hamiltonian} if there exists a smooth function $H\in C^\infty(M,\RR)$, 
called a \emph{Hamiltonian} for $X$,
such that $X=\Lambda^\sharp(\d H)$. An equivalent definition is that 
 $$\mathi(X)\d g=\{H,g\}\quad\hbox{for any}\ g\in C^\infty(M,\RR)\,,
 $$
where $\{H,g\}=\Lambda(\d H,\d g)$ denotes the Poisson bracket of the functions $H$ and $g$.
\end{defi}

On a symplectic or a Poisson manifold, any smooth function can be a Hamiltonian.

\begin{prop}\label{HamiltonImpliesSymplectic} 
A Hamiltonian vector field on a presymplectic (resp. symplectic, resp. Poisson) manifold
automatically is a presymplectic (resp. symplectic, resp. Poisson) vector field.
\end{prop}

The proof of this result, which is easy, can be found in any book on symplectic and Poisson geoemetry, for example
\cite{LibermannMarle1987}, \cite{LaurentPichereauVanhaecke} or \cite{OrtegaRatiu2004}.  

\subsection{Lie algebras and Lie groups actions}

\begin{defis}\label{DefiActionLieGroupLieAlgebra}
An \emph{action on the left} (resp. an \emph{action on the right})
of a Lie group $G$ on a smooth manifold $M$ is a smooth map $\Phi:G\times M\to M$ (resp. a smooth map
$\Psi:M\times G\to M$) such that
 
\begin{itemize}

\item{}
for each fixed $g\in G$, the map $\Phi_g:M\to M$ defined by
$\Phi_g(x)=\Phi(g,x)$ (resp. the map $\Psi_g:M\to M$ defined by $\Psi_g(x)=\Psi(x,g)$) is a 
smooth diffeomorphism of $M$,

\item{}
$\Phi_e=\id_M$ (resp. $\Psi_e=\id_M$), $e$ being the neutral element of $G$,

\item{}
for each pair $(g_1,g_2)\in G\times G$, $\Phi_{g_1}\circ\Phi_{g_2}=\Phi_{g_1g_2}$ 
(resp. $\Psi_{g_1}\circ\Psi_{g_2}=\Psi_{g_2g_1}$).
\end{itemize}
\par\smallskip

An \emph{action} of a Lie algebra $\mathcal G$ on a smooth manifold $M$ is a 
\emph{Lie algebras morphism} of $\mathcal G$ into the Lie algebra $A^1(M)$ of smooth vector fields on $M$,
\emph{i.e.} a linear map
$\psi:{\mathcal G}\to A^1(M)$ which associates to each $X\in{\mathcal G}$ a smooth vector field $\psi(X)$
on $M$ such that for each pair $(X,Y)\in{\mathcal G}\times{\mathcal G}$,
$\psi\bigl([X,Y]\bigr)=\bigl[\psi(X),\psi(Y)\bigr]$.
\end{defis}

\begin{prop}\label{LieAlgebraActionDeterminedByLieGroupAction}
An action $\Psi$, either on the left or on the right, of a Lie group $G$ on a smooth manifold 
$M$, automatically determines an action $\psi$ of its Lie algebra $\mathcal G$ on that manifold, which 
associates to each $X\in{\mathcal G}$ the vector field $\psi(X)$ on $M$, often denoted by $X_M$ and
called the fundamental vector field on $M$ associated to $X$. It is defined by
 $$\psi(X)(x)=X_M(x)=\frac{\d}{\d s}\bigl(\Psi_{\exp(sX)}(x)\bigr)\bigm|_{s=0}\,,\quad x\in M\,,
 $$
with the following convention: $\psi$ is a Lie algebras homomorphism when we take for Lie algebra 
$\mathcal G$ of the Lie group $G$ the Lie algebra or \emph{right invariant} vector fields on $G$ if $\Psi$
is an action on the left, and the Lie algebra of \emph{left invariant} vector fields on $G$ if $\Psi$
is an action on the right. 
\end{prop}

\begin{proof}
If $\Psi$ is an action of $G$ on $M$ on the left (respectively, on the right), 
the vector field on $G$ which is right invariant (respectively, left invariant) and whose value
at $e$ is $X$, and the associated fundamental vector field $X_M$ on $M$, are compatible by the map 
$g\mapsto\Psi_g(x)$. Therefore the map $\psi:{\mathcal G}\to A^1(M)$ is a Lie algebras homomorphism, 
if we take for definition of the bracket on $\mathcal G$ the bracket of right invariant 
(respectively, left invariant) vector fields on $G$.
\end{proof}

\begin{defis}\label{PresymplecticOrPoissonAction}
When $M$ is a presymplectic (or a symplectic, or a Poisson) manifold, an action 
$\Psi$ of a Lie group $G$ (respectively, an action $\psi$ of a Lie algebra 
$\mathcal G$) on the manifold $M$ 
is called a presymplectic (or a symplectic, or a Poisson) 
action if for each $g\in G$, $\Psi_g$ is a presymplectic, or a symplectic, or a Poisson 
diffeomorphism of $M$ (respectively, if for each $X\in{\mathcal G}$, $\psi(X)$ is a presymplectic,
or a symplectic, or a Poisson
vector field on $M$.
\end{defis} 

\begin{defis}\label{DefisHamiltonianAction} 
An action $\psi$ of a Lie algeba $\mathcal G$ on a presymplectic or symplectic manifold
$(M,\omega)$, or on a Poisson manifold $(M,\Lambda)$, is said to be \emph{Hamiltonian} if for each 
$X\in{\mathcal G}$, the vector field $\psi(X)$ on $M$ is Hamiltonian.
\par\smallskip  

An action $\Psi$ (either on the left or on the right) of a Lie group $G$ on a presymplectic or symplectic manifold $(M,\omega)$, or on a Poisson manifold $(M,\Lambda)$, is said to be \emph{Hamiltonian} if that action is presymplectic, or symplectic, or Poisson (according to the structure of $M$), and if in addition the associated action of the Lie algebra $\mathcal G$ of $G$ is Hamiltonian.
\end{defis}   

\begin{rmk} A Hamiltonian action of a Lie group, or of a Lie algebra, on a presymplectic, symplectic or Poisson manifold,
is automatically a presymplectic, symplectic or Poisson action. This result immediately follows from 
\ref{HamiltonImpliesSymplectic}
\end{rmk}

\subsection{Momentum maps of Hamiltonian actions}

\begin{prop}\label{ExistenceMomentum} 
Let $\psi$ be a Hamiltonian action of a finite-dimensional Lie algebra $\mathcal G$
on a presymplectic, symplectic or Poisson manifold $(M,\omega)$ or $(M,\Lambda)$. There exists 
a smooth map $J:M\to{\mathcal G}^*$, taking its values in the dual space ${\mathcal G}^*$ of the Lie algebra
$\mathcal G$, such that for each $X\in {\mathcal G}$ the Hamiltonian vector field $\psi(X)$ on $M$ admits 
as Hamiltonian the function $J_X:M\to \RR$, defined by
 $$J_X(x)=\bigl\langle J(x),X\bigr\rangle\,,\quad x\in M\,.
 $$
The map $J$ is called a \emph{momentum map} for the Lie algebra action $\psi$. When $\psi$ is the action of the Lie algebra $\mathcal G$ of a Lie group $G$ associated to a Hamiltonian action $\Psi$ of a Lie group
$G$, $J$ is called a \emph{momentum map} for the Hamiltonian Lie group action $\Psi$.  
\end{prop}

The proof of that result, which is easy, can be found for example in 
\cite{LibermannMarle1987}, \cite{LaurentPichereauVanhaecke} or \cite{OrtegaRatiu2004}.  

\begin{rmk}\label{NonUnicityMomentum}
The momentum map $J$ is not unique: 
\begin{itemize}

\item{}
when $(M,\omega)$ is a connected symplectic manifold, 
$J$ is determined up to addition of an arbitrary constant element in ${\mathcal G}^*$; 

\item{}
when $(M,\Lambda)$ is a connected Poisson manifold, the momentum map $J$ is determined up to addition of an arbitrary ${\mathcal G}^*$-valued smooth map which, coupled with any $X\in{\mathcal G}$, yields a Casimir of the 
Poisson algebra of $(M,\Lambda)$, \emph{i.e.} a smooth function on $M$ 
whose Poisson bracket with any other smooth function on that manifold
is the function identically equal to $0$.
\end{itemize}
\end{rmk}

\subsection{Noether's theorem in Hamiltonian formalism}

\begin{theo}[Noether's theorem in Hamiltonian formalism]\label{NoetherHamilton}
Let $X_f$ and $X_g$ be two Hamiltonian vector fields on a presymplectic or symplectic manifold
$(M,\omega)$, or on a Poisson manifold $(M,\Lambda)$, which admit as Hamiltonians, respectively, the 
smooth functions $f$ and $g$ on the manifold $M$. The function $f$ remains constant on 
each integral curve of $X_g$ if and only if $g$ remains constant on each integral curve of $X_f$.
\end{theo}

\begin{proof}
The function $f$ is constant on each integral curve of $X_g$ if and only if $\mathi(X_g)\d f=0$, since each
integral curve of $X_g$ is connected. We can use the Poisson bracket, even when $M$ is a presymplectic manifold,
since the Poisson bracket of two Hamiltonians on a presymplectic manifold still can be defined. So we can write
 $$\mathi(X_g)\d f=\{g,f\}=-\{f,g\}=-\mathi(X_f)\d g\,.\eqno{\qedhere}
 $$
\end{proof}

\begin{coro}[of Noether's theorem in Hamiltonian formalism]\label{CoroNoetherHamilton}
Let $\psi:{\mathcal G}\to A^1(M)$ be a
Hamiltonian action of a finite-dimensional Lie algebra $\mathcal G$ on a presymplectic 
or symplectic manifold $(M,\omega)$, or on a Poisson manifold $(M,\Lambda)$, and 
let $J:M\to{\mathcal G}^*$ be a momentum map of this action. Let $X_H$ be a Hamiltonian vector field
on $M$ admitting as Hamiltonian a smooth function $H$. If for each $X\in{\mathcal G}$ we have
$\mathi\bigl(\psi(X)\bigr)(\d H)=0$, the momentum map $J$ remains constant on each integral curve of $X_H$. 
\end{coro}

\begin{proof}
This result is obtained by applying \ref{NoetherHamilton} to the pairs of Hamiltonian vector fields made
by $X_H$ and each vector field associated to an element of a basis of $\mathcal G$. 
\end{proof}

\subsection{Symplectic cocycles}

\begin{theo}[J.M.~Souriau]\label{SouriauEquivarianceTheorem}
Let $\Phi$ be a Hamiltonian action (either on the left or on the right)  of a Lie group $G$ on 
a connected symplectic manifold $(M,\omega)$ and let $J:M\to{\mathcal G}^*$ be 
a momentum map of this action. 
There exists an affine action $A$ (either on the left or on the right) of the Lie group $G$ on the dual 
${\mathcal G}^*$ of its Lie algebra $\mathcal G$  such that the momentum map $J$ is equivariant 
with respect to the actions $\Phi$ of $G$ on $M$ and $A$ of $G$ on ${\mathcal G}^*$, \emph{i.e.} such that
 $$J\circ\Phi_g(x)=A_g\circ J(x)\quad\hbox{for all } g\in G\,,\ x\in M\,.
 $$
The action $A$ can be written, with $g\in G$ and $\xi\in{\mathcal G}^*$, 
 \begin{equation*}
 \begin{cases}
        A(g,\xi)=\Ad^*_{g^{-1}}(\xi)+\theta(g)& \text{if $\Phi$ is an action on the left,}\\
        A(\xi,g)=\Ad^*_{g}(\xi) -\theta(g^{-1})& \text{if $\Phi$ is an action on the right.}
 \end{cases}
 \end{equation*}
\end{theo}

\begin{proof}
Let us assume that $\Phi$ is an action on the left. The fundamental vector field $X_M$ associated to
each $X\in{\mathcal G}$ is Hamiltonian, with the function
$J_X:M\to \RR$, given by
 $$J_X(x)=\bigl\langle J(x),X\bigr\rangle\,,\quad x\in M\,,
 $$
as Hamiltonian. For each $g\in G$ the direct image $(\Phi_{g^{-1}})_*(X_M)$ of $X_M$ by the symplectic diffeomerphism
$\Phi_{g^{-1}}$ is Hamiltonian, with $J_X\circ\Phi_g$ as Hamiltonian. An easy calculation
shows that this vector field is the fundamental vector field associated to $\Ad_{g^{-1}}(X)\in{\mathcal G}$. 
The function
 $$x\mapsto \bigl\langle J(x), \Ad_{g^{-1}}(X)\bigr\rangle 
           =\bigl\langle \Ad^*_{g^{-1}}\circ J(x),X\bigr\rangle
 $$
is therefore a Hamiltonian for that vector field. These two functions defined on the connected manifold $M$, which both are admissible Hamiltonians for the same Hamiltonian vector field, differ only by a constant (which may depend on $g\in G$).
We can set, for any $g\in G$,
 $$\theta(g)=J\circ\Phi_g(x)
                     -\Ad^*_{g^{-1}}\circ J(x)
 $$
and check that the map $A:G\times{\mathcal G}^*\to{\mathcal G}^*$ defined in the statement 
is indeed an action for which $J$ is equivariant.
\par\smallskip

A similar proof, with some changes of signs, holds when $\Phi$ is an action on the right.
\end{proof}

\begin{prop}\label{CocycleProperty}
Under the assumptions and with the notations of \ref{SouriauEquivarianceTheorem}, the map $\theta:G\to{\mathcal G}^*$
is a cocycle of the Lie group $G$ with values in ${\mathcal G}^*$, for the coadjoint representation.
It means that is satisfies, for all $g$ and $h\in G$,
 $$\theta(gh)=\theta(g)+\Ad^*_{g^{-1}}\bigl(\theta(h)\bigr)\,.
 $$
More precisely $\theta$ is a symplectic cocycle. It means that its differential 
$T_e\theta:T_eG\equiv{\mathcal G}\to{\mathcal G}^*$ at the neutral element $e\in G$ 
can be considered as a skew-symmetric bilinear form on $\mathcal G$:
 $$\Theta(X,Y)=\bigl\langle T_e\theta(X),Y\bigr\rangle=-\bigl\langle T_e\theta(Y),X\bigr\rangle\,.
 $$
The skew-symmetric bilinear form $\Theta$ is a symplectic cocycle of the Lie algebra $\mathcal G$. It means that it is skew-symmetric and satisfies, for all $X$, $Y$ and $Z\in{\mathcal G}$,
 $$\Theta\bigl([X,Y],Z\bigr) + \Theta\bigl([Y,Z],X\bigr) + \Theta\bigl([Z,X],Y\bigr)=0\,.
 $$
\end{prop} 

\begin{proof}
These properties easily follow from the fact that when $\Phi$ is an action on the left, for $g$ and $h\in G$,
$\Phi_g\circ\Phi_h=\Phi_{gh}$ (and a similar equality when $\Phi$ is an action on the right).
The interested reader will find more details in \cite{LibermannMarle1987}, \cite{Souriau1969} or \cite{Marle2014}.
\end{proof} 

\begin{prop}\label{MomentumPoisson}
Still under the assumptions and with the notations of \ref{SouriauEquivarianceTheorem},
the composition law which associates to each pair
$(f,g)$ of smooth real-valued functions on ${\mathcal G}^*$ the function $\{f,g\}_\Theta$ given by
 $$\{f,g\}_{\Theta}(x)=\bigl\langle x,[\d f(x),\d g(x)]\bigr\rangle
  -{\Theta}\bigl(\d f(x),\d g(x)\bigr)\,,\quad x\in{\mathcal G}^*\,,
 $$
($\mathcal G$ being identified with its bidual ${\mathcal G}^{**}$),
determines a Poisson structure on ${\mathcal G}^*$, and the momentum map
$J:M\to{\mathcal G}^*$ is a Poisson map, $M$ being endowed with the Poisson structure associated to its
symplectic structure.
\end{prop}

\begin{proof}
The fact that the bracket $(f,g)\mapsto\{f,g\}_\Theta$ on $C^\infty({\mathcal G}^*,\RR)$ is a
Poisson bracket was already indicated in \ref{ExamplesPoissonManifolds}. It can be verified by easy
calculations. The fact that $J$ is a Poisson map can be proven by first looking at linear functions on 
${\mathcal G}^*$, \emph{i.e.} elements in $\mathcal G$. The reader will find a detailed proof in 
\cite{Marle2014}.  
\end{proof}

\begin{rmk}
When the momentum map $J$ is replaced by another momentum map $J_1=J+\mu$, where $\mu\in{\mathcal G}^*$ is a constant, the symplectic Lie group cocycle $\theta$ and the symplectic Lie algebra cocycle $\Theta$ are 
replaced by $\theta_1$ and $\Theta_1$, respectively, given by
 \begin{align*}
  \theta_1(g)&=\theta(g)+\mu -\Ad^*_{g^{-1}}(\mu)\,,\quad g\in G\,,\\ 
   \Theta_1(X,Y)&=\Theta(X,Y)+\bigl\langle\mu,[X,Y]\bigr\rangle\,,\quad X\ \hbox{and}\ 
   Y\in{\mathcal G}\,.
 \end{align*}
These formulae show that $\theta_1-\theta$ and $\Theta_1-\Theta$ are \emph{symplectic coboudaries} of the Lie
group $G$ and the Lie algebra $\mathcal G$. In other words, the \emph{cohomology classes} of the cocycles
$\theta$ and $\Theta$ only depend on the Hamiltonian action $\Phi$ of $G$ on the symplectic manifold
$(M,\omega)$. 
\end{rmk}

\subsection{The use of symmetries in Hamiltonian Mechanics}\label{SymmetriesInHamiltonianMechanics}

\subsubsection{Symmetries of the phase space}\label{SymmetriesPhaseSpace}
Hamiltonian Symmetries are often used for the search of solutions of the equations of motion
of mechanical systems. The symmetries considered are those of the \emph{phase space} of the mechanical
system. This space is very often a \emph{symplectic manifold}, 
either the cotangent bundle to the configuration 
space with its canonical symplectic structure, or a more general symplectic manifold.
Sometimes, after some simplifications, the phase space is a \emph{Poisson manifold}. 
\par\smallskip

The \emph{Marsden-Weinstein reduction procedure} \cite{MarsdenWeinstein74, Meyer73}
or one of its generalizations \cite{OrtegaRatiu2004} is the method most often used to
facilitate the determination of solutions of the equations of motion. 
In a first step, a possible value of the momentum map is chosen and the subset 
of the phase space on which the momentum map takes this value is determined.
In a second step, that subset (when it is a smooth manifold) is quotiented 
by its isotropic foliation. The quotient manifold is a symplectic manifold of 
a dimension smaller than that of the original phase space, and one has an easier 
to  solve Hamiltonian system on that reduced phase space. 
\par\smallskip

When Hamiltonian symmetries are used  for the reduction of the dimension of the phase space of a mechanical system, the symplectic cocycle of the Lie group of symmetries action, or of the Lie algebra 
of symmetries action, is almost always the \emph{zero cocycle}.
\par\smallskip

For example, if the goup of symmetries is the canonical lift to the cotangent bundle of a group of symmetries of the configuration space, not only the 
\emph{canonical symplectic form}, but the \emph{Liouville $1$-form} 
of the cotangent bundle itself
remains invariant under the action of the symmetry group, and this fact implies that the symplectic
cohomology class of the action is zero.

\subsubsection{Symmetries of the space of motions}\label{SymmetriesSpaceOfMotions}
A completely different way of using symmetries was initiated by Jean-Marie Souriau, who proposed
to consider the symmetries of the \emph{manifold of motions} of the mechanical system.
He observed that the Lagrangian and Hamiltonian formalisms, 
in their usual formulations, \emph{involve the choice of a particular reference  frame}, in which the motion is described. This choice
\emph{destroys a part of the natural symmetries of the system}.
\par\smallskip

For example, in classical (non-relativistic) Mechanics, the natural symmetry group of an isolated
mechanical system must contain the symmetry group of the \emph{Galilean space-time}, called the
\emph{Galilean group}. This group is of \emph{dimension 10}. It contains not only the 
\emph{group of Euclidean displacements of space} which is of \emph{dimension 6} and the 
\emph{group of time translations} which is of \emph{dimension 1}, but the 
\emph{group of linear changes of Galilean reference frames} which is of \emph{dimension 3}.
\par\smallskip

If we use the \emph{Lagrangian  formalism} or the \emph{Hamiltonian formalism}, 
the Lagrangian or the Hamiltonian of the system 
\emph{depends on the reference frame}: \emph{it is not invariant 
with respect to linear changes of Galilean reference frames}.
\par\smallskip

It may seem strange to consider the set of all possible motions of
a system, {which is unknown} as long as we have not determined all these possible 
motions. One may ask if it is really useful when we want to determine not all 
possible motions, but only one motion with prescribed initial data, since that 
motion is just one point of the (unknown) manifold of motion!
\par\smallskip

Souriau's answers to this objection are the following.

\par\smallskip\noindent
{\rm 1.\quad} We know that the manifold of motions has a 
\emph{symplectic structure}, and  very often many things are 
known about its \emph{symmetry properties}.

\par\smallskip\noindent
{\rm 2.\quad} In classical (non-relativistic) mechanics, 
there exists a natural mathematical object which \emph{does not depend 
on the choice of a particular reference frame} 
(even if the decriptions given to that object by different 
observers depend on the reference frame used by these observers):
it is the \emph{evolution space} of the
system. 
\par\smallskip

The knowledge of the equations which govern the system's 
evolution allows the full mathematical description of the 
\emph{evolution space}, even when these equations are not yet solved.
\par\smallskip

Moreover, the symmetry properties of the \emph{evolution space} 
are the same as those of the manifold of motions.
\par\smallskip

For example, the \emph{evolution space} of a classical mechanical system 
with configuration manifold $N$ is

\begin{enumerate}

\item{} in the Lagrangian formalism, the space $\RR\times TN$ endowed with the presymplectic
form $\d\widehat\varpi_L$, whose kernel is of dimension $1$ when the Lagrangian $L$ is hyper-regular,

\item{} in the Hamiltonian formalism, the space $\RR\times T^*N$ with the presymplectic form
$\d\widehat\varpi_H$, whose kernel too is of dimension $1$.
\end{enumerate}

The Poincaré-Cartan $1$-form $\widehat\varpi_L$ in the Lagrangian formalism, or $\widehat\varpi_H$ in the 
Hamiltonian formalism, depends on the choice of a particular reference frame, made for using the Lagrangian or the Hamiltonian formalism. 
But their exterior differentials, the presymplectic forms
$\d\widehat\varpi_L$ or $\d\widehat\varpi_H$, \emph{do not depend on that choice}, modulo a simple change of variables in the evolution space. 
\par\smallskip

Souriau defined this presymplectic form in a framework 
more general than those of Lagrangian or Hamiltonian formalisms, and called it the
\emph{Lagrange form}. In this more general setting, it may not be an exact $2$-form. Souriau proposed as a new 
\emph{Principle}, the assumption that it always projects on the space of motions
of the systems as a 
\emph{symplectic form}, even in Relativistic Mechanics 
in which the definition of an evolution space is not clear. He called
this new principle the \emph{Maxwell Principle}.
\par\smallskip

\emph{V.~Bargmann} proved that the symplectic cohomology of the Galilean group is of dimension $1$, and 
Souriau proved that the cohomology class of its action on the manifold of motions of
an isolated classical (non-relativistic) mechanical system can be identified with the
\emph{total mass} of the system \cite{Souriau1969}, chapter III, page 153.
\par\smallskip

Readers interested in the Galilean group and momentum maps of its actions are referred to the recent book by G.~de~Saxcé and C.~Vallée \cite{deSaxceVallee2016}.

\section{Statistical Mechanics and Thermodynamics}\label{StatisticalMechanicsThermodynamics} 

\subsection{Basic concepts in Statistical Mechanics}\label{BasicConceptsInStatisticalMechanics}

During the XVIII--th and XIX--th centuries, the idea that material bodies (fluids as well as solids)
are assemblies of a very large number of small, moving particles, began to be considered by some scientists,
notably Daniel Bernoulli (1700--1782), Rudolf Clausius (1822--1888), James Clerk Maxwell 
(1831--1879) and Ludwig Eduardo Boltzmann (1844--1906), 
as a reasonable possibility. Attemps were made to explain the nature of some measurable macroscopic quantities (for example the temperature of a material body, the pressure exerted by a gas on the walls of the vessel in which it is contained), and the laws which govern the variations of these macroscopic quantities, by application of the laws of Classical Mechanics to the motions of these very small particles. Described in
the framework of the Hamiltonian formalism, the material body is considered as a Hamiltonian system whose phase space is a very high dimensional symplectic manifold $(M,\omega)$, since an element of that space gives a perfect information about the positions and the velocities of all the particles of the system. The experimental determination of the exact state of the system being impossible, one only can use the 
probability of presence, at each instant, of the state of the system in various parts of the phase space.
Scientists introduced the concept of a \emph{statistical state}, defined below.

\begin{defi}\label{DefiStatisticalState} 
Let $(M,\omega)$ be a symplectic manifold. A \emph{statistical state} is a probability measure
$\mu$ on the manifold $M$.
\end{defi}   

\subsubsection{The Liouville measure on a symplectic manifold}\label{LiouvilleMeasure}
On each symplectic manifold $(M,\omega)$, with $\dim M=2n$, there exists a positive measure 
$\lambda_\omega$, called the \emph{Liouville measure}. Let us briefly recall its definition.
Let $(U,\varphi)$ be a Darboux chart of $(M,\omega)$
(\ref{HamiltonianFormalismOnSymplecticManifolds}). 
The open subset $U$ of $M$ is, by means of the diffeomorphism $\varphi$, 
identified with an open subset $\varphi(U)$ of $\RR^{2n}$ on which the coordinates 
(Darboux coordinates) will be denoted by $(p_1,\ldots,p_n,x^1,\ldots,x^n)$. With 
this identification, the Liouville measure (restricted to $U$) is simply the 
Lebesgue measure on the open subset $\varphi(U)$ of $\RR^{2n}$. In other words, 
for each Borel subset $A$ of $M$ contained in $U$, we have
 $$\lambda_\omega(A)=\int_{\varphi(A)}\d p_1\ldots\,\d p_n\,\d x^1\ldots\,\d x^n\,.
 $$   
One can easily check that this definition does not depend on the choice of the Darboux coordinates
$(p_1,\ldots,p_n,x^1,\ldots, x^n)$ on $\varphi(A)$. By using an atlas of Darboux charts on $(M,\omega)$,
one can easily define $\lambda_\omega(A)$ for any Borel subset $A$ of $M$.

\begin{defi}\label{DefiContinuousOrSmoothStatisticalState} 
A statistical state $\mu$ on the syplectic manifold $(M,\omega)$ is said to be \emph{continuous}
(respectively, is said to be \emph{smooth}) if it has a continuous (respectively, a smooth) density with respect to the Liouville measure $\lambda_\omega$, \emph{i.e.} if there exists a continuous function 
(respectively, a smooth function) $\rho:M\to \RR$ such that, for each Borel subset $A$ of $M$
 $$\mu(A)=\int_A\rho\d \lambda_\omega\,.
 $$  
\end{defi}

\begin{rmk}\label{RemarkTotalProbaIs1}
The density $\rho$ of a continuous statistical state on $(M,\omega)$ takes its values in
$\RR^+$ and of course satisfies
 $$\int_M\rho\d\lambda_\omega=1\,.
 $$
\end{rmk}

For simplicity we only consider in what follows continuous, very often even smooth statistical states.

\subsubsection{Variation in time of a statistical state}\label{TimeVariationOfStatisticalState}
Let $H$ be a smooth time independent Hamiltonian on a symplectic manifold $(M,\omega)$, $X_H$ 
the associated Hamiltonian vector field and $\Phi^{X_H}$ its reduced flow. We consider the mechanical system whose time evolution is described by the flow of $X_H$. 
\par\smallskip

If the state of the system at time
$t_0$, assumed to be perfectly known, is a point $z_0\in M$, 
its state at time $t_1$ is the point $z_1=\Phi^{X_H}_{t_1-t_0}(z_0)$.
\par\smallskip

Let us now assume that the state of the system at time $t_0$ is not perfectly known, but that a
continuous probability measure on the phase space $M$, whose density with respect to the Liouville measure 
$\lambda_\omega$ is $\rho_0$, describes the probability distribution of presence of the state of the system
at time $t_0$. In other words, $\rho_0$ is the density of the statistical state of the system at time
$t_0$. For any other time $t_1$, the map $\Phi^{X_H}_{t_1-t_0}$ is a symplectomorphism, therefore leaves
invariant the Liouville measure $\lambda_\omega$. The probability density $\rho_1$ of the statistical state 
of the system at time $t_1$ therefore satisfies, for any $z_0\in M$ for which
$x_1=\Phi^{X_H}_{t_1-t_0}(x_0)$ is defined,
 $$\rho_1(x_1)=\rho_1\bigl(\Phi^{X_H}_{t_1-t_0}(x_0)\bigr)=\rho_0(x_0)\,.
 $$ 
Since $\bigl(\Phi^{X_H}_{t_1-t_0}\bigr)^{-1}=\Phi^{X_H}_{t_0-t_1}$, we can write
 $$\rho_1=\rho_0\circ\Phi^{X_H}_{t_0-t_1}\,.
 $$

\begin{defi}\label{DefiEntropyOfStatisticalState}
Let $\rho$ be the density of a continuous statistical state $\mu$ on the symplectic manifold 
$(M,\omega)$. The number
 $$s(\rho)=\int_M\rho\log\left(\frac{1}{\rho}\right)\d\lambda_\omega
 $$
is called the \emph{entropy} of the statistical state $\mu$ or, with a slight abuse of language, 
the \emph{entropy} of the density $\rho$.
\end{defi}

\begin{rmks}\hfill
\label{ConventionsForEntropyDefinition}

\par\smallskip\noindent 
{\rm 1.\quad}By convention we state that $0\log 0=0$. With that convention 
the function $x\mapsto x\log x$ is continuous on $\RR^+$.
If the integral on the right hand side of the equality which defines
$s(\rho)$ does not converge, we state that $s(\rho)=-\infty$. With these conventions, 
$s(\rho)$ exists for any continuous probability density $\rho$.

\par\smallskip\noindent
{\rm 2.\quad} The above definition (\ref{DefiEntropyOfStatisticalState}) of the entropy 
of a statistical state, founded on ideas developed by Boltzmann in his Kinetic Theory
of Gases \cite{Boltzmann}, specially in the derivation of his famous (and controversed)
\emph{Theorem \^Eta}, is too related with the ideas of Claude Shannon \cite{Shannon1948}
on Information theorey. The use of Information theory in Thermodynamics was more recently 
proposed by Jaynes \cite{Jaynes1957a, Jaynes1957b} and Mackey
\cite{Mackey1963}. For a very nice discussion of the use of probability concepts in Physics and application of Information theory in Quantum Mechanics, the reader is referred to
the paper by R. Balian \cite{Balian2005}.
\end{rmks}

The entropy $s(\rho)$ of a probability density $\rho$ has very remarkable variational properties discussed in the following definitions and proposition.

\begin{defis}\label{DefisMeanValueVariationsStationarity} 
Let $\rho$ be the density of a smooth statistical state on a symplectic manifold
$(M,\omega)$.
\par\smallskip

{\rm 1.\quad} For each function $f$ defined on $M$, taking its values in $\RR$ or in some finite-dimensional vector space, such that the integral on the right hand side of the
equality
 $${\mathcal E}_\rho(f)= \int_Mf\rho\d\lambda_\omega
 $$
converges, the value ${\mathcal E}_\rho(f)$ of that integral is called the \emph{mean value
of $f$ with respect to $\rho$}.
\par\smallskip

{\rm 2.\quad} Let $f$ be a smooth function on $M$, taking its values in $\RR$ or in some finite-dimensional vector space, satisfying the properties stated above. A 
\emph{ smooth infinitesimal variation of $\rho$ with fixed mean value of $f$}
is a smooth map, defined on the product $]-\varepsilon,\varepsilon[\,\times M$, with values in $\RR^+$, 
where $\varepsilon>0$,
 $$(\tau,z)\mapsto\rho(\tau,z)\,,\quad \tau\in]-\varepsilon,\varepsilon[,\ z\in M\,,
 $$ 
such that 
\begin{itemize}

\item{} for $\tau=0$ and any $z\in M$, $\rho(0,z)=\rho(z)$,

\item{} for each $\tau\in]-\varepsilon,\varepsilon[$\,, $z\mapsto\rho_\tau(z)=\rho(\tau,z)$ is a smooth probability density on $M$ such that
 $${\mathcal E}_{\rho_\tau}(f)= \int_M\rho_\tau f\d\lambda_\omega={\mathcal E}_{\rho}(f)\,.
 $$
\end{itemize}

{\rm 3.\quad} The entropy function $s$ is said to be \emph{stationary} at the probability density $\rho$
with respect to smooth infinitesimal variations of $\rho$ with fixed mean value of $f$, if
for any smooth infinitesimal variation $(\tau, z)\mapsto\rho(\tau,z)$ of $\rho$ with fixed mean value of $f$
 $$\frac{\d s(\rho_\tau)}{\d\tau}\Bigm|_{\tau=0}=0\,.
 $$
\end{defis}

\begin{prop}\label{StationarityOfEntropy}
Let $H:M\to\RR$ be a  smooth Hamiltonian on a symplectic manifold $(M,\omega)$ and $\rho$
be the density of a smooth statistical state on $M$ such that the integral defining the mean value
${\mathcal E}_\rho(H)$ of $H$ with respect to $\rho$ converges. 
The entropy function $s$ is stationary at $\rho$
with respect to smooth infinitesimal variations of $\rho$ with fixed mean value of $H$, if and only if there exists a real $b\in\RR$ such that, for all $z\in M$,
 $$\rho(z)=\frac{1}{P(b)}\exp\bigl(-bH(z)\bigr)\,,\quad\hbox{with}\quad
      P(b)=\int_M\exp(-bH)\d\lambda_\omega\,.
 $$
\end{prop}

\begin{proof}
Let $\tau\mapsto\rho_\tau$ be a smooth infinitesimal variation of $\rho$ with fixed mean value of $H$.
Since $\displaystyle \int_M\rho_\tau\d\lambda_\omega$ and $\displaystyle\int_M\rho_\tau H\d\lambda_\omega$
do not depend on $\tau$, it satisfies, for all $\tau\in]-\varepsilon,\varepsilon[$\,,
 $$\int_M\frac{\partial\rho(\tau,z)}{\partial\tau}\d\lambda_\omega(z)=0\,,
   \int_M\frac{\partial\rho(\tau,z)}{\partial\tau}H(z)\d\lambda_\omega(z)=0\,.
 $$
Moreover an easy calculation leads to 
 $$\frac{\d s(\rho_\tau)}{\d \tau}\Bigm|_{\tau=0}=-\int_M
    \frac{\partial\rho(\tau,z)}{\partial\tau}\Bigm|_{\tau=0}
     (1+\log\bigl(\rho(z)\bigr)\d\lambda_\omega(z)\,.
 $$
A well known result in calculus of variations shows that
the entropy function $s$ is stationary at $\rho$
with respect to smooth infinitesimal variations of $\rho$ with fixed mean value of $H$, if and only if there exist two real constants $a$ and  $b$, called \emph{Lagrange multipliers}, such that, for all $z\in M$,
 $$1+\log(\rho)+a+bH=0\,,
 $$
which leads to
 $$\rho=\exp(-1-a-bH)\,.
 $$
By writing that $\displaystyle\int_M\rho\d\lambda_\omega=1$, we see that $a$ is determined by $b$:
 $$\exp(1+a)=P(b)=\int_M\exp(-bH)\d\lambda_\omega\,.\eqno{\qedhere}
 $$
\end{proof}

\begin{defis}\label{DefiGibbsStatePartitionFunction}
Let $H:M\to\RR$ be a  smooth Hamiltonian on a symplectic manifold $(M,\omega)$. For each $b\in \RR$ such that the integral on the right side of the equality
 $$P(b)=\int_M\exp(-bH)\d\lambda_\omega
 $$
converges, the smooth probability measure on $M$ with density (with respect to the Liouville measure)
 $$\rho(b)=\frac{1}{P(b)}\exp\bigl(-bH\bigr)
 $$
is called the \emph{Gibbs statistical state} associated to $b$. The function $P:b\mapsto P(b)$ is called the
\emph{partition function}.
\end{defis}

The following proposition shows that the entropy function, not only is stationary at any Gibbs statistical state, but in a certain sense attains at that state a strict maximum.

\begin{prop}\label{MaximalityEntropy}
Let $H:M\to\RR$ be a  smooth Hamiltonian on a symplectic manifold $(M,\omega)$ and $b\in \RR$ be such that 
the integral defining the value $P(b)$ of the partition function $P$ at $b$ converges. Let
 $$\rho_b=\frac{1}{P(b)}\exp(-bH)
 $$
be the probability density of the Gibbs statistical state associated to $b$. We assume that
the Hamiltonian $H$ is bounded by below, i.e. that there exists a constant $m$ such that 
$m\leq H(z)$ for any $z\in M$. Then the integral defining
 $${\mathcal E}_{\rho_b}(H)=\int_M\rho_b H\d\lambda_\omega
 $$
converges. For any other smooth probability density $\rho_1$ such that
 $${\mathcal E}_{\rho_1}(H)={\mathcal E}_{\rho_b}(H)\,,
 $$   
we have
 $$s(\rho_1)\leq s(\rho_b)\,,$$
and the equality $s(\rho_1)=s(\rho_b)$ holds if and only if $\rho_1=\rho_b$. 
\end{prop}

\begin{proof}
Since $m\leq H$, the function $\rho_b\exp(-bH)$  satisfies 
$0\leq \rho_b\exp(-bH)\leq \exp(-mb)\rho_b$, 
therefore is integrable on $M$. Let $\rho_1$ 
be any smooth probability density on $M$ satisfying
${\mathcal E}_{\rho_1}(H)={\mathcal E}_{\rho_b}(H)$.   
The function defined on $\RR^+$
 $$x\mapsto h(x)=\begin{cases}
       x\log\left({\displaystyle\frac{1}{x}}\right)& \text{ if $x>0$}\\
         0& \text{if $x=0$}\end{cases}
 $$
being convex, its graph is below the tangent at any of its points $\bigl(x_0, h(x_0)\bigr)$. We therefore have, for all $x>0$ and $x_0>0$,
 $$h(x)\leq h(x_0)-(1+\log x_0)(x-x_0)=x_0 -x(1+\log x_0)\,.
 $$  
With $x=\rho_1(z)$ and $x_0=\rho_b(z)$, $z$ being any element in
$M$, that inequality becomes
 $$h\bigl(\rho_1(z)\bigr)=\rho_1(z)\log\left(\frac{1}{\rho_1(z)}\right)\leq \rho_b(z)
    -\bigl(1+\log \rho_b(z)\bigr)\rho_1(z)\,.
 $$
By integration over $M$, using the fact that $\rho_b$ is the probability density of the Gibbs state associated to $b$, we obtain
 $$s(\rho_1)\leq 1 -1 -\int_M\rho_1\log \rho_b\d\lambda_\omega=s(\rho_b)\,.
 $$
We have proven the inequality $s(\rho_1)\leq s(\rho_b)$. If $\rho_1=\rho_b$, we have of course 
the equality $s(\rho_1)=s(\rho_b)$. Conversely if $s(\rho_1)=s(\rho_b)$, the functions
defined on $M$
 $$z\mapsto\varphi_1(z)=\rho_1(z)\log\left(\frac{1}{\rho_1(z)}\right)\quad\hbox{and}\quad
   z\mapsto\varphi(z)=\rho_b(z)-\bigl(1+\log \rho_b(z)\bigr)\rho_1(z)
 $$
are continuous on $M$ except, maybe, for $\varphi$, at points $z$ at which $\rho_b(z)=0$ and 
$\rho_1(z)\neq 0$, but the set of such points is of measure $0$ since $\varphi$ is integrable.
They satisfy the inequality $\varphi_1\leq \varphi$. Both are integrable on $M$ and have the same integral.
The function $\varphi-\varphi_1$ is everywhere $\geq 0$, 
is integrable on $M$ and its integral is 
$0$. That function is therefore everywhere equal to $0$ on $M$. 
We can write, for any $z\in M$,
 $$\rho_1(z)\log\left(\frac{1}{\rho_1(z)}\right)=\rho_b(z)-\bigl(1+\log \rho_b(z)\bigr)\rho_1(z)\,.\eqno{(*)}
 $$     
For each $z\in M$ such that $\rho_1(z)\neq 0$, we can divide that equality by $\rho_1(z)$. We obtain
 $$\frac{\rho_b(z)}{\rho_1(z)} - \log\left(\frac{\rho_b(z)}{\rho_1(z)}\right)=1\,.$$
Since the function $x\mapsto x - \log x$ reaches its minimum, equal to $1$, for a unique value of $x>0$,
that value being $1$, we see that for each $z\in M$ at which $\rho_1(z)>0$, we have $\rho_1(z)=\rho_b(z)$.
At points $z\in M$ at which $\rho_1(z)=0$, the above equality $(*)$ shows that $\rho_b(z)=0$. Therefore 
$\rho_1=\rho_b$.   
\end{proof}

The following proposition shows that a Gibbs statistical state remains invariant under the flow of the Hamiltonian vector field $X_H$. In that sense, one can say that a Gibbs statistical state is a
statistical equilibrium state.

\begin{prop}\label{InvarianceOfGibbsState}   
Let $H$ be a smooth Hamiltonian bounded by below on a symplectic manifold $(M,\omega)$, $b\in \RR$ be
such that the integral defining the value $P(b)$ of the partition function $P$ at $b$ converges. The 
Gibbs state associated to $b$ remains invariant under the flow of of the Hamiltonian vector field $X_H$.  
\end{prop}

\begin{proof}
The density $\rho_b$ of the Gibbs state associated to $b$, with respect to the Liouville measure 
$\lambda_\omega$, is
 $$\rho_b= \frac{1}{P(b)}\exp(-bH)\,.
 $$
Since $H$ is constant along each integral curve of $X_H$, $\rho_b$ too is constant along each integral 
curve of $X_H$. Moreover, the Liouville measure $\lambda_\omega$ remains invariant under the flow of $X_H$.
Therefore the Gibbs probability measure associated to $b$ too remains invariant under that flow.
\end{proof}

\subsection{Thermodynamic equilibria and thermodynamic functions}\label{ThermodynamicEquilibria}

\subsubsection{Assumptions made in this section.\quad}\nobreak
\label{AssumptionsInThermodynamicEquilibria} 
Any Hamiltonian $H$ defined on a symplectic manifold
$(M,\omega)$ considered in this section will be assumed to be smooth, bounded by below and such that for
any real $b>0$, each one of the three functions, defined on $M$,
$z\mapsto \exp\bigl(-bH(z)\bigr)$, $z\mapsto \big| H(z)\big|\exp\bigl(-bH(z)\bigr)$ and
$z\mapsto \bigl(H(z)\bigr)^2\exp\bigl(-bH(z)\bigr)$
is everywhere smaller than some function defined on $M$ integrable with respect to the Liouville measure
$\lambda_\omega$. The integrals which define
 $$P(b)=\int_M\exp(-bH)\d\lambda_\omega\quad \hbox{and}\quad 
        {\mathcal E}_{\rho_b}(H)=\int_M H\exp(-bH)\d \lambda_\omega
 $$
therefore converge. 

\begin{prop}\label{PropertiesOfThermodynamicFunctions}
Let $H$ be a Hamiltonian defined on a symplectic manifold $(M,\omega)$ satisfying the assumptions
indicated in \ref{AssumptionsInThermodynamicEquilibria}. For any real $b>0$ let
 $$ P(b)=\int_M\exp(-bH)\d\lambda_\omega\quad\hbox{and}\quad \rho_b=\frac{1}{P(b)}\exp(-bH)
 $$
be the value at $b$ of the partition function $P$ and the probability density of the Gibbs statistical state associated to $b$, and
 $$E(b)={\mathcal E}_{\rho_b}(H)=\frac{1}{P(b)}\int_MH\exp(-bH)\d\lambda_\omega
 $$
be the mean value of $H$ with respect to the probability density $\rho_b$. The first and second derivatives
with respect to $b$ of the partition function $P$  exist, are continuous functions of $b$ given by
 $$\frac{\d P(b)}{\d b}=-P(b)E(b)\,,\quad 
    \frac{\d^2 P(b)}{\d b^2}=\int_M H^2\exp(-bH)\d\lambda_\omega= P(b){\mathcal E}_{\rho_b}(H^2)\,.
 $$
The derivative with respect to $b$ of the function $E$ exists and is a continuous function of $b$ given by 
  $$\frac{\d E(b)}{\d b} = - \frac{1}{P(b)}\int_M\bigl(H-{\mathcal E}_{\rho_b}(H)\bigr)^2\d\lambda_\omega
                         = - {\mathcal E}_{\rho_b}\Bigl(\bigl( H-{\mathcal E}_{\rho_b}(H)\bigr)^2\Bigr)\,.
 $$
Let $S(b)$ be the entropy $s(\rho_b)$ of the Gibbs statistical state associated to $b$. The function $S$
can be expressed in terms of $P$ and $E$ as
 $$S(b)=\log\bigl(P(b)\bigr)+bE(b)\,.
 $$
Its derivative with respect to $b$ exists and is a continuous function of $b$ given by
 $$\frac{\d S(b)}{\d b}=b\frac{\d E(b)}{\d b}\,.
 $$
\end{prop}

\begin{proof}
Using the assumptions \ref{AssumptionsInThermodynamicEquilibria}, 
we see that the functions $b\mapsto P(b)$ and 
$b\mapsto {\mathcal E}_{\rho_b}(H)=E(b)$, defined by integrals on $M$, have a derivative with respect to $b$ which is continuous and which can be calculated by derivation under the sign $\displaystyle \int_M$.  
The indicated results easily follow, if we observe that for any function $f$ on $M$ such that
${\mathcal E}_{\rho_b}(f)$ and  ${\mathcal E}_{\rho_b}(f^2)$ exist, we have the formula, well known in Probability theory,
 $${\mathcal E}_{\rho_b}(f^2)-\bigl({\mathcal E}_{\rho_b}(f)\bigr)^2
         ={\mathcal E}_{\rho_b}\Bigl(\bigl(f -{\mathcal E}_{\rho_b}(f)\bigr)^2\Bigr)\,.\eqno{\qedhere}
 $$
\end{proof}

\subsubsection{Physical meaning of the introduced functions}\label{PhysicalMeaningOfThermodynamicFunctions}
Let us consider a physical system, for example a gas contained in a vessel bounded by 
rigid, thermally insulated walls, at rest in a Galilean reference frame. We assume that 
its evolution can be mathematically described by means of a Hamiltonian system on a 
symplectic manifold $(M,\omega)$ whose Hamiltonian $H$ satisfies the assumptions 
\ref{AssumptionsInThermodynamicEquilibria}. 
For physicists, a Gibbs statistical state, \emph{i.e.} a
probability measure of density $\displaystyle \rho_b=\frac{1}{P(b)}\exp(-bH)$ on $M$, 
is a \emph{thermodynamic equilibrium} of the physical system. The set of possible 
thermodynamic equilibria of the system is therefore indexed by a real parameter $b>0$. 
The following argument will show what physical meaning can have that parameter.
\par\smallskip

Let us consider two similar physical systems, mathematically described by two
Hamiltonian systems, of Hamiltonians $H_1$ on the symplectic manifold $(M_1,\omega_1)$ 
and $H_2$ on the symplectic manifold $(M_2,\omega_2)$. We first assume that 
they are independent and both in thermodynamic equilibrium, with different values 
$b_1$ and $b_2$ of the parameter $b$. We denote by $E_1(b_1)$ and $E_2(b_2)$ 
the mean values of $H_1$ on the manifold $M_1$ with respect to the Gibbs state 
of density $\rho_{1,b_1}$
and of $H_2$ on the manifold $M_2$ with respect to the Gibbs state of density 
$\rho_{2,b_2}$. We assume now that the two systems are coupled in a way allowing 
an exchange of energy. For example, the two vessels
containing the two gases can be separated by a wall allowing a heat transfer between 
them. Coupled together, they make a new physical system, mathematically described 
by a Hamiltonian system on the
symplectic manifold $(M1\times M_2, p_1^*\omega_1+p_2^*\omega_2)$, where
$p_1:M_1\times M_2\to M_1$ and $p_2:M_1\times M_2\to M_2$ are the canonical projections. 
The Hamiltonian of this new system can be made as close to
$H_1\circ p_1+H_2\circ p_2$ as one wishes, by making very small 
the coupling between the two systems.
The mean value of the Hamiltonian of the new system is therefore 
very close to $E_1(b_1)+E_2(b_2)$. 
When the total system will reach a state of thermodynamic equilibrium, 
the probability densities 
of the Gibbs states of its two parts, $\rho_{1,b'}$ on $M_1$ and 
$\rho_{2,b'}$ on $M_2$ will 
be indexed by the same real number $b'>0$, which must be such that
 $$E_1(b')+E_2(b')=E_1(b_1)+E_2(b_2)\,.
 $$   
By \ref{PropertiesOfThermodynamicFunctions}, we have, for all $b>0$,
 $$\frac{\d E_1(b)}{\d b}\leq 0\,,\quad \frac{\d E_2(b)}{\d b}\leq 0\,.
 $$ 
Therefore $b'$ must lie between $b_1$ and $b_2$. If, for example, $b_1<b_2$, we see that
$E_1(b')\leq E_1(b_1)$ and $E_2(b')\geq E_2(b_2)$. In order to reach a state of thermodynamic 
equilibrium, energy must be transferred from the part of the system where $b$ has the smallest value, towards the part of the system where $b$ has the highest value, until, at thermodynamic equilibrium, $b$ has the same value everywhere. Everyday experience shows that thermal energy flows from parts of a system where the temperature is higher, towards parts where it is lower. For this reason physicists consider the real variable
$b$ as a way to appreciate the temperature of a physical system in a state of thermodynamic equilibrium. 
More precisely, they state that
 $$b=\frac{1}{kT}$$
where $T$ is the absolute temperature and $k$ a constant depending on the choice of units of energy and temperature, called \emph{Boltzmann's constant} in honour of the great Austrian scientist Ludwig
Eduard Boltzmann (1844--1906).

For a physical system mathematically described by a Hamiltonian system on a symplectic manifold
$(M,\omega)$, with $H$ as Hamiltonian, in a state of thermodynamic equilibrium, $E(b)$ and $S(b)$ are
the \emph{internal energy} and the \emph{entropy} of the system.

\subsubsection{Towards thermodynamic equilibrium}\label{TowardsThermodynamicEquilibrium}
Everyday experience shows that a physical system, when submitted to external 
conditions which remain unchanged for a sufficiently long time, 
very often reaches a state of thermodynamic equilibrium. At first look, it seems that
Lagrangian or Hamiltonian systems with time-independent Lagrangians or Hamiltonians 
cannot exhibit a similar behaviour. Let us indeed consider a
mechanical system whose configuration space is a smooth manifold $N$, described in the Lagrangian formalism by a smooth time-independent
hyper-regular Lagarangian $L:TN\to \RR$ or, in the Hamiltonian formalism, by the associated
Hamiltonian $H_L:T^*N\to\RR$. Let $t\mapsto \vect{x(t)}$ be a motion of that system, 
$\vect{x_0}=\vect{x(t_0)}$ and $\vect{x_1}=\vect{x(t_0)}$ be the configurations of the system for that motion at times $t_0$ and $t_1$. There exists another motion 
$t\mapsto\vect{x'(t)}$ of the system for which $\vect{x'(t_0)}=\vect{x_1}$ and
$\vect{x'(t_1)}=\vect{x_0}$: since the equations of motion are invariant by time reversal,
the motion $t\mapsto\vect{x'(t)}$ is obtained simply by taking as initial condition at time
$t_0$
$\vect{x'(t_0)}=\vect{x(t_1)}$   and 
$\displaystyle\frac{\d \vect{x'(t)}}{\d t}\Bigm|_{t=t_0}
 =-\frac{\d \vect{x(t)}}{\d t}\Bigm|_{t=t_1}$.
Another more serious argument against a kind of thermodynamic behaviour
of Lagarangian or Hamiltonian systems rests on
the famous recurrence theorem due to H.~Poincaré 
\cite{Poincare1890}. This theorem asserts indeed that when the useful part of the
phase space of the system is of a finite total measure, almost all points in an 
arbitrarily small open subset of the phase space are recurrent, \emph{i.e.}, the motion starting of such a point at time  $t_0$ repeatedly cross that open subset again and again, infinitely many times when $t\to +\infty$.
\par\smallskip
 
Let us now consider, instead of perfectly defined states, \emph{i.e.}, points in phase space, statistical states, and ask the question:
When at time $t=t_0$ a Hamiltonian system on a symplectic manifold $(M,\omega)$
is in a statistical state given by some probability measure of density
$\rho_0$ with respect to the Liouville measure $\lambda_\omega$, does its statistical state
converge, when $t\to+\infty$, towards the probability measure of a Gibbs state?
This question should be made more precise by specifying what physical meaning has a statistical state and in what mathematical sense a statistical state can converge towards the probability measure of a Gibbs state. A positive partial answer was given by Ludwig Boltzmann when, developing his kinetic theory of gases, he proved his famous (but controversed)
\emph{\^Eta theorem} stating that the entropy of the statistical state of a gas of small particles is a monotonously increasing function of time.
This question, linked with time irreversibility in Physics, is still the subject of important researches, both by physicists and by mathematicians. The reader is referred to
the paper \cite{Balian2005} by R.~Balian for a more thorough discussion of that question.  

\subsection{Examples of thermodynamic equilibria}\label{ExamplesThermodynamicEquilibria}

\subsubsection{Classical monoatomic ideal gas}\label{ClassicalIdealGas}
In Classical Mechanics, a dilute gas contained in a vessel at rest in a Galilean reference frame
is mathematically described by a Hamiltonian system made by a large number of very small massive 
particles, which interact by very brief collisions between themselves or with the walls of the vessel, 
whose motions between two collisions are free. Let us first assume that these particles are material points and that no external field is acting on them, other than that describing the interactions by collisions 
with the walls of the vessel. 
\par\smallskip

The Hamiltonian of one particle in a part of the phase space in which its 
motion is free is simply
 $$\frac{1}{2m}\Vert \vect p\Vert^2=\frac{1}{2m}(p_1^2+p_2^2+p_3^2)\,,\quad
                \hbox{with}\quad \vect p=m\vect v\,,
 $$ 
where $m$ is the mass of the particle, $\vect v$ its velocity vector and $\vect p$ its linear momentum vector
(in the considered Galilean reference frame), $p_1$, $p_2$ and $p_3$ the components of $\vect p$ in a
fixed orhtonormal basis of the physical space.
\par\smallskip

Let $N$ be the total number of particles, which may not have all the same mass. 
We use a integer $i\in\{1,\ 2,\ \ldots,\ N\}$ to label 
the particles and denote by $m_i$, $\vect{x_i}$, $\vect{v_i}$, $\vect{p_i}$ the mass and the
vectors position, velocity and linear momentum of the $i$-th particle. 
\par\smallskip

The Hamiltonian of the gas is therefore
 $$ H = \sum_{i=1}^N\frac{1}{2 m_i}\Vert \vect{p_i}\Vert^2 +\ \hbox{terms involving the collisions between particles and with the walls}\,.
 $$
Interactions of the particles with the walls of the vessel
are essential for allowing the motions of particles to remain confined. Interactions between particles
are essential to allow the exchanges between them of energy and momentum, which play an important part 
in the evolution with time of the statistical state of the system. However it appears that while these terms
are very important to determine the system's evolution with time, they can be neglected, when the gas is dilute enough, if we only want to determine the final statistical state of the system, once
a thermodynamic equilibrium is established. The Hamiltonian used will therefore be
 $$ H = \sum_{i=1}^N\frac{1}{2 m_i}\Vert \vect{p_i}\Vert^2\,.
 $$
The partition function is
 $$ P(b)=\int_M\exp(-bH)\d\lambda_\omega=\int_D\exp\left(-b\sum_{i=1}^N\frac{1}{2m_i}\Vert \vect 
             p_i\Vert^2\right)\prod_{i=1}^N(\d\vect{x_i}\d\vect{p_i})\,,
 $$
where $D$ is the domain of the $6N$-dimensional space spanned by the position vectors $\vect{x_i}$ and linear
momentum vectors $\vect{p_i}$ of the particles in which all the $\vect{x_i}$ lie within the vessel containing the gas. An easy calculation leads to
 $$P(b)=V^N\left(\frac{2\pi}{b}\right)^{3N/2}\prod_{i=1}^N ({m_i}^{3/2})
       =\prod_{i=1}^N\left[V\left(\frac{2\pi m_i}{b}\right)^{3/2}\right]\,,
 $$
where $V$ is the volume of the vessel which contains the gas. The probability density of the Gibbs
state associated to $b$, with respect to the Liouville measure, therefore is
 $$\rho_b=\prod_{i=1}^N\left[\frac{1}{V}\left(\frac{b}{2\pi m_i}\right)^{3/2}\exp
    \left(\frac{-b\Vert\vect{p_i}\Vert^2}{2m_i}\right)\right]\,.
 $$
We observe that $\rho_b$ is the product of the probability densities $\rho_{i,b}$ for the $i$-th particle
 $$\rho_{i,b}=\frac{1}{V}\left(\frac{b}{2\pi m_i}\right)^{3/2}\exp
    \left(\frac{-b\Vert\vect{p_i}\Vert^2}{2m_i}\right)\,.
 $$
The $2N$ stochastic vectors $\vect{x_i}$ and $\vect{p_i}$, $i=1,\ \ldots\,,\ N$ are therefore independent.
The position $\vect{x_i}$ of the $i$-th particle is uniformly distributed in the volume of the vessel, while the probability measure of the its linear momentum $\vect{p_i}$ is the classical
\emph{Maxwell-Boltzmann probability distribution of linear momentum} for an ideal gas of particles 
of mass $m_i$, first obtained by Maxwell in 1860. 
Moreover we see that the three components $p_{i\,1}$, $p_{i\,2}$ and $p_{i\,3}$ of the linear momentum 
$\vect{p_i}$ in an orhonormal basis of the physical space are independent stochastic variables. 
\par\smallskip

By using the formulae  given in \ref{PropertiesOfThermodynamicFunctions} the internal energy $E(b)$ and the entropy $S(b)$ of the gas can be easily deduced from the partition function $P(b)$. Their expressions are
 $$E(b)=\frac{3N}{2b}\,,\quad 
   S(b)=\frac{3}{2}\sum_{i=1}^N \log m_i + \left(\frac{3}{2}\bigl(1+\log(2\pi)\bigr) 
       + \log V\right)N-\frac{3N}{2}\log b\,.
 $$
We see that each of the $N$ particles present in the gas has the same contribution 
$\displaystyle\frac{3}{2b}$ to the internal energy $E(b)$, which does not depend on the mass of the particle. Even more: each degree of freedom of each particle, \emph{i.e.} each of the the three components of the the linear momentum of the particle on the three axes of an orthonormal basis, has the same contribution 
$\displaystyle\frac{1}{2b}$ to the internal energy $E(b)$. This result is known in Physics under the name \emph{Theorem of equipartition of the energy at a thermodynamic equilibrium}. It can be easily generalized
for polyatomic gases, in which a particle may carry, in addition to the kinetic energy due to the velocity of its centre of mass, a kinetic energy due to the particle's rotation around its centre of mass. The reader can consult the books by Souriau \cite{Souriau1969} and Mackey \cite{Mackey1963} where the kinetic theory of polyatomic gases is discussed.  
\par\smallskip
  
The pressure in the gas, denoted by $\Pi(b)$ because the notation $P(b)$ is already used for the partition function, is due to the change of linear momentum of the particles which occurs at a collision of the particle with the walls of the vessel containing rhe gas (or with a probe used to measure that pressure). A classical argument in the kinetic theory of gases (see for example \cite{HyperPhysics} or \cite{Gastebois}) leads to
 $$\Pi(b)= \frac{2}{3}\frac{E(b)}{V}=\frac{Nb}{V}\,.
 $$
This formula is the well known \emph{equation of state} of an ideal monoatomic gas relating the number of particles by unit of volume, the pressure and the temperature.
\par\smallskip

With $\displaystyle b=\frac{1}{kT}$, the above expressions are exactly those used in classical 
Thermodynamics for an ideal monoatomic gas.

\subsubsection{Classical ideal monoatomic gas in a gravity field}\label{ClassicalGasWithGravity}
Let us now assume that the gas, contained in a cylindrical vessel of section $\Sigma$ and length $h$, 
with a vertical axis, is submitted to the vertical gravity field of intensity $g$ directed downwards. 
We choose Cartesian coordinates $x$, $y$, $z$, the $z$ axis being vertical directed upwards, 
the bottom of the vessel being in the horizontal surface $z=0$. The Hamiltonian of a free particle 
of mass $m$, position and linear momentum vectors $\vect x$ (components $x$, $y$, $z$) and $\vect p$
(components $p_x$, $p_y$ and $p_z$) is
 $$\frac{1}{2m}(p_x^2+p_y^2+p_z^2) + mgz\,.
 $$
As in the previous section we neglect the parts of the Hamiltonian of the gas corresponding 
to collisions between the particles, or between a particle and the walls of the vessel. 
The Hamiltonian of the gas is therefore
 $$ H = \sum_{i=1}^N\left(\frac{1}{2m_i}(p_{i\,x}^2 + p_{i\,y}^2 + p_{i\,z}^2)+m_igz_i\right)\,.
 $$
Calculations similar to those of the previous section lead to
 \begin{align*}
  P(b)   &=\prod_{i=1}^N\left[\Sigma\left(\frac{2\pi m_i}{b}\right)^{3/2}\frac{1-\exp(-m_igbh)}{m_igb}\right]\,, 
\\
  \rho_b &=\frac{1}{P(b)}\exp\left[-b\sum_{i=1}^N\left(\frac{\Vert\vect{p_i}\Vert^2}{2m_i}+m_igz_i
            \right)\right]\,.\\
 \end{align*}
The expression of $\rho_b$ shows that the $2N$ stochastic vectors $\vect{x_i}$ and $\vect{p_i}$ still are independent, and that for each $i\in\{1,\ldots,N\}$, the probability law of each stochastic vector 
$\vect{p_i}$ is the same as in the absence of gravity, for the same value of $b$. Each stochastic vector 
$\vect{x_i}$ is no more uniformly distributed in the vessel containing the gas: its probability density
is higher at lower altitudes $z$, and this nonuniformity is more important for the heavier particles than for the lighter ones.
\par\smallskip

As in the previous section, the formulae  given in \ref{PropertiesOfThermodynamicFunctions} allow the calculation of $E(b)$ and $S(b)$. We observe that $E(b)$ now includes the potential energy of the gas in the gravity field, therefore should no more be called the internal energy of the gas.

\subsubsection{Relativistic monoatomic ideal gas}\label{RelativistiGas}

In a Galilean reference frame, we consider a relativistic point particle of rest mass $m$, moving at a velocity $\vect v$. We denote by $v$ the modulus of $\vect v$ and by $c$ the modulus of the velocity of light. The motion of the particle can be mathematically described by means of the Euler-Lagrange equations, with the Lagrangian
 $$L=-mc^2\sqrt{1-\frac{v^2}{c^2}}\,.
 $$
The components of the linear momentum $\vect p$ of the particle, in an orthonormal frame at rest in the considered Galilean reference frame, are
 $$p_i=\frac{\partial L}{\partial v^i}=\frac{\displaystyle mv^i}{\displaystyle\sqrt{1-\frac{v^2}{c^2}}}\,,
     \quad\hbox{therefore}\quad \vect p=\frac{\displaystyle m\vect v}{\displaystyle\sqrt{1-\frac{v^2}{c^2}}}\,.
 $$
Denoting by $p$ the modulus of $\vect p$, the Hamiltonian of the particle is
 $$H=\vect p\cdot\vect v-L=\frac{\displaystyle mc^2}{\displaystyle\sqrt{1-\frac{v^2}{c^2}}}
                          =c\sqrt{p^2+m^2c^2}\,.
 $$
Let us consider a relativistic gas, made of $N$ point particles indexed by $i\in\{1,\ldots,N\}$, $m_i$
being the rest mass of the $i$-th particle. With the same assumptions as those made in Section
\ref{ClassicalIdealGas}, we can take for Hamiltonian of the gas
 $$H=c\sum_{i=1}^N\sqrt{{p_i}^2+m^2c^2}\,.
 $$
With the same notations as those of Section \ref{ClassicalIdealGas}, the partition function $P$ of the gas takes the value, for each $b>0$,
 $$P(b)=\int_D\exp\left(-bc\sum_{i=1}^N\sqrt{(p_i)^2+m^2c^2}\right)\prod_{i=1}^N(\d\vect{x_i}\d\vect{p_i})\,. 
 $$
This integral can be expressed in terms of the Bessel function $K_2$, whose expression is, for each $x>0$,
 $$K_2(x)=x\int_0^{+\infty}\exp(-x\ch \chi)\sh^2\chi\ch\chi\d\chi\,.
 $$
We have
 \begin{align*}
   P(b)&=\left(\frac{4\pi Vc}{b}\right)^N\prod_{i=1}^N
         \bigl({m_i}^2K_2(m_ibc^2)\bigr)\,,\\
   \rho_b&=\frac{1}{P(b)}\exp\left(-bc\sum_{i=1}^N\sqrt{{p_i}^2+{m_i}^2c^2}\right)\,.
 \end{align*}
This probability density of the Gibbs state shows that the $2N$ stochastic vectors 
$\vect{x_i}$ and $\vect{p_i}$ are independent, that each $\vect{x_i}$ is uniformly distributed 
in the vessel containing the gas and that the probability density of each $\vect{p_i}$ is 
exactly the probability distribution of the linear momentum of particles in a relativistic gas 
called the \emph{Maxwell-J\"uttner distribution}, obtained by Ferencz J\"uttner (1878--1958) 
in 1911, discussed in the book by the Irish mathematician and physicist J.~L.~Synge \cite{Synge1957}.
\par\smallskip

Of course, the formulae  given in \ref{PropertiesOfThermodynamicFunctions} allow the calculation of 
the internal energy $E(b)$, the entropy $S(b)$ and the pressure $\Pi(b)$ of the relativistic gas.

\subsubsection{Relativistic ideal gas of massless particles}\label{RelativisticGasOfMasslessParticles}
We have seen in the previous chapter that in an inertial reference frame, the Hamiltonian of a relativistic point particle of rest mass $m$ is $c\sqrt{p^2+m^2 c^2}$, where $p$ is the modulus of the linear momentum vector $\vect p$ of the particle in the considered reference frame. This expression still has a meaning when the rest mass $m$ of the particle is $0$. In an orthonormal reference frame, the equations of motion of a particle whose motion is mathematically described by a Hamiltonian system with Hamiltonian 
 $$H=cp=c\sqrt{{p_1}^2+{p_2}^2+{p_3}^2}
 $$ 
are
 \begin{equation*}\left\{
 \begin{aligned}
  \frac{\d x^i}{\d t}&=\frac{\partial H}{\partial p_i}=c\,\frac{p_i}{p}\,\\
  \frac{\d p_i}{\d t}&=-\frac{\partial H}{\partial x^i}=0\,,
 \end{aligned}\quad(1\leq i\leq 3)\,,
 \right. 
 \end{equation*}
which shows that the particle moves on a straight line at the velocity of light $c$. It seems therefore
reasonable to describe a gas of $N$ photons in a vessel of volume $V$ at rest in an inertial reference frame
by a Hamiltonian system, with the Hamiltonian
 $$H=c\sum_{i=1}^N\Vert\vect{p_i}\Vert=c\sum_{i=1}^N\sqrt{{p_{i\,1}}^2+{p_{i\,2}}^2+{p_{i\,3}}^2}\,.
 $$
With the same notations as those used in the previous section, the partition function $P$ of the gas 
takes the value, for each $b>0$,
 $$P(b)=\int_D\exp\left(-bc\sum_{i=1}^N\Vert\vect{p_i}\Vert\right)\prod_{i=1}^N(\d\vect{x_i}\d\vect{p_i})
       =\left(\frac{8\pi V}{c^3b^3}\right)^N\,. 
 $$
The probability density of the corresponding Gibbs state, with respect to the Liouville measure 
$\lambda_\omega=\prod_{i=1}^N(\d\vect{x_i}\d\vect{p_i})$,  is
 $$\rho_b= \prod_{i=1}^N\left(\frac{c^3b^3}{8\pi V}\right)\exp(-bc\Vert\vect{p_i}\Vert)\,.
 $$
This formula appears in the books by Synge \cite{Synge1957} and Souriau \cite{Souriau1969}. Physicists consider it as not adequate for the description of a gas of photons contained in a vessel at thermal equilibrium because the number of photons in the vessel, at any given temperature, 
cannot be imposed: it results from the processes of absorption and emission of photons by the walls of
the vessel, heated at the imposed temperature, which spontaneously occur. In other words, this number is a stochastic function 
whose probability law is imposed by Nature. Souriau proposes, in his book \cite{Souriau1969}, a way
to account for the possible variation of the number of photons. Instead of using the 
\emph{phase space} of the system of $N$ massless relativistic particles contained in a vessel, 
he uses the \emph{manifold of motions}
$M_N$ of that system (which is symplectomorphic to its phase space). He considers that the manifold of 
motions $M$ of a system of photons in the vessel is the disjoint union
 $$M=\bigcup_{N\in\NN} M_N\,,
 $$ 
of all the manifolds of motions $M_N$ of a system of $N$ massless relativistic particles in the vessel, for all possible values of $N\in\NN$. Fo $N=0$ the manifold $M_0$ is reduced to a singleton with, as
Liouville measure, the measure which takes the value $1$ on the only non empty part of that manifold 
(the whole manifold $M_0$). Moreover, since any photon cannot be distinguished from any other photon, two motions of the system with the same number $N$ of massless particles which only differ by the labelling of these particles must be considered as identical. Souriau considers too that since the number $N$ of photons freely adjusts itself, the value of the parameter $\displaystyle b=\frac{1}{kT}$ must, at thermodynamic equilibrium, be the same in all parts $M_N$ of the system, $N\in\NN$.
He uses too the fact that a photon can have two different states of (circular) polarization. With these assumptions the value at any $b$ of the partition function of the system is
 $$P(b)=\sum_{N=0}^{+\infty}\frac{1}{N!}\left(\frac{16\pi V}{c^3b^3}\right)^N
       =\exp\left(\frac{16\pi V}{c^3b^3}\right)\,.
 $$  
The number $N$ of photons in the vessel at thermodynamic equilibrium is a stochastic function which takes the value $n$ with the probability
 $$\text{Probability}\bigl([N=n]\bigr)=\frac{1}{N!}\left(\frac{16\pi V}{c^3b^3}\right)^N
                                       \exp\left(-\frac{16\pi V}{c^3b^3}\right)\,.
 $$ 
The expression of the partition function $P$ allows the calculation of the internal energy, the entropy and all other thermodynamic functions of the system. However, the formula so obtained for the distribution of photons of various energies at a given temperature does not agree with the law, in very good agreement with experiments, obtained by Max Planck
(1858--1947) in 1900. An assembly of photons in thermodynamic equilibrium 
evidently cannot be described as a classical Hamiltonian system. This fact played an important part for the development of Quantum Mechanics.

\subsubsection{Specific heat of solids}\label{SpecificHeatOfSolids}
The motion of a one-dimensional harmonic oscillator can be described by a Hamiltonian system
with, as Hamiltonian,
 $$H(p,q)=\frac{p^2}{2m}+\frac{\mu q^2}{2}\,.
 $$
The idea that the heat energy of a solid comes from the small vibrations, at a microscopic scale,
of its constitutive atoms, lead physicists to attempt to mathematically describe a solid 
as an assembly of a large number $N$ of three-dimensional harmonic oscillators. By dealing separately 
with each proper oscillation mode, the solid can even be described as an assembly of $3N$ 
one-dimensional harmonic oscillators. Exanges of energy between these oscillators is allowed
by the existence of small couplings between them. However, for the determination of the thermodynamic 
equilibria of the solid we will, as in the previous section for ideal gases, consider as negligible 
the energy of interactions between the oscillators. We therefore take for Hamiltonian of the solid
 $$ H=\sum_{i=1}^{3N}\left(\frac{{p_i}^2}{2m_i}+\frac{\mu_i {q_i}^2}{2}\right)\,.
 $$
The value of the paritition function $P$,   for any $b>0$, is
 $$P(b)=\int_{\RR^{6N}}\exp\left[-b\sum_{i=1}^{3N}\left(\frac{{p_i}^2}{2m_i}+\frac{\mu_i {q_i}^2}{2}
         \right)\right]\prod_{i=1}^{3N}(\d p_i\d q_i)=\prod_{i=1}^{3N}\left(\frac{1}{\nu_i}\right)b^{-3N}\,,
 $$
where
 $$\nu_i=\frac{1}{2\pi}\sqrt{\displaystyle\frac{\mu_i}{m_i}}
 $$
is the frequency of the $i$-th harmonic oscillator. 
\par\smallskip

The internal energy of the solid is
 $$E(b)=-\frac{\d\log P(b)}{\d b}=\frac{3N}{b}\,.
 $$
We observe that it only depends on the the temperature and on the number of atoms in the solid, 
not on the frequencies $\nu_i$ of the harmonic oscillators. With
$\displaystyle b=\frac{1}{kT}$ this result is in agreement with the empirical law for the specific heat of solids, in good agreement with experiments at high temperature, discovered in 1819 by the French scientists
Pierre Louis Dulong (1785--1838) and Alexis Thérèse Petit (1791--1820).

\section{Generalization for Hamiltonian actions}\label{GeneralizationForLieGroupsActions}

\subsection{Generalized Gibbs states}\label{GeneralizedGibbsStates}
In his book \cite{Souriau1974} and in several papers \cite{Souriau1966, Souriau1975, Souriau1984},
J.-M.~Souriau extends 
the concept of a Gibbs state for a Hamiltonian action of a 
Lie group $G$ on a symplectic manifold $(M,\omega)$. Usual Gibbs states defined in 
section \ref{StatisticalMechanicsThermodynamics} for a smooth Hamiltonian $H$ on a 
symplectic manifold $(M,\omega)$ appear as special cases, in which the Lie group 
is a one-parameter group.
If the symplectic manifold $(M,\omega)$ is the \emph{phase space} of the Hamiltonian system, that one-parameter
group, whose parameter is the time $t$, is the group of evolution, as a function of time,
of the state of the system, starting from its state at some arbitrarily chosen initial time $t_0$. 
If $(M,\omega)$ is the symplectic manifold of all the \emph{motions} of the system, that one-parameter group,
whose parameter is a real $\tau\in\RR$, is the transformation group which maps one motion of the system with some initial state at time $t_0$ onto the motion of the system with the same initial state at another time
$(t_0+\tau)$. We discuss below this generalization.

\subsubsection{Notations and conventions}\label{NotationsConventionsForGeneralizedGibbsStates} 
In this section, $\Phi:G\times M\to M$ is a Hamiltonian action
(for example on the left) of a Lie group $G$ on a symplectic manifold $(M,\omega)$. 
We denote by $\mathcal G$ the Lie algebra of $G$, by ${\mathcal G}^*$ its dual space 
and by $J:M\to{\mathcal G}^*$ a momentum map of the action $\Phi$. 

\begin{defis}\label{DefiGeneralizedGibbsState}
Let $b\in {\mathcal G}$ be such that the integrals on the right hand sides of the equalities
  \begin{align*}
   P(b)&=\int_M\exp\bigr(-\langle J, b\rangle\bigr)\d\lambda_\omega\quad\text{and}\\
   E_J(b)&={\mathcal E}_{\rho_b}(J)=\frac{1}{P(b)}\int_M J\exp\bigr(-\langle J, b\rangle\bigr)\d\lambda_\omega
 \end{align*}
converge. The smooth probability measure on $M$ with density (with respect to the Liouville measure
$\lambda_\omega$ on $M$)
 $$\rho_b=\frac{1}{P(b)}\exp\bigl(-\langle J,b\rangle\bigr)
 $$
is called the \emph{generalized Gibbs statistical state} associated to $b$. The functions $b\mapsto P(b)$ 
and $b\mapsto E_J(b)$ so defined on the subset of $\mathcal G$ made by elements $b$ for which the integrals defining $P(b)$ and $E(J,b)$ converge are called the \emph{partition function associated to the momentum map} $J$
and the \emph{mean value of $J$ at generalized Gibbs states}. 
\end{defis}

The following Proposition generalizes \ref{MaximalityEntropy}.

\begin{prop}\label{MaximalityEntropyGeneralized}
Let $b\in{\mathcal G}$ be such that the integrals defining $P(b)$ and $E_J(b)$ in Proposition
\ref{DefiGeneralizedGibbsState} converge, and $\rho_b$ be the density of the generalized Gibbs state
associated to $b$. The entropy $s(\rho_b)$, which will be denoted by $S(b)$, exists and is given by
 $$S(b)=\log \bigl(P(b)\bigr)+\bigl\langle E_J(b),b\bigr\rangle
       =\log\bigl(P(b)\bigr)-\Bigl\langle D\bigl(\log P(b)\bigr), b\Bigr\rangle\,.\eqno{(*)}
 $$
Moreover, for any other smooth probability density $\rho_1$ such that
 $${\mathcal E}_{\rho_1}(J)={\mathcal E}_{\rho_b}(J)=E_J(b)\,,
 $$   
we have
 $$s(\rho_1)\leq s(\rho_b)\,,$$
and the equality $s(\rho_1)=s(\rho_b)$ holds if and only if $\rho_1=\rho_b$. 
\end{prop}

\begin{proof} The equality $(*)$ immediately follows from 
$\displaystyle\log\left(\frac{1}{\rho_b}\right)=\log\bigl(P(b)\bigr)+\langle J,b\rangle$, 
and from $D\bigl(\log P(b)\bigr)=-E_J(b)$. 
The remaining of the proof is the same as that of Proposition
\ref{MaximalityEntropy}.
\end{proof}

\begin{rmks}\hfill
\label{NonInvarianceOfGeneralizedGibbsStates}

\par\smallskip\noindent
{\rm 1.\quad} The second equality $(*)$ above,
 $S(b)=\log\bigl(P(b)\bigr)-\Bigl\langle D\bigl(\log P(b)\bigr), b\Bigr\rangle
 $,
expresses the fact that the functions $\log\bigl(P(b)\bigr)$ and $-S(b)$ are 
\emph{Legendre transforms} of each other: they are linked by the same relation 
as the relation which links a smooth Lagrangian
$L$ and the associated energy $E_L$.

\par\smallskip\noindent
{\rm 2.\quad} The Liouville measure $\lambda_\omega$ remains invariant under the Hamiltonian action $\Phi$, since the symplectic form $\omega$ itself remains invariant under that action. However, we have not a full analogue of Proposition \ref{InvarianceOfGibbsState} because the momentum map $J$ does not remain invariant under the action $\Phi$. We only have the partial anologue stated below.

\par\smallskip\noindent
{\rm 3.\quad}
Legendre transforms were used by F.~Massieu in Thermodynamics in his very early works
\cite{Massieu1869a, Massieu1869b}, more systematically presented in \cite{Massieu1876},
in which he introduced his \emph{characteristic functions} 
(today called \emph{thermodynamic potentials}) allowing the determination of all the thermodynamic functions of a physical system by partial derivations of a suitably chosen characteristic function. For a modern presentation of that subject the reader is referred to \cite{Balian2015} and \cite{Callen}, chapter 5, pp.~131--152.  

\end{rmks}

\begin{prop}\label{PartialInvarianceOfGeneralizedGibbsState}   
Let $b\in{\mathcal G}$ be such that the integrals defining $P(b)$ and $M(b)$ in Proposition
\ref{DefiGeneralizedGibbsState} converge. The generalized Gibbs state
associated to $b$ remains invariant under the restriction of the Hamiltonian action $\Phi$ to the 
one-parameter subgroup of $G$ generated by $b$, $\bigl\{\exp(\tau b)\bigm|\tau\in\RR\bigr\}$.
\end{prop}

\begin{proof}
The orbits of the action on $M$ of the subgroup $\bigl\{\exp(\tau b)\bigm|\tau\in\RR\bigr\}$ of $G$
are the integral curves of the Hamiltonian vector field whose Hamiltonian is $\langle J,b\rangle$, 
which of course is constant on each of these curves. Therefore the proof of \ref{InvarianceOfGibbsState}  
is valid for that subgroup. 
\end{proof}

\subsection{Generalized thermodynamic functions}\label{GeneralizedThermodynamicFunctions}

\subsubsection{Assumptions made in this section.\quad}
\label{AssumptionsForGeneralizedThermodynamicFunctions}
Notations and conventions being the same as in \ref{NotationsConventionsForGeneralizedGibbsStates}, 
let $\Omega$ be the largest open subset of the Lie algebra $\mathcal G$ of $G$ containing all
$b\in {\mathcal G}$ satisfying the following properties:

\begin{itemize}
\item{}
the functions defined on $M$, with values, respectively,
in $\RR$ and in the dual ${\mathcal G}^*$ of $\mathcal G$,
 $$z\mapsto \exp\Bigl(-\big\langle J(z),b\bigr\rangle\Bigl)\quad\hbox{and}\quad
   z\mapsto J(z)\exp\Bigl(-\big\langle J(z),b\bigr\rangle\Bigl)
 $$
are integrable on $M$ with respect to the Liouville measure $\lambda_\omega$; 

\item{} moreover their integrals are differentiable with repect to $b$, their differentials are continuous and can be calculated by differentiation under the sign $\int_M$.
\end{itemize}

It is assumed in this section that the considered Hamiltonian action $\Phi$ of the Lie group $G$ on the 
symplectic manifold $(M,\omega)$ and its momentum map $J$ are such that
the open subset $\Omega$ of $\mathcal G$ is not empty. This condition is not always satisfied
when $(M,\omega)$ is a cotangent bundle, but of course it is satisfied
when it is a compact manifold.

\begin{prop}\label{PropertiesOfGeneralizedThermodynamicFunctions}
Let $\Phi:G\times M\to M$ be a Hamiltonian action of a Lie group $G$ on a symplectic manifold 
$(M,\omega)$ satisfying the assumptions indicated in \ref{AssumptionsForGeneralizedThermodynamicFunctions}. The partition function $P$ associated to the momentum map
$J$ and the mean value $E_J$ of $J$ for generalized Gibbs states (\ref{DefiGeneralizedGibbsState}) 
are defined and continuously differentiable on the open subset $\Omega$ of $\mathcal G$. For each $b\in\Omega$, the differentials at $b$
of the functions $P$ and $\log P$ (which are linear maps defined on $\mathcal G$, with values in $\RR$, in other words elements of ${\mathcal G}^*$) are given by  
 $$ DP(b)=-P(b)E_J(b)\,,\quad D(\log P)(b)=-E_J(b)\,.
 $$
For each $b\in\Omega$, the differential at $b$ of the map $E_J$ (which is a linear map defined on $\mathcal G$,
with values in its dual ${\mathcal G}^*$) is given by
 $$\bigl\langle DE_J(b)(Y),Z\bigr\rangle=\bigl\langle E_J(b),Y\bigr\rangle\bigl\langle E_J(b),Z\bigr\rangle
   -{\mathcal E}_{\rho_b}\bigl(\langle J,Y\rangle\langle J,Z\rangle\bigr)\,,\quad
      \hbox{with}\ Y\ \hbox{and}\ Z\in{\mathcal G}\,,
 $$
where we have written, as in \ref{DefisMeanValueVariationsStationarity},
 $${\mathcal E}_{\rho_b}\bigl(\langle J,Y\rangle\langle J,Z\rangle\bigr)
   =\frac{1}{P(b)}\int_M\langle J,Y\rangle\langle J,Z\rangle\exp\bigl(-\langle J,b\rangle\bigr)\d\lambda_\omega\,.
 $$
At each $b\in\Omega$, the differential of the entropy function $S$ (\ref{MaximalityEntropyGeneralized}), 
which is a linear map defined on $\mathcal G$, with values in $\RR$, 
in other words an element of ${\mathcal G}^*$, is given by
 $$\bigl\langle DS(b),Y\bigr\rangle=\bigl\langle DE_J(b)(Y),b\bigr\rangle\,,\quad Y\in{\mathcal G}\,.
 $$
\end{prop}

\begin{proof}
By assumptions \ref{AssumptionsForGeneralizedThermodynamicFunctions}, the differentials of $P$ and $E_J$ can be calculated by differentiation under the sign $\int_M$. Easy (but tedious) calculations lead to the indicated results. 
\end{proof}

\begin{coro}\label{CorollaryPropertiesOfGeneralizedThermodynamicFunctions} 
With the same assumptions and notations as those in Proposition 
\ref{PropertiesOfGeneralizedThermodynamicFunctions}, for any $b\in\Omega$ and $Y\in{\mathcal G}$,
 $$\bigl\langle DE_J(b)(Y),Y\bigr\rangle=-\frac{1}{P(b)}
         \int_M\bigl\langle J-E_J(b),Y\bigr\rangle^2\d\lambda_\omega\leq 0\,.
 $$
\end{coro}

\begin{proof} This result follows from the well known result in Probability theory already
used in the proof of \ref{PropertiesOfThermodynamicFunctions}.
\end{proof}

The momentum map $J$ of the Hamiltonian action $\Phi$ is not uniquely determined: for any
constant $\mu\in{\mathcal G}^*$, $J_1=J+\mu$ too is a momentum map for $\Phi$. The following
proposition indicates how the generalized thermodynamic functions $P$, $E_J$ and $S$
change when $J$ is replaced by $J_1$.

\begin{prop}\label{EffectChangeOfJ}
With the same assumptions and notations as those in Proposition 
\ref{PropertiesOfGeneralizedThermodynamicFunctions}, let $\mu\in{\mathcal G}^*$ be a constant. When the momentum map $J$ is replaced by $J_1=J+\mu$, the open subset $\Omega$ of $\mathcal G$ remains unchanged,
while the generalized thermodynamic functions $P$, $E_J$ and $S$, are replaced, respectively, 
by $P_1$, $E_{J_1}$ and $S_1$, given by
 $$P_1(b)=\exp\bigl(-\langle \mu,b\rangle\bigr)P(b),\quad
   E_{J_1}(b)=E_J(b)+\mu\,,\quad S_1(b)=S(b)\,.
 $$
The Gibbs satistical state and its density $\rho_b$ with respect to the Liouville measure $\lambda_\omega$
remain unchanged. 
\end{prop}

\begin{proof} We have 
 $$\exp\bigl(-\langle J+\mu, b\rangle=\exp\bigl(-\langle\mu,b\rangle\bigl)
                                       \exp\bigl(-\langle J,b\rangle\bigl)\,.
 $$
The indicated results follow by easy calculations.
\end{proof}

The following proposition indicates how the generalized thermodynamic functions
$P$, $E_J$ and $S$ vary along orbits of the adjoint action of the Lie group $G$
on its Lie algebra $\mathcal G$.

\begin{prop}\label{EffectAdjointActionOfG}
The assumptions and notations are the same as those in Proposition 
\ref{PropertiesOfGeneralizedThermodynamicFunctions}. The open subset $\Omega$ of $\mathcal G$
is an union of orbits of the adjoint action of $G$ on $\mathcal G$. In other words,
for each $b\in\Omega$ and each $g\in G$, $\Ad_g b\in\Omega$. Moreover, let $\theta:G\to{\mathcal G}^*$
be the symplectic cocycle of $G$ for the coadjoin action of $G$ on ${\mathcal G}^*$ such that, 
for any $g\in G$,
 $$ J\circ\Phi_g=\Ad^*_{g^{-1}}\circ\, J+\theta(g)\,.
 $$ 
Then for each $b\in\Omega$ and each $g\in G$
 \begin{align*} 
       P(\Ad_gb)  &=\exp\Bigl(\bigl\langle\theta(g^{-1}),b\bigr\rangle\Bigr)P(b)
                   =\exp\Bigl(-\bigl\langle\Ad^*_g\theta(g),b\bigr\rangle\Bigr)P(b)\,,\\
       E_J(\Ad_gb)&=\Ad^*_{g^{-1}}E_J(b)+\theta(g)\,,\\
       S(\Ad_gb)  &=S(b)\,. 
 \end{align*}
\end{prop}

\begin{proof}
We have
 \begin{align*}
  P(\Ad_gb)&=\int_M\exp\bigl(-\langle J,\Ad_gb\rangle\bigr)\d\lambda_\omega
            =\int_M\exp\bigl(-\langle \Ad^*_g J,b\rangle\bigr)\d\lambda_\omega\\
           &=\int_M\exp\Bigl(-\bigl\langle J\circ\Phi_{g^{-1}}-\theta(g^{-1},b\bigr\rangle\Bigr)
              \d\lambda_\omega\\
           &=\exp\Bigl(\bigl\langle\theta(g^{-1}),b\bigr\rangle\Bigr)P(b)
            =\exp\Bigl(-\bigl\langle\Ad^*_g\theta(g),b\bigr\rangle\Bigr)P(b)\,,\\
 \end{align*}
since $\theta(g^{-1})=-\Ad^*_g\theta(g)$. By using 
\ref{PropertiesOfGeneralizedThermodynamicFunctions} and \ref{MaximalityEntropyGeneralized},
the other results easily follow.
\end{proof} 

\begin{rmk}\label{EquivarianceOfEJ}
The equality
 $$ E_J(\Ad_gb)=\Ad^*_{g^{-1}}E_J(b)+\theta(g)$$
means that the map $E_J:\Omega\to{\mathcal G}^*$ is equivariant with respect to the adjoint action
of $G$ on the open subset $\Omega$ of its Lie algebra $\mathcal G$ and its affine action on the left
on ${\mathcal G}^*$
 $$(g,\xi)\mapsto \Ad^*_{g^{-1}}\xi+\theta(g)\,,\quad g\in G\,,\quad \xi\in{\mathcal G}^*\,.
 $$
\end{rmk} 

\begin{prop}\label{EffectInfinitesimalAdjointAction}
The assumptions and notations are the same as those in Proposition 
\ref{PropertiesOfGeneralizedThermodynamicFunctions}. 
For each $b\in\Omega$ and each $X\in{\mathcal G}$, we have
   \begin{align*}
    \bigl\langle E_J(b),[X,b]\bigr\rangle&=\bigl\langle\Theta(X),b\bigr\rangle\,,\\
    DE_J(b)\bigl([X,b]\bigr)&=-\ad^*_XE_J(b) + \Theta(X)\,,
   \end{align*}
where $\Theta=T_e\theta:{\mathcal G}\to{\mathcal G}^*$ is the $1$-cocycle of the Lie algebra ${\mathcal G}$
associated to the $1$-cocycle $\theta$ of the Lie group $G$.
\end{prop}

\begin{proof}
Let us set $g=\exp(\tau X)$ in the first equality in \ref{EffectAdjointActionOfG}, 
derive that equality with respect to $\tau$, and evaluate the result at $\tau=0$.
We obtain
 $$DP(b)\bigl([X,b]\bigr)=-P(b)\bigl\langle\Theta(X),b\bigr\rangle\,.
 $$
Since, by the first equality of \ref{PropertiesOfGeneralizedThermodynamicFunctions},
$DP(b)=-P(b)E_J(b)$, the first stated equality follows.
\par\smallskip

Let us now set $g=\exp(\tau X)$ in the second equality in \ref{EffectAdjointActionOfG}, 
derive that equality with respect to $\tau$, and evaluate the result at $\tau=0$.
We obtain the second equality stated.
\end{proof}

\begin{coro}\label{cocycleThetab}
With the assumptions and notations of \ref{EffectInfinitesimalAdjointAction}, let us define,
for each $b\in\Omega$, a linear map $\Theta_b:{\mathcal G}\to{\mathcal G}^*$ by setting
 $$\Theta_b(X)=\Theta(X) -\ad^*_XE_J(b)\,.
 $$
The map $\Theta_b$ is a symplectic $1$-cocycle of the Lie algebra $\mathcal G$ for the coadjoint
representation, which satisfies
 $$\Theta_b(b)=0\,.
 $$
Moreover if we replace the momentum map $J$ by $J_1=J+\mu$, with $\mu\in{\mathcal G}^*$ constant,
the $1$-cocycle $\Theta_b$ remains unchanged.
\end{coro} 

\begin{proof}
For $X$, $Y$ and $Z$ in ${\mathcal G}$, we have since $\Theta$ is a $1$-cocycle, 
$\displaystyle\sum_{{\rm circ}(X,Y,Z)}$
meaning a sum over circular permutations of $X$, $Y$ and $Z$, using the Jacobi identity in $\mathcal G$, 
we have
 \begin{align*}
 \sum_{{\rm circ}(X,Y,Z)}\bigl\langle\Theta_b(X),[Y,Z]\bigr\rangle 
      &=\sum_{{\rm circ}(X,Y,Z)}\bigl\langle-\ad^*_XE_J(b),[Y,Z]\bigr\rangle\\
      &=\sum_{{\rm circ}(X,Y,Z)}\bigl\langle-E_J(b),\bigl[X,[Y,Z]\bigr]\bigr\rangle\\
      &=0\,.
 \end{align*}
The linear map $\Theta_b$ is therefore a $1$ cocycle, even a symplectic $1$-cocycle since for all $X$ and $Y\in{\mathcal G}$, $\bigl\langle\Theta_b(X),Y\bigr\rangle
=-\bigl\langle\Theta_b(Y),X\bigr\rangle$. 
\par\smallskip

Using the first equality stated in \ref{EffectInfinitesimalAdjointAction}, we have for any $X\in{\mathcal G}$
 $$
  \bigl\langle\Theta_b(b),X\bigr\rangle
   =\bigl\langle\Theta(b)-\ad^*_bE_J(b),X\bigr\rangle
   =-\bigl\langle\Theta(X),b\bigr\rangle+\bigl\langle E_J(b),[X,b]\bigr\rangle
   =0\,.
 $$
If we replace $J$ by $J_1=J+\mu$, the map $X\mapsto\Theta(X)$ 
is replaced by
$X\mapsto\Theta_1(X)=\Theta(X)+\ad^*_X\mu$ and $E_J(b)$ by $E_{J_1}(b)=E_J(b)+\mu$, therefore $\Theta_b$ remains unchanged.   
\end{proof}

The following lemma will allow us to define, for each $b\in\Omega$, a remarkable symmetric bilinear
form on the vector subspace $[b,{\mathcal G}]=\bigl\{[b,X]\,;X\in{\mathcal G}\bigr\}$ of the
Lie algebra $\mathcal G$.

\begin{lemma}\label{EuclideanProduct}
Let $\Xi$ be a $1$-cocycle of a finite-dimensional Lie algebra $\mathcal G$ for 
the coadjoint representation. For each $b\in\ker\Xi$, let $F_b=[{\mathcal G},b]$ be 
the set of elements $X\in{\mathcal G}$ which can be written $X=[X_1,b]$ for some 
$X_1\in{\mathcal G}$. Then $F_b$ is a vector subspace of ${\mathcal G}$, and the value 
of the right hand side of the equality
 $$\Gamma_b(X,Y)=\bigl\langle \Xi(X_1), Y\bigr\rangle\,,\quad\hbox{with}\ X_1\in{\mathcal G}\,,\ 
   X=[X_1,b]\in   F_b\,,\ Y\in F_b\,,
 $$ 
depends only on $X$ and $Y$, not on the choice of $X_1\in{\mathcal G}$ such that $X=[X_1,b]$. That equality defines a bilinear form $\Gamma_b$ on $F_b$ which is symmetric, \emph{i.e.} satisfies
 $$\Gamma_b(X,Y)=\Gamma_b(Y,X)\quad\hbox{for all}\ X\ \hbox{and}\ Y\in F_b\,.
 $$   
\end{lemma}

\begin{proof}
Let $X_1$ and $X'_1\in{\mathcal G}$ be such that $[X_1,b]=[X'_1,b]=X$. 
Let $Y_1\in{\mathcal G}$ be such that $[Y_1,b]=Y$. We have
  \begin{align*}
   \bigl\langle\Xi(X_1-X'_1),Y\bigr\rangle
    &=\bigl\langle\Xi(X_1-X'_1),[Y_1,b]\bigr\rangle\\
    &=-\bigl\langle\Xi(Y_1),[b,X_1-X'_1]\bigr\rangle-\bigl\langle\Xi(b),[X_1-X'_1,Y_1]\bigr\rangle\\
    &=0
  \end{align*}
since $\Xi(b)=0$ and $[b,X_1-X'_1,]=0$. We have shown that 
$\bigl\langle\Xi(X_1),Y\bigr\rangle=\bigl\langle\Xi(X'_1),Y\bigr\rangle$. Therefore
$\Gamma_b$ is a bilinear form on $F_b$. Similarly
  \begin{align*}
   \bigl\langle\Xi(X_1),Y\bigr\rangle
    &=\bigl\langle\Xi(X_1),[Y_1,b]\bigr\rangle
     =-\bigl\langle\Xi(Y_1),[b,X_1]\bigr\rangle-\bigl\langle\Xi(b),[X_1,Y_1]\bigr\rangle
     =\bigl\langle\Xi(Y_1),X\bigr\rangle\,,
  \end{align*}
which proves that $\Gamma_b$ is symmetric. 
\end{proof}

\begin{theo}\label{SymmetricBilinearFormOnGammab}
The assumptions and notations are the same as those in Proposition 
\ref{PropertiesOfGeneralizedThermodynamicFunctions}. For each $b\in\Omega$, there exists on the
vector subspace $F_b=[{\mathcal G},b]$ of elements $X\in{\mathcal G}$ which can be written 
$X=[X_1,b]$ for some $X_1\in{\mathcal G}$, a symmetric negative bilinear form $\Gamma_b$ given by
  $$\Gamma_b(X,Y)=\bigl\langle \Theta_b(X_1), Y\bigr\rangle\,,\quad\hbox{with}\ X_1\in{\mathcal G}\,,\ 
   X=[X_1,b]\in   F_b\,,\ Y\in F_b\,,
 $$ 
where $\Theta_b:{\mathcal G}\to{\mathcal G}^*$ is the symplectic $1$-cocycle defined in 
\ref{cocycleThetab}. 
\end{theo}

\begin{proof} We have seen in \ref{cocycleThetab} that $b\in\ker\Theta_b$. 
The fact that the equality given in the statement above defines indeed a symmetric 
bilinear form on $F_b$ directly follows from
Lemma \ref{EuclideanProduct}. We only have to prove that this symmetric bilinear 
form is negative. Let $X\in F_b$ and $X_1\in{\mathcal G}$ such that $X=[X_1,b]$. Using 
\ref{EffectInfinitesimalAdjointAction}
and \ref{CorollaryPropertiesOfGeneralizedThermodynamicFunctions}, we have 
 \begin{align*}\Gamma_b(X,X)
       &=\bigl\langle\Theta_b(X_1),[X_1,b]\bigr\rangle
        =\bigl\langle\Theta(X_1)-\ad^*_{X_1}E_J(b),[X_1,b]\bigr\rangle 
        =\bigl\langle DE_J(b)[X_1,b],[X_1,b]\bigr\rangle\\
       &\leq 0\,.
 \end{align*}
The symmetric bilinear form $\Gamma_b$ on $F_b$ is therefore negative.
\end{proof}

\begin{rmk} The symmetric negative bilinear forms encountered in 
\ref{SymmetricBilinearFormOnGammab} and 
\ref{CorollaryPropertiesOfGeneralizedThermodynamicFunctions} seem to be linked with
the Fisher metric in Information Geometry discussed in
\cite{Barbaresco2014, Barbaresco2015, Barbaresco2016}.
\end{rmk}

\subsection{Examples of generalized Gibbs states}\label{ExamplesGeneralizedGibbsStates}

\subsubsection{Action of the group of rotations on a sphere}\label{RotationSphere}
The symplectic manifold $(M,\omega)$ considered here is the two-dimensional sphere of radius $R$
centered at the origin $O$ of a three-dimensional oriented Euclidean vector space $\vect E$, equipped with
its area element as symplectic form. The group $G$ of rotations around the origin 
(isomorphic to $\SO(3)$) acts on the sphere $M$ by a Hamiltonian action. The Lie algebra $\mathcal G$ of
$G$ can be identified with $\vect E$, the fundamental vector field on $M$ associated to an element
$\vect b$ in ${\mathcal G}\equiv \vect E$ being the vector field on $M$ whose value at a point
$m\in M$ is given by the vector product $\vect b\times \vect{Om}$. The dual ${\mathcal G}^*$ of
$\mathcal G$ will be too identified with $\vect E$, the coupling by duality being given by the 
Euclidean scalar product. The momentum map $J:M\to{\mathcal G}^*\equiv \vect E$ is given by
 $$J(m)=-R\,\vect{Om}\,,\quad m\in M\,.
 $$
Therefore, for any $\vect b\in{\mathcal G}\equiv \vect E$,
 $$\bigl\langle J(m),\vect b\bigr\rangle=-R\,\vect{Om}\cdot\vect b\,.
 $$
Let $\vect b$ be any element in ${\mathcal G}\equiv \vect E$. To calculate the partition function
$P(\vect b)$ we choose an orthonormal basis $(\vect{e_x}, \vect{e_y}, \vect{e_z})$ of $\vect E$ such that
$\vect b=\Vert \vect b\Vert\vect{e_z}$, with $\Vert \vect b\Vert\in\RR^+$, and we use angular coordinates 
$(\varphi,\theta)$ on the sphere $M$. The coordinates of a point $m\in M$ are
 $$x=R\cos\theta\cos\varphi\,,\quad y=R\cos\theta\sin\varphi\,,\quad z=R\sin\theta\,.
 $$
We have
 $$P(\vect b)=\int_0^{2\pi}\left(\int_{-\pi/2}^{\pi/2}R^2\exp(R\Vert\vect b\Vert\sin\theta\,\d\theta\right)
               \d\varphi
               =\frac{4\pi R}{\Vert\vect b\Vert}\,\sh\bigl(R\Vert\vect b\Vert\bigr)\,.
 $$
The probability density (with respect to the natural area measure on the sphere $M$)
of the generalized Gibbs state associated to $\vect b$ is
 $$\rho_b(m)=\frac{1}{P(\vect b)}\,\exp(\vect{Om}\cdot\vect b)\,,\quad m\in M\,.$$
We observe that $\rho_b$ reaches its maximal value at the point $m\in M$ such that
$\displaystyle\vect{Om}=\frac{R\vect b}{\Vert\vect b\Vert}$ and its minimal value at the
diametrally opposed point.

\subsubsection{The Galilean group, its Lie algebra and its actions}
\label{GalileanGroup}

In view of the presentation, made below, of
some physically meaningful generalized Gibbs states for Hamiltonian actions
of subgroups of the Galilean group, we recall in this section some notions about
the space-time of classical (non-relativistic) Mechanics, the Galilean group, its Lie algebra
and its Hamiltonian actions. The interested reader will find a much more detailed treatment
on these subjects in the book by Souriau \cite{Souriau1969} or in the recent book 
by G.~de~Saxcé and C.~Vallée \cite{deSaxceVallee2016}. The paper \cite{deSaxceVallee2012}
presents a nice application of Galilean invariance in Thermodynamics.  
\par\smallskip

The space-time of classical Mechanics is a four-dimensional real affine space which, once an
inertial reference frame, units of length and time, orthonormal bases of space  and time
are chosen, can be identified with $\RR^4\equiv\RR^3\times \RR$ (coordinates $x$, $y$, $z$, $t$).
The first three coordinates $x$, $y$ and $z$ can be considered as the three components of a 
vector $\vect r\in \RR^3$, therefore an element of space-time can be denoted by $(\vect r,t)$. 
However, as the action of the Galilean group will show, the splitting of space-time into space and time is not uniquely determined, it depends on the choice of an inertial reference frame.
In classical Mechanics, there exists an absolute time, but no absolute space. There exists instead a space (which is an Euclidean affine three-dimensional space)
for each value of the time. The spaces for two distinct values of the time 
should be considered as disjoint.
\par\smallskip

The space-time being identified with $\RR^3\times\RR$ as explained above, the Galilean group $G$ 
can be identified with the set of matrices of the form
 $$\begin{pmatrix}
       A&\vect b&\vect d\\0&1&e\\0&0&1
          \end{pmatrix}\,,\quad\hbox{with}\ A\in \SO(3)\,,\ \vect b\ \hbox{and} \vect d\in\RR^3\,,
           \ e\in\RR\,,\eqno{(*)}
 $$ 
the vector space $\RR^3$ being oriented and endowed with its usual Euclidean structure,
the matrix $A\in\SO(3)$ acting on it.
\par\smallskip

The action of the Galilean group $G$ on space-time, identified as indicated above with 
$\RR^3\times \RR$, is the affine action 
 $$\begin{pmatrix}
    \vect r\\t\\1
     \end{pmatrix}
      \mapsto
       \begin{pmatrix}
       A&\vect b&\vect d\\0&1&e\\0&0&1
          \end{pmatrix}
   \begin{pmatrix}\vect r\\t\\1
   \end{pmatrix}
 =\begin{pmatrix}A\vect r+t\vect b+\vect d\\t+e\\1
  \end{pmatrix}\,.
 $$
The Lie algebra $\mathcal G$ of the Galilean group $G$ can be identified 
with the space of matrices of the form
 $$\begin{pmatrix}
       j(\vect \omega)&\vect \beta&\vect \delta\\0&0&\varepsilon\\0&0&0
          \end{pmatrix}\,,\quad\hbox{with}\ \vect\omega\,,\ \vect\beta\ \hbox{and}\ \vect\delta
           \in\RR^3\,,\ \varepsilon\in\RR\,.\eqno{(**)}
 $$
We have denoted by $j(\vect\omega)$ the $3\times 3$ skew-symmetric matrix
 $$j(\vect\omega)=\begin{pmatrix}0&-\omega_z&\omega_y\\
                                 \omega_z&0&-\omega_x\\
                                 -\omega_y&\omega_x&0
                  \end{pmatrix}\,.
 $$
The matrix $j(\vect \omega)$ is an element in the Lie algebra $\so(3)$, and its action on a vector 
$\vect r\in\RR^3$ is given by the vector product
 $$j(\vect\omega)\vect r=\vect\omega\times\vect r\,.
 $$    
Let us consider a mechanical system made by a point particle of mass $m$ whose position 
and velocity at time $t$, in the reference frame allowing the identification of space-time with
$\RR^3\times \RR$, are the vectors $\vect r$ and $\vect v\in \RR^3$. The action of an element
of the Galilean group on $\vect r,\vect v$ and $t$ can be written as
 $$\begin{pmatrix}\vect r&\vect v\\t&1\\1&0
    \end{pmatrix}
     \mapsto
       \begin{pmatrix}
       A&\vect b&\vect d\\0&1&e\\0&0&1
          \end{pmatrix}
\begin{pmatrix}\vect r&\vect v\\t&1\\1&0
\end{pmatrix}
=\begin{pmatrix}A\vect r+t\vect b+\vect d&A\vect v+\vect b\\t+e&1\\1&0
  \end{pmatrix}\,.
$$
Souriau has shown in his book \cite{Souriau1969} that this action is Hamiltonian, with the map
$J$, defined on the evolution space of the particle, with value in the dual ${\mathcal G}^*$
of the Lie algebra $\mathcal G$ of the Galilean group, as momentum map
 $$J(\vect r,t,\vect v,m)=m\left(\vect r\times\vect v,\ \vect r-t\vect v,\ \vect v,\ \frac{1}{2}\Vert
\vect v\Vert^2\right)\,.
 $$ 
Let
$\displaystyle b=\begin{pmatrix}
       j(\vect \omega)&\vect \beta&\vect \delta\\0&0&\varepsilon\\0&0&0
          \end{pmatrix}$ be an element in $\mathcal G$. Its coupling with
$J(\vect r,t,\vect v,m)\in{\mathcal G}^*$ is given by the formula
 $$\bigl\langle J(\vect r,t,\vect v,m),b\bigr\rangle
   =m\Bigl(\vect\omega\cdot(\vect r\times\vect v) - (\vect r-t\vect v)\cdot\vect\beta
      + \vect v\cdot \vect\delta - \frac{1}{2}\Vert\vect v\Vert^2\varepsilon\Bigr)\,. 
 $$ 

\subsubsection{One-parameter subgroups of the Galilean group}
\label{OneParameterSubgroups}
In his book \cite{Souriau1969}, J.-M. Souriau has shown that when the considered 
Lie group action is the action of the full Galilean group on the space of motions of
an isolated mechanical system, the open subset $\Omega$ of the Lie algebra $\mathcal G$
of the Galilean group on which the conditions specified in 
\ref{AssumptionsForGeneralizedThermodynamicFunctions} are satisfied is empty.
In other words, generalized Gibbs states of the full Galilean group do not exist.
However, generalized Gibbs states for one-parameter subgroups of the Galilean group 
do exist which have an interesting physical meaning.
\par\smallskip

Let us consider the element $b$ of $\mathcal G$ given by formula
$(*)$ of \ref{GalileanGroup}, and assume that $\varepsilon\neq 0$. 
The one-parameter subgroup $G_1$ of the Galilean group generated by $b$ is the set of matrices
$\exp(\tau b)$, with $\tau\in\RR$. We have
 $$\exp(\tau b)
   =\begin{pmatrix}
       A(\tau)&\vect b(\tau)&\vect d(\tau)\\
       0&1&\tau\varepsilon\\0&0&1
          \end{pmatrix}\,,
 $$
with
 \begin{align*}
  A(\tau)&=\exp\bigl(\tau j(\vect\omega)\bigr)\,,\\
  \vect b(\tau)&=\left(\sum_{n=1}^\infty\frac{\tau^n}{n!}
                 \bigl(j(\vect\omega)\bigr)^{n-1}\right)\vect\beta\,,\\
  \vect d(\tau)&=\left(\sum_{n=1}^\infty\frac{\tau^n}{n!}
                 \bigl(j(\vect\omega)\bigr)^{n-1}\right)\vect\delta
                +\varepsilon\left(\sum_{n=2}^\infty\frac{\tau^n}{n!}
                 \bigl(j(\vect\omega)\bigr)^{n-2}\right)\vect\beta\,, 
 \end{align*} 
with the usual convention that $\bigl(j(\vect\omega)\bigr)^0$ is the unit matrix.
\par\smallskip

The physical meaning of this one-parameter subgroup of the Galilean group can be understood
as follows.
Let us call \emph{fixed} the affine Euclidean reference frame of space
$(O,\vect{e_x},\vect{e_y},\vect{e_z})$ used to represent, at time $t=0$, a point 
in space by a vector $\vect r$ or by its three components $x$, $y$ and $z$. Let us set
$\displaystyle \tau=\frac{t}{\varepsilon}$. For each time $t\in \RR$, the action of
$\displaystyle A(\tau)=A\left(\frac{t}{\varepsilon}\right)$ maps the fixed reference frame 
$(O,\vect{e_x},\vect{e_y},\vect{e_z})$ onto another affine Euclidean reference frame
$\bigl(O(t),\vect{e_x}(t),\vect{e_y}(t),\vect{e_z}(t)\bigr)$, which we call the
\emph{moving} reference frame. The velocity and the acceleration of the relative  
motion of the moving reference frame
with respect to the fixed reference frame is given, at time $t=0$, by the fundamental 
vector field associated to the element $b$ of the Lie algebra ${\mathcal G}$ of 
the Galilean group: we see that each point in space has a motion composed of a 
rotation around the axis through $O$ parallel to $\vect\omega$, 
at an angular velocity $\displaystyle\frac{\Vert\vect\omega\Vert}{\varepsilon}$,
and simultaneously a uniformly accelerated motion of translation at an initial
velocity $\displaystyle\frac{\vect \delta}{\varepsilon}$ and acceleration
$\displaystyle\frac{\vect \beta}{\varepsilon}$. At time $t$, the velocity and acceleration
of the moving reference frame with respect to its instantaneous position
at that time can be described in a similar manner, but instead of $O$, $\vect\omega$,
$\vect \beta$ and $\vect\delta$ we must use the corresponding transformed elements
by the action of $\displaystyle A(\tau)=A\left(\frac{t}{\varepsilon}\right)$.

\subsubsection{A gas contained in a moving vessel}

We consider a mechanical system made by a gas of $N$ point particles, 
indexed by $i\in\{1,2,\ldots,N\}$, contained in a vessel with rigid, undeformable walls,
whose motion in space is given by the action of the one-parameter subgroup $G_1$ of the
Galilean group made by the
$\displaystyle A\left(\frac{t}{\varepsilon}\right)$, with $t\in\RR$, above described.
We denote by $m_i$, $\vect{r_i}(t)$ and $\vect{v_i}(t)$ the mass, position vector 
and velocity vector, respectively, of the
$i$-th particle at time $t$. 
Since the motion of the vessel containing the gas
is precisely given by the action of $G_1$, the boundary conditions imposed to the system
are invariant by that action, which leaves invariant the evolution 
space of the mechanical system, is Hamiltonian and projects onto a Hamiltonian 
action of $G_1$ on the symplectic manifold of motions of the system. We can therefore 
consider the generalized Gibbs states of the system, as discussed in 
\ref{GeneralizedGibbsStates}. 
We must evaluate the momentum map $J$ of that action and
its coupling with the element $b\in{\mathcal G}$. 
As in \ref{ClassicalIdealGas} we will neglect,
for that evaluation, the contributions of the collisions of the particles between 
themselves and with the walls of the vessel. The momentum map can therefore be evaluated as if all particles were free, and its coupling $\langle J,b\rangle$ with $b$ is the 
sum $\sum_{i=1}^N\langle J_i,b\rangle$ of the momentum map $J_i$ of the $i$-th particle, considered as free, with $b$. We have
 $$\bigl\langle J_i(\vect{r_i},t,\vect{v_i},m_i),b\bigr\rangle
   =m_i\Bigl(\vect\omega\cdot(\vect{r_i}\times\vect{v_i}) - (\vect{r_i}-t\vect{v_i})\cdot\vect\beta
      + \vect{v_i}\cdot \vect\delta - \frac{1}{2}\Vert\vect{v_i}\Vert^2\varepsilon\Bigr)\,. 
 $$ 
Following Souriau \cite{Souriau1969}, chapter IV, pages 299--303, we observe that
$\langle J_i,b\rangle$ is invariant by the action of $G_1$. We can therefore define
$\vect{r_{i\,0}}$, $t_0$ and $\vect{v_{i\,0}}$ by setting
 $$\begin{pmatrix}\vect{r_{i\,0}}&\vect{v_{i\,0}}\\t_0&1\\1&0
    \end{pmatrix}=
     \exp\left(-\frac{t}{\varepsilon}\,b\right)
 \begin{pmatrix}\vect{r_i}&\vect{v_i}\\t&1\\1&0
 \end{pmatrix}
 $$
and write 
 $$\bigl\langle J_i(\vect{r_i},t,\vect{v_i},m_i),b\bigr\rangle
   =\bigl\langle J_i(\vect{r_{i\,0}},t_0,\vect{v_{i\,0}},m_i),b\bigr\rangle\,.
 $$
The vectors $\vect{r_{i\,0}}$ and $\vect{v_{i\,0}}$ have a clear physical meaning: they are
the vectors $\vect{r_i}$ and $\vect{v_i}$ as seen by an observer moving with the
moving affine Euclidean reference frame $\bigl(O(t),\vect{e_x}(t),\vect{e_y}(t),
\vect{e_z}(t)\bigr)$. Moreover, as can be easily verified, $t_0=0$ of course. 
We therefore have  
 \begin{align*}
  \bigl\langle J_i(\vect{r_i},t,\vect{v_i},m_i),b\bigr\rangle
   &=m_i\Bigl(\vect\omega\cdot(\vect{r_{i\,0}}\times\vect{v_{i\,0}}) 
    - \vect{r_{i\,0}}\cdot\vect\beta
      + \vect{v_{i\,0}}\cdot \vect\delta 
     - \frac{1}{2}\Vert\vect{v_{i\,0}}\Vert^2\varepsilon\Bigr)\\
   &=m_i\Bigl(\vect{v_{i\,0}}\cdot(\vect\omega\times\vect{r_{i\,0}} +\vect\delta) 
    - \vect{r_{i\,0}}\cdot\vect\beta
     - \frac{1}{2}\Vert\vect{v_{i\,0}}\Vert^2\varepsilon\Bigr)
 \end{align*}  
where we have used the well known property of the mixed product
 $$\vect\omega\cdot(\vect{r_{i\,0}}\times\vect{v_{i\,0}})=
    \vect{v_{i\,0}}\cdot(\vect\omega\times\vect{r_{i\,0}})\,.
 $$
Let us set
 $$\vect U^*=\frac{1}{\varepsilon}(\vect\omega\times\vect{r_{i\,0}} +\vect\delta)\,.
 $$
Using $\vect{v_{i\,0}}-\vect U^*$ and $\vect U^*$ instead of $\vect{v_{i\,0}}$, 
we can write
 $$\bigl\langle J_i(\vect{r_i},t,\vect{v_i},m_i),b\bigr\rangle
   =m_i\varepsilon
        \left(-\frac{1}{2}\,\Vert\vect{v_{i\,0}}-\vect U^*\Vert^2
       - \vect{r_{i\,0}}\cdot\frac{\vect\beta}{\varepsilon}
        + \frac{1}{2}\,\Vert\vect U^*\Vert^2
           \right)\,.
 $$
We observe that the vector $\vect U^*$ 
only depends on $\varepsilon$, $\vect \omega$,
$\vect \delta$, which are constants once the element $b\in{\mathcal G}$ is chosen,
and of $\vect{r_{i\,0}}$,
not on $\vect{v_{i\,0}}$. It has a clear physical meaning: it is the value of the
velocity of the moving affine reference frame with respect to the fixed affine reference frame,
at point $\vect{r_{i\,0}}$ seen by an observer linked to the moving reference frame.
Therefore the vector $\vect{w_{i\,0}}=\vect{v_{i\,0}}-\vect U^*$ is the \emph{relative velocity}
of the $i$-th particle with respect to the moving affine reference frame, seen 
by an observer linked to the moving reference frame. 
\par\smallskip

The three components of $\vect{r_{i\,0}}$ and the three components of
$\vect{p_{i\,0}}=m_i\vect{w_{i\,0}}$ make a system of Darboux coordinates 
on the six-dimensional symplectic manilold $(M_i,\omega_i)$ of motions of 
the $i$-th particle. With a slight abuse of notations, we can consider the
momentum map $J_i$ as defined on the space of motions of the $i$-th particle, 
instead of being defined on the evolution space of this particle, and write
  $$\bigl\langle J_i(\vect{r_{i\,0}},\vect{p_{i,0}}),b\bigr\rangle
   = - \varepsilon
        \left(\frac{1}{2m_i}\,\Vert\vect{p_{i\,0}}\Vert^2 + m_if_i(\vect{r_{i\,0}})\right)\,,
         \ \vect{p_{i\,0}}=m_i\vect{w_{i\,0}}=m_i(\vect{v_{i\,0}}-\vect U^*)\,,
         \eqno{(*)}
 $$
and
 $$f_i(\vect{r_{i\,0}})=\vect{r_{i\,0}}\cdot\frac{\vect\beta}{\varepsilon}
        - \frac{1}{2\varepsilon^2}\,\Vert\vect \omega\times\vect{r_{i\,0}}\Vert^2
        - \frac{\vect\delta}{\varepsilon}\cdot\left(\frac{\vect\omega}{\varepsilon}
           \,\times\vect{r_{i\,0}}\right)
        -\frac{1}{2\varepsilon^2}\Vert\vect \delta\Vert^2
           \,.
 $$ 
The above equality $(*)$ is well suited for the determination of generalized Gibbs states of the system. Let us set
 $$P_i(b)=\int_{M_i}\exp\bigl(-\langle J_i,b\rangle\bigr)\d\lambda_{\omega_i}\,,\quad
   E_{J_i}(b)=\int_{M_i}J_i\exp\bigl(-\langle J_i,b\rangle\bigr)\d\lambda_{\omega_i}\,.
 $$
The integrals in the right hand sides of these equalities converge if and only if
$\varepsilon<0$. It means that the matrix $b$ belongs to the subset $\Omega$ of the
one-dimensional Lie algebra of the considered one-parameter subgroup $G_1$ of the Galilean group on which generalized Gibbs states can be defined if and only if
$\varepsilon<0$. Assuming that condition satisfied, we can use Definitions 
\ref{DefiGeneralizedGibbsState}. The generalized Gibbs state determined by $b$ has the smooth
density, with respect to the Liouville measure $\prod_{i=1}^N\lambda_{\omega_i}$ on the symplectic manifold of motions $\Pi_{i=1}^N(M_i,\omega_i)$,
 $$\rho(b)=\prod_{i=1}^N\rho_i(b)\,,\quad\hbox{with}\ 
    \rho_i(b)=\frac{1}{P_i(b)}\exp\bigl(-\langle J_i,b\rangle\bigr)\,.
 $$    
The partition function, whose expression is
 $$P(b)=\prod_{i=1}^N P_i(b)\,,
 $$
can be used, with the help of the formulae given in \ref{GeneralizedThermodynamicFunctions}, 
to determine all the generalized thermodynamic functions of the gas in a generalized 
thermodynamic equilibrium state.

\begin{rmks}\hfill

\par\noindent
{\rm 1.\quad} The physical meaning of the parameter $\varepsilon$ which appears in the expression
of the matrix $b$ is clearly apparent in the above expression $(*)$ of $\langle J_i,b\rangle$:
 $$\varepsilon=-\frac{1}{kT}\,,
 $$
$T$ being the absolute temperature and $k$ the Boltzmann's constant.

\par\smallskip\noindent
{\rm 2.\quad} The same expression $(*)$ above shows that the relative motion of the gas with respect to the moving vessel in which it is contained, seen by an observer linked to that moving vessel, is described by a Hamiltonian system 
in which the kinetic and potential energies of the $i$-th particle are, respectively, 
$\displaystyle\frac{1}{2m_i}\Vert\vect{p_{i\,0}}\Vert^2$ and $m_if_i(\vect{r_{i\,0}})$. This result can be obtained in another way: by deriving the Hamiltonian which governs the relative motion of a mechanical system with respect to a moving frame, as used by Jacobi \cite{Jacobi1836} to determine the famous Jacobi integral of the restricted circular three-body problem 
(in which two big planets move on concentric circular orbits around their common center of mass, and a third planet of negligible mass moves in the gravitational field created by the two big planets).

\par\smallskip\noindent
{\rm 3.\quad} The generalized Gibbs state of the system imposes to the various parts of the system, \emph{i.e.}, to the various particles, to be at the same temperature 
$\displaystyle T=-\frac{1}{k\varepsilon}$ and to be statistically at rest in the same moving reference frame.  
\end{rmks}

\subsubsection{Three examples}\hfill
\par\noindent
{\rm 1.\quad} Let us set $\vect\omega=0$ and $\vect\beta=0$. The motion of the moving vessel
containing the gas (with respect to the so called \emph{fixed reference frame}) is a translation at a constant velocity $\displaystyle\frac{\vect\delta}{\varepsilon}$. The function
$f_i(\vect{r_{i\,0}})$ is then a constant. In the moving reference frame, which is an inertial frame, we recover the thermodynamic equilibrium state of a monoatomic gas discussed in
\ref{ClassicalIdealGas}.

\par\smallskip\noindent
{\rm 2.\quad} Let us set now $\vect\omega=0$ and $\vect\delta=0$. The motion of the moving vessel containing the gas (with respect to the so called \emph{fixed reference frame}) is now an uniformly accelerated translation, with acceleration 
$\displaystyle\frac{\vect\beta}{\varepsilon}$.
The function $f_i(\vect{r_{i\,0}})$ now is 
 $$f_i(\vect{r_{i\,0}})=\vect{r_{i\,0}}\cdot\frac{\vect\beta}{\varepsilon}\,.
 $$
In the moving reference frame, which is no more inertial, we recover the thermodynamic 
equilibrium state of a monoatomic gas in a gravity field 
$\displaystyle \vect g=-\frac{\vect\beta}{\varepsilon}$ discussed in
\ref{ClassicalGasWithGravity}.

\par\smallskip\noindent
{\rm 3.\quad} Let us now set $\vect\omega=\omega\vect{e_z}$, $\vect\beta=0$ and $\vect\delta=0$.
The motion of the moving vessel containing the gas (with respect to the so called 
\emph{fixed reference frame}) is now a rotation around the coordinate $z$ axis at a constant
angular velocity $\displaystyle\frac{\omega}{\varepsilon}$. The function $f_i(\vect{r_{i\,0}})$
is now
 $$f_i(\vect{r_{i\,0}})=- \frac{\omega^2}{2\varepsilon^2}\,
                           \Vert\vect{e_z}\times\vect{r_{i\,0}}\Vert^2\,.
 $$
The length $\Delta=\Vert\vect{e_z}\times\vect{r_{i,0}}\Vert$ is the distance between the 
$i$-th particle and the axis of rotation of the moving frame (the coordinate $z$ axis). 
Moreover, we have seen that
$\displaystyle \varepsilon=\frac{-1}{kT}$. Therefore in 
the generalized Gibbs state, the probability density $\rho_i(b)$ of presence 
of the $i$-th particle in its symplectic manifold of motion $M_i,\omega_i$, with respect to the Liouville measure
$\lambda_{\omega_i}$, is 
 $$\rho_i(b)=\frac{1}{P_i(b)}\,\exp\bigl(-\langle J_i,b\rangle\bigr)
            =\hbox{Constant}\cdot\exp\left(-\frac{1}{2m_ikT}\,\Vert\vect{p_{i\,0}}\Vert^2
              +\frac{m_i}{2kT}\left(\frac{\omega}{\varepsilon}\right)^2\Delta^2\right)\,.
 $$
This formula describes the behaviour of a gas made of point particles of various masses in a centrifuge rotating at a constant angular velocity $\displaystyle\frac{\omega}{\varepsilon}$:
the heavier particles concentrate farther from the rotation axis than the lighter ones.

\subsubsection{Other applications of generalized Gibbs states}
Applications of generalized Gibbs states in Thermodynamics of Continua, with the use of affine tensors, are presented in the papers by G.~de~Saxcé \cite{deSaxce2015, deSaxce2016}.
\par\smallskip
 
Several applications of generalized Gibbs states of subgroups of the Poincaré group
were considered by J.-M. Souriau. For example, he presents  in
his book \cite{Souriau1969}, chapter IV, page 308, a generalized Gibbs 
which describes the behaviour of a gas in a 
relativistic centrifuge, and 
in his papers \cite{Souriau1974, Souriau1975}, very nice applications of such generalized
Gibbs states in Cosmology.
 
\section{Acknowledgements}
I address my thanks to Alain Chenciner for his interest and his help to study
the works of Claude Shannon, to Roger Balian for his comments and his explanations about thermodynamic potentials, and to Frédéric Barbaresco for his kind invitation to participate
in the GSI 2015 conference and his encouragements.


\begin{thebibliography}{999}

\bibitem{AbrahamMarsden} Abraham, R., and Marsden, J.~E., 
\emph{Foundations of Mechanics, 2nd edn.}, Addison-Wesley, Reading (1978).

\bibitem{Arnold} Arnold, V.I., \emph{Mathematical methods of Classical Mechanics, 2nd edn.},
 Springer, New York (1978).

\bibitem{Balian2005} Balian,~R., \emph{Information in statistical physics}, Studies in History and Philosophy of Modern Physics, part B, February 2005.

\bibitem{Balian2015} Balian,~R., \emph{François Massieu et les potentiels thermodynamiques}, \'Evolution des disciplines et histoire des découvertes, Académie des Sciences, Avril 2015.

\bibitem{Barbaresco2014} Barbaresco,~F., \emph{Koszul Information Geometry and Souriau 
Geometric Temperature/Capacity of Lie Group Thermodynamics}. Entropy, vol. 16, 2014, 
pp.~4521-4565. Published in the book \emph{Information, Entropy and Their Geometric Structures}, MDPI Publisher, September 2015.

\bibitem{Barbaresco2015} Barbaresco,~F., \emph{Symplectic Structure of Information Geometry:
Fisher Metric and Euler-Poincaré Equation of Souriau Lie Group Thermodynamics}. In
\emph{Geometric Science of Information, Second International Conference GSI 2015 Proceedings}, (Franck Nielsen and Frédéric Barbaresco, editors), Lecture Notes in Computer Science vol.~9389, Springer 2015, pp. 529--540.

\bibitem{Barbaresco2016} Barbaresco,~F., \emph{Geometric Theory of Heat from Souriau Lie Groups Thermodynamics and Koszul Hessian Geometry: Applications in Information Geometry for Exponential Families}. To appear in the Special Issue \lq\lq Differential Geometrical Theory of Statistics\rq\rq, MDPI, Entropy, 2016. 

\bibitem{berest} B\'erest,~P., \emph{Calcul des variations}. 
Les cours de l'\'Ecole Polytechnique, Ellipses/éditions marketing, Paris 1997.

\bibitem{Boltzmann} Boltzmann,~L.~E., \emph{Leçons sur la théorie des gaz}, 
Gauthier-Villars, Paris, 1902--1905. Reprinted by \'Editions Jacques Gabay, Paris, 1987.
The second part can be downloaded at \url{http://iris.univ-lille1.fr/handle/1908/1523}
 
\bibitem{bourguignon} Bourguignon,~J.-P., \emph{Calcul variationnel}. \'Editions de
l'\'Ecole Polytechnique, 1991. Reprinted in 2016.

\bibitem{Callen} Callen,~H.~B., \emph{Thermodynamics and an Introduction to Thermostatics},
second edition, John Wiley and Sons, New York, 1985. 

\bibitem{cannasdasilva} Cannas da Silva,~A., \emph{Lectures on symplectic geometry}, 
Lecture Notes in Mathematics n.~1764, Springer, 2001, corrected printing 2008. 

\bibitem{HyperPhysics} HyperPhysics, \emph{Kinetic Theory}, Georgia State University, 2016,
\url{http://hyperphysics.phy-astr.gsu.edu/hbase/kinetic/kinthe.html}

\bibitem{Gastebois} Gastebois, G., \emph{Théorie cinétique des gaz}, 2016,
\url{http://gilbert.gastebois.pagesperso-orange.fr/java/gaz/gazparfait/theorie_gaz.pdf}

\bibitem{GuilleminSternberg} Guillemin,~V. and Sternberg,~S., 
\emph{Symplectic Techniques in Physics}, Cambridge University Press, Cambridge (1984).

\bibitem{hamilton8}
Hamilton,~W.~R., \emph{On a general method in
Dynamics}. Read April 10, 1834, Philosophical Transactions of the Royal
Society, part II for 1834, pp.~247--308. In \emph Sir William Rowan Hamilton
mathematical Works, vol. II, Cambridge University Press, London, 1940. 

\bibitem{hamilton9}
Hamilton,~W.~R., \emph{Second essay on a general method in
Dynamics}. Read January 15, 1835, Philosophical Transactions of the Royal
Society, part I for 1835, pp.~95--144. In \emph Sir William Rowan Hamilton
mathematical Works, vol. II, Cambridge University Press, London, 1940.

\bibitem{holm3} Holm,~D., \emph{Geometric Mechanics}, Part I: \emph{Dynamics and Symmetry} (354 pages), 
Part II: \emph{Rotating, Translating and Rolling} (294 pages). World Scientific, London, 2008.

\bibitem{Iglesias2000} Iglesias,~P., \emph{Symétries et moment}, 
Hermann, Paris, 2000.

\bibitem{Jacobi1836} Jacobi,~C.~G.~J., \emph{Sur le mouvement d'un point et sur un cas particulier du problème des trois corps}, C.R.Acad.Sc.Paris, 3 (1836) pp.~59--61.  

\bibitem{Jaynes1957a} Jaynes,~E.~T., \emph{Information Theory and Statistical Mechanics}, Phys. Rev. vol. 106, n. 4 (1957), pp.~620--630.

\bibitem{Jaynes1957b} Jaynes,~E.~T., \emph{Information Theory and Statistical Mechanics II}, 
Phys. Rev. vol. 108, n. 2 (1957), pp.~171--190.

\bibitem{Kirillov76} Kirillov,~A., {\it Local Lie algebras\/},
Russian Math. Surveys 31 (1976), pp.~55--75.

\bibitem{Kosmann2011} Kosmann-Schwarzbach,~Y., 
\emph{The Noether theorems}, Springer, 2011.

\bibitem{Koszul85} Koszul,~J.-L., {\it Crochet de Schouten-Nijenhuis et
cohomologie\/}, in {\it {\'E}. Cartan et les math{\'e}\-ma\-tiques
d'aujourd'hui\/} Ast{\'e}risque, num{\'e}ro hors s{\'er}ie, 1985,
257--271.

\bibitem{Lagrange3}
J.-L.~Lagrange, \emph{Mémoire sur la théorie générale de la
variation des constantes arbitraires dans tous les problèmes de mécanique}.
Lu le 13 mars 1809  à l'Institut de France.
Dans \emph{\OE uvres de Lagrange}, volume VI,
Gauthier-Villars, Paris,
1877, pages 771--805.

\bibitem{Lagrange4}
J.-L.~Lagrange, \emph{Second mémoire sur la théorie de
la variation des constantes arbitraires dans les problèmes de mécanique}.
Mémoire lu le 19 février  1810 à l'Institut de France.
Dans \emph{\OE uvres de Lagrange}, volume VI, Gauthier-Villars, Paris,
1877, pages 809--816.

\bibitem{Lagrange5}
Lagrange,~J.~L., \emph{Mécanique analytique}. Première
édition chez la veuve Desaint, Paris 1808. Réimprimé par
Jacques Gabay, Paris, 1989. Deuxième édition par Mme veuve
Courcier, Paris, 1811. Réimprimé par Albert Blanchard,
Paris. Quatrième édition (la plus complète) en deux volumes,
avec des notes par M.~Poinsot, M.~Lejeune-Dirichlet,
J.~Bertrand, G.~Darboux, M.~Puiseux, J.~A.~Serret,
O.~Bonnet, A.~Bravais, dans \emph{\OE uvres de Lagrange},
volumes XI et XII, Gauthier-Villars, Paris, 1888.

\bibitem{lanczos} Lanczos,~C.~S., \emph{The variational principles of Mechanics}, 4-th edition.
University of Toronto Press, Toronto, 1970. Reprinted by Dover, New York, 1970.

\bibitem{LaurentPichereauVanhaecke} Laurent-Gengoux,~C., Pichereau,~A., and Vanhaecke,~P.,
   \emph{Poisson structures}, Springer, Berlin (2013).

\bibitem{LibermannMarle1987} Libermann,~P., and Marle,~C.-M., \emph{Symplectic Geometry and Analytical Mechanics},
  D. Reidel Publishing Company, Dordrecht (1987).

\bibitem{Lichnerowicz77} Lichnerowicz,~A., \emph{Les variétés de Poisson et leurs algèbres de Lie associées},
Journal of Differential Geometry 12 (1977), pp.~253--300.

\bibitem{Lichnerowicz79} Lichnerowicz,~A., \emph{Les variétés de Jacobi et leurs algèbres de Lie associées},
 \emph{Journal de Mathématiques Pures et Appliquées} 57 (1979), pp.~453--488.

\bibitem{Mackey1963} Mackey,~G.~W., {\it The Mathematical Foundations of Quantum
Mechanics\/}, W. A. Benjamin, Inc., New York, 1963.

\bibitem{Malliavin} Malliavin,~P., \emph{G{\'e}om{\'e}trie diff{\'e}rentielle intrins{\`e}que}, Hermann, Paris 1972.

\bibitem{Marle2008} Marle,~C.-M., \emph{Calculus on Lie algebroids,
Lie groupoids and Poisson manifolds}, Dissertationes Mathematicae {\bf 457}, Warszawa (2008) 1--57.

\bibitem{marlejgsp} Marle,~C.-M., \emph{On Henri Poincaré's note 
\lq\lq Sur une forme nouvelle des équations de la Mécanique\rq\rq},
Journal of Geometry and Symmetry in Physics, vol. 29, 2013,
pp.~1--38.

\bibitem{Marle2014} Marle,~C.-M., 
\emph{Symmetries of Hamiltonian Systems on Symplectic and Poisson manifolds},
in \emph{Similarity and Symmetry Methods, Applications in Elasticity and Mechanics of Materials}, 
Lecture Notes in Applied and Computational Mechanics 
(J.-F. Ganghoffer and I. Mladenov, editors), Springer, 2014, pp.~183--269. 

\bibitem{MarsdenWeinstein74} Marsden,~J.E., and Weinstein,~A., 
   \emph{Reduction of symplectic manifolds with symmetry},
   Reports on Mathematical Physics 5, 1974, pp.~121--130.

\bibitem{Massieu1869a} Massieu,~F., \emph{Sur les Fonctions caractéristiques des divers fluides}. C.~R.~Acad.~Sci.~Paris vol. 69, 1869, pp.~858--862.

\bibitem{Massieu1869b} Massieu,~F., \emph{Addition au précédent Mémoire sur les Fonctions caractéristiques}. C.~R.~Acad.~Sci.~Paris vol. 69, 1869, pp.~1057--1061.

\bibitem{Massieu1876} Massieu,~F., \emph{Thermodynamique. Mémoire sur les Fonctions Caractéristiques des Divers Fluides et sur la Théorie des Vapeurs}. Mémoires présentés par divers savants à l'Académie des Sciences de l'Institut National de France, XXII, n.~2, 1876,
pp.~1--92.

\bibitem{Meyer73} Meyer,~K., Symmetries and integrals in mechanics. In \emph{Dynamical systems} 
(M. Peixoto, ed.), Academic Press (1973) pp.~259--273.

\bibitem{Newton1687} Newton,~I., \emph{Philosophia Naturalis Principia Mathematica}, London, 1687. Translated in French by \'Emilie du Chastelet (1756).

\bibitem{OrtegaRatiu2004} Ortega,~J.-P., and Ratiu,~T.-S., 
\emph{Momentum maps and Hamiltonian reduction},
Birkhäuser, Boston, Basel, Berlin, 2004.

\bibitem{Poincare1890} Poincar\'e,~H., \emph{Sur le problème des trois corps et les équations de la dynamique},  Acta Mathematica, vol.~13, 1890, pp.~1--270.

\bibitem{Poincare01} Poincar\'e,~H., \emph{Sur une forme nouvelle des \'equations de la M\'eanique}, C.~R.~Acad.~Sci.~Paris, T.~CXXXII, n. 7 (1901), pp.~369--371.

\bibitem{Poisson2}
S.~D.~Poisson, \emph{Sur la
variation des constantes
arbitraires dans les questions de mécanique}. Mémoire lu le 16 octobre 1809 à
l'Institut de France. Journal de l'école Polytechnique,
quinzième cahier, tome VIII, pages 266--344.

\bibitem{deSaxce2015} de Saxcé,~G., \emph{Entropy and Structure for the Thermodynamic Systems}.
In \emph{Geometric Science of Information, Second International Conference GSI 2015 Proceedings}, (Franck Nielsen and Frédéric Barbaresco, editors), Lecture Notes in Computer Science vol.~9389,
Springer 2015, pp. 519--528. 

\bibitem{deSaxce2016} de Saxcé,~G., \emph{Link Between Lie Group Statistical Mechanics and Thermodynamics of  Continua}. To appear in the Special Issue \lq\lq 
Differential Geometrical Theory of Statistics\rq\rq, MDPI, Entropy, 2016.

\bibitem{deSaxceVallee2012} de Saxcé,~G., and Vallée,~C., \emph{Bargmann group, 
momentum tensor and Galilean invariance of Clausius-Duhem inequality}. 
International Journal of Engineering Science, Vol. 50, 1, January 2012, pp.~216--232.

\bibitem{deSaxceVallee2016} de Saxcé,~G., and Vallée,~C., \emph{Galilean Mechanics and Thermodynamics of Continua}, John Wiley and Sons, Hoboken, USA, 2016.

\bibitem{Shannon1948} Shannon,~C.~E., \emph{A Mathematical Theory of Communication}, 
The Bell System Technical Journal, vol. 27, pp.~379--423 and 623--656, July and October 1948. 

\bibitem{Souriau1966} Souriau,~J.-M., \emph{D\'efinition covariante des \'equilibres thermodynamiques}, Supplemento al Nuovo cimento vol. IV n.1, 1966, pp.~203--216. 

\bibitem{Souriau1969}
Souriau,~J.-M., \emph{Structure des syst\`emes dynamiques}, Dunod, Paris, 1969.

\bibitem{Souriau1974} Souriau,~J.-M., \emph{M\'ecanique statistique, groupes de Lie et cosmologie},
Colloques internationaux du CNRS numéro 237\emph{Géométrie symplectique et physique mathématique}, 1974, pp.~59--113.

\bibitem{Souriau1975} Souriau,~J.-M., \emph{Géométrie symplectique et Physique mathématique}, Deux conférences de Jean-Marie Souriau, Colloquium de la Société Mathématique de France, 19 février et 12 novembre 1975.

\bibitem{Souriau1984} Souriau,~J.-M., \emph{Mécanique classique et Géométrie symplectique}, preprint, Université de Provence et Centre de Physique Théorique, 1984.

\bibitem{Sternberg1964} Sternberg,~S., \emph{Lectures on differential geometry}, Prentice-Hall, Englewood Cliffs, 1964. 

\bibitem{Synge1957} Synge,~J.L., \emph{The Relativistic Gas}, North Holland Publishing Company, 
Amsterdam, 1957.

\bibitem{Tulczyjew1974} Tulczyjew,~W.~M., 
\emph{Hamiltonian systems, Lagrangian systems and the Legendre transformation}, 
Symposia Mathematica, 14 (1974), pp.~247--258.

\bibitem{Tulczyjew1989} Tulczyjev,~W.~M., 
\emph{Geometric Formulations of Physical Theories}, 
Monographs and Textbooks in Physical Science, Bibliopolis, Napoli, 1989.

\bibitem{Vaisman94} Vaisman,~I., \emph{Lectures on the Geometry of Poisson
manifolds}, Birkh{\"a}user, Basel, Boston, Berlin, 1994.

\bibitem{Weinstein83} Weinstein,~A., \emph{The local structure of Poisson manifolds},
\emph{J. Differential Geometry} {\bf 18} (1983), pp.~523--557 and {\bf 22} 
(1985), p. 255.

\end{thebibliography}
\end{document}